\documentclass{article}
\usepackage[%
journal=AIPHC,    
lang=british,   
]{ems-journal}

\usepackage[abbrev,lite,nobysame]{amsrefs}
\usepackage{mathtools}
\usepackage{enumitem}

\usepackage[capitalise]{cleveref}

\newtheorem{theorem}{Theorem}
\newtheorem{proposition}{Proposition}[section]
\newtheorem{lemma}[proposition]{Lemma}

\theoremstyle{definition}

\newtheorem{remark}[proposition]{Remark}

\numberwithin{equation}{section}

\newcommand{\CC}{\mathbb{C}}

\newcommand{\ZZ}{\mathbb{Z}}
\newcommand\TT {{\mathbb T}}
\newcommand\T {{\mathbb T}}
\newcommand\RR {{\mathbb R}}
\newcommand\R {{\mathbb R}}
\newcommand{\mH}{\mathcal{H}}
\newcommand{\mK}{\mathcal{K}}

\newcommand{\bbS}{\mathbb{S}}

\newcommand\eps{\varepsilon}
\newcommand\na{\nabla}
\newcommand\pa{\partial}
\newcommand\e{{\mathrm{e}}}
\newcommand\dd{{\mathrm{d}}}
\newcommand\ii{{\mathrm{i}}}
\newcommand\ddt{{\frac{\dd}{\dd t}}}
\def\l {\langle}
\def\r {\rangle}
\def\Re{\operatorname{Re}}

\newcommand{\init}{\mathrm{in}} 

\newcommand{\tnorm}[1]{{\left\vert\kern-0.25ex\left\vert\kern-0.25ex\left\vert #1
    \right\vert\kern-0.25ex\right\vert\kern-0.25ex\right\vert}}

\newcommand\sfV{{\mathsf V}}
\newcommand\bv{{v}}
\newcommand\bV{{V}}

\numberwithin{equation}{section}

\begin{document}

\title{Nonlinear stability for active suspensions}



\emsauthor{1}{
	\givenname{Michele}
	\surname{Coti Zelati}
	\mrid{879261}
	\orcid{0000-0002-2495-2212}}{M.~Coti Zelati}
\emsauthor{2}{
	\givenname{Helge}
	\surname{Dietert}
	\mrid{1101840}
	\orcid{0000-0002-2071-9202}}{H.~Dietert}
\emsauthor{3}{
	\givenname{David}
	\surname{Gérard-Varet}
	\mrid{697517}
	\orcid{0000-0001-6507-6653}}{D.~Gérard-Varet}

\Emsaffil{1}{
	\department{Department of Mathematics}
	\organisation{Imperial College London}
	\rorid{041kmwe10}
	\zip{SW7 2AZ}
	\city{London}
	\country{UK}
	\affemail{m.coti-zelati@imperial.ac.uk}}
\Emsaffil{2}{
    \department{Institut de Math\'ematiques de Jussieu-Paris Rive Gauche (IMJ-PRG)}
	\organisation{Universit\'e Paris Cité and Sorbonne Universit\'e, CNRS}
	\rorid{03fk87k11}
	\zip{75013}
	\city{Paris}
	\country{France}
	\affemail{helge.dietert@imj-prg.f}}
\Emsaffil{3}{
    \department{Institut de Math\'ematiques de Jussieu-Paris Rive Gauche (IMJ-PRG)}
	\organisation{Universit\'e Paris Cité and Sorbonne Universit\'e, CNRS}
	\rorid{03fk87k11}
	\zip{75013}
	\city{Paris}
	\country{France}
	\affemail{david.gerard-varet@u-paris.fr}}



\begin{abstract}
  This paper is devoted to the nonlinear analysis of a kinetic model introduced
  by Saintillan and Shelley for rodlike particles in viscous flows. We investigate
  the stability of the constant state $\Psi(t,x,p) = \frac{1}{4\pi} $ corresponding
  to a distribution of particles that is homogeneous in space (variable
  $x \in \mathbb{T}^3$) and uniform in orientation (variable $p \in \mathbb{S}^2$). We prove
  its nonlinear stability under the optimal condition of linearized spectral
  stability. The main achievement in this work is that the smallness condition
  on the initial perturbation is independent of the translational diffusion and
  only depends on the rotational diffusion, which is particularly relevant for
  dilute suspensions. Upgrading our previous linear study \cite{CZDGV22} to such
  nonlinear stability result requires new mathematical ideas, due to the
  presence of a quasilinear term in $x$ associated with nonlinear
  convection. This term cannot be treated as a source, because it is not
  controllable by the rotational diffusion in $p$. Also, it prevents the
  decoupling of $x$-Fourier modes crucially used in \cite{CZDGV22}.  A key
  feature of our work is an analysis of enhanced dissipation and mixing
  properties of the advection diffusion operator
  \[\partial_t + (p + u(t,x)) \cdot \nabla_x - \nu \Delta_p \]
  on $\mathbb{T}^3 \times \mathbb{S}^2$ for a given appropriately small vector field $u$. We
  hope this linear analysis to be of independent interest, and useful in other
  contexts with partial or anisotropic diffusions.
\end{abstract}

\maketitle

\section{The model and main results}
Microswimmers are organisms (\textit{e.g.}\ bacteria) or objects
operating in a fluid at the microscale. They exhibit inherent
self-propulsion driven by mechanisms such as flagellar motion, cilia
beating, or synthetic fluid flow inducers. Beyond understanding the
individual behaviour of microswimmers \cite{Lauga}, special attention
was paid recently to their collective motion, with impact on fluid
mixing, rheological properties or self-organization of biological
active matter
\cites{Aranson,Rafai,Berlyand,saintillan-2018-rheology,Degond}.

We focus in this paper on one popular model, due to Saintillan and Shelley \cite{saintillan-shelley-2008-instabilities}, describing dilute suspensions of self-propelled rodlike particles. This model, relevant to the  dynamical behavior of bacteria, is a coupled fluid-kinetic model, detailing the
interaction of the particles  with the surrounding fluid medium and encapsulating the propulsion mechanisms governing their motion. It reads
\begin{equation} \label{SS}
  \begin{aligned}
    &\partial_{t} \Psi
      +(U_0 \, p +u)\cdot \nabla_{x} \Psi
      + \nabla_p\cdot \Big(\mathbb{P}_{p^\perp}
      \left[(\gamma E(u) + W(u))p\right] \Psi\Big)
      = \nu \Delta_p \Psi + \kappa \Delta_x \Psi ,\\
    &-\Delta_x u +\nabla_x q= \iota \nabla_{x} \cdot \int_{\bbS^2} \Psi(t,x,p)\, p\otimes p \, \dd p,\\
    &\nabla_{x} \cdot u=0,
  \end{aligned}
\end{equation}
where
\begin{align}
  E(u)=\frac12 \left[\nabla_{x} u + (\nabla_{x} u)^T\right]
  \qquad\text{and}\qquad
  W(u)=\frac12 \left[\nabla_{x} u - (\nabla_{x} u)^T\right]
\end{align}
are the symmetric and skew-symmetric parts of $\nabla_x u$, respectively. The unknowns of the model are $\Psi = \Psi(t,x,p)$, the distribution of rodlike particles in space and orientation, and  $u = u(t,x)$, $q= q(t,x)$ the fluid velocity and pressure. The space variable $x$ is taken in a periodic box $\TT_L := (\R/L\ZZ)^3$ of size $L>0$, while the orientation of the particles is given by $p \in \bbS^2$. The first equation in \eqref{SS} describes the evolution of $\Psi$ under three effects:
\begin{itemize}
\item transport by the fluid velocity $u(t,x)$ and by the
  self-propulsion $U_0 \, p$ with velocity \(U_0 > 0\).
    \item rotation by angular velocity
    $\mathbb{P}_{p^\perp}  \left[(\gamma E(u) + W(u))p\right]$, where $\mathbb{P}_{p^\perp}$ denotes projection tangentially to the sphere. This  expression for the angular velocity is due to Jeffery \cite{jeffery-1922-ellipsoidal} and corresponds to the angular velocity of a slender particle in a Stokes flow: interactions are neglected as the suspension is assumed to be dilute. The constant $\gamma\in [-1,1]$ is related to the geometric properties of the particle.
    \item rotational and translational diffusion, with respective coefficients $\nu>0$ and $\kappa > 0$.
\end{itemize}
Finally, the last two equations in \eqref{SS} are    Stokes equations for the fluid flow. They incorporate an additional stress $\Sigma$, where
\[\Sigma = \iota \nabla_{x} \cdot \int_{\bbS^2} \Psi(t,x,p)\, p\otimes p \, \dd p \]
reflects the constraint exerted by the particles on the flow. It is obtained by a continuous approximation of the sum of all single particle contributions, modeled as dipoles of opposite forces along $p$. Parameter $\iota\neq 0$ distinguishes between two types of swimmers: pullers (resp. pushers) correspond to $\iota>0$ (resp. $\iota<0$).
We refer to \cite{saintillan-shelley-2008-instabilities} or to the introduction in \cite{CZDGV22} for more on the derivation of the model.

Of special interest is the stability of the constant state equilibrium
$\Psi^{\mathrm{iso}}=1/4\pi$, which corresponds to a distribution of particles
homogeneous in space and uniform in orientation. In particular, physicists are
interested in loss of stability, with possible emergence of collective patterns
and rheological changes. It is therefore natural to work with
\[  \psi = \Psi - \Psi^{\mathrm{iso}} = \Psi - \frac{1}{4\pi}.\]
Moreover, one can put the system in dimensionless form, introducing
\[ t := \frac{U_0 t}{L}, \quad x := \frac{x}{L}, \quad  u := \frac{u}{U_0}, \quad q := \frac{qL}{U_0}, \quad \nu :=   \frac{\nu L}{U_0}, \quad \kappa :=  \frac{\kappa}{U_0 L},  \quad  \iota :=  \frac{\iota L}{U_0}. \]
At this stage, it is worth noticing that the dimensionless  \emph{ratio} $\frac{\kappa}{\nu}$ is expected to be very small. Indeed, we remind that in the simplified setting of a single spherical passive particle of radius $r$, at temperature $T$, inside a Stokes flow of viscosity $\mu$,  Einstein-Stokes laws
\[D_t = \frac{k_B T}{6\pi\eta r},  \quad D_r = \frac{k_B T}{8\pi\eta r^3}\]
yield $\frac{\kappa}{\nu} = \frac{4}{3} \left(\frac{r}{L}\right)^2 \ll 1$. It is likely
that such smallness persists in the case of a dilute suspension of ellipsoidal
particles. Moreover, in the case of self-propelled bacteria, it is acknowledged
that propulsion strongly dominates over translational diffusion, so that the
latter is often neglected, see \cites{tenHagen2011,Rosser2014}. Therefore, it is important to obtain
stability results that are uniform in small $\kappa$. This is the target
of this work and for simplicity we take \(\kappa=0\) from now on. This is no loss of generality, see the discussion after Theorem~\ref{thm:nonstab}. We
end up with the following dimensionless system, for $x \in \T$, $p \in \bbS^2$:
\begin{subequations}\label{SS3}
\begin{alignat}{3}
&\partial_t \psi  +   (p+u) \cdot \nabla_x\psi  - \frac{3\gamma}{4\pi} (p \otimes p) : E(u)
      + \nabla_p\cdot \Big(\mathbb{P}_{p^\perp}   \left[(\gamma E(u) + W(u))p\right] \psi\Big)
      = \nu \Delta_p \psi, \label{SS3eq1}\\
    &-\Delta_x u +\nabla_x q= \iota \nabla_x \cdot \int_{\bbS^2} \psi(t,x,p)\, p\otimes p \, \dd p,\\
    &\nabla_x \cdot u=0.
    \end{alignat}
\end{subequations}

Since the seminal paper \cite{saintillan-shelley-2008-instabilities}, system \eqref{SS3} has been the matter of several numerical and theoretical studies. In particular,  simulations align well with experimental observations of suspensions of bacteria.
\begin{itemize}
\item For $\iota > 0$ (pullers), no coherent behaviour is observed, which can be interpreted as stability of $\Psi^{\mathrm{iso}}$.
\item For $\iota < 0$ (pushers), one observes formation of patterns when hydrodynamic interactions are increased.
\end{itemize}
Such observations are valid for small rotational diffusion. They were
confirmed analytically in
\cite{saintillan-shelley-2008-instabilities}, \emph{in the special
  case $\nu = 0$}, for the linearization of \eqref{SS3} around
$\psi = 0$ (that is the linearization of \eqref{SS} around
$\Psi = \Psi^{\mathrm{iso}}$). Spectral stability was studied, through a mode
by mode Fourier analysis in variable $x$ and refined in subsequent
works \cites{HoheneggerShelley2010,ohm-shelley-2022-weakly}.
Specifically, all perturbations located at mode $k \in 2\pi \ZZ^3_*$
decay if and only if
\begin{equation*}
 \frac{\gamma |\iota|}{|k|} < \Gamma_c
\end{equation*}
where the threshold $\Gamma_c$ is given by:
\begin{itemize}
 \item if $\iota >0$, $\Gamma_c = +\infty$  (unconditional stability)
 \item if $\iota < 0$,   $\Gamma_c = \frac{4}{3 \pi b_c^2 (1-b_c^2)}$ with $b_c \approx 0.623$ the unique positive root of the function
 \[
 s(b)=2b^3  - \frac{4}{3} b + (b^4- b^2) \ln \frac{1-b}{1+b}.
 \]
 \end{itemize}
  In particular, all modes decay under the  condition
 \begin{equation} \label{SC}
   \frac{\gamma |\iota|}{2 \pi} < \Gamma_c .
   \end{equation}
In the recent work \cite{AO23},  such linear stability for $\Gamma < \Gamma_c$ in the non-diffusive case $\nu = 0$ was revisited and linked to a  mixing phenomenon, that is a transfer from low to high frequencies in $p$. This phenomenon is related to the free-transport operator $\pa_t + p \cdot \na_x$, itself related to the self-propulsion of the particles. It leads to the decay of integral quantities in $p$, notably to the decay of the right-hand side of the Stokes equation in \eqref{SS} and from there to the decay of the velocity field $u$. As $\nu=0$, explicit computations are possible, and the stability analysis comes down to  the decay of Fourier transforms \emph{on the sphere} (as the orientation $p \in \bbS^2$ substitutes to the usual velocity variable $v \in \R^d$ in other kinetic models). One key feature is that the decay of such Fourier transforms is limited, leading to a weaker mixing than in classical kinetic models, where the orientation variable $p \in \bbS^2$ is replaced by a usual velocity variable $v \in \R^d$.

In our recent paper \cite{CZDGV22}, we carried a sharp stability study
of the linearized version of \eqref{SS3} for both $\nu=0$ and
$0 < \nu \ll 1$.  We provided slightly more accurate stability results
in the case $\nu = 0$, but more importantly, were able to show that
the linear stability criterion \eqref{SC} is still the right one for
small $\nu > 0$. The introduction of rotational diffusion makes the
analysis much harder, as all explicit expressions are lost. We will
recall elements of this analysis below.

The purpose of this paper is to extend the linear stability result of
\cite{CZDGV22} to the full nonlinear model \eqref{SS3}. We present in
the next section our main result, explain the main difficulties,
notably in comparison to previous linear results. We also provide the
general strategy of the proof.

\subsection{Statement and sketch of proof}

Our main result can be stated as follows.
\begin{theorem}\label{thm:nonstab}
  Let \(s \in \mathbb{N}\) with $s > \frac{7}{2}$ and assume \eqref{SC}. There exist
  constants $C_0, \nu_0, \delta_0 > 0$ depending on $\gamma$ and $\iota$ such that for all
  $\nu \le \nu_0$, and for all initial data $\psi^{\mathrm{in}}$ satisfying
  \begin{equation}\label{eq:stabthreshold}
    \|\psi^{\mathrm{in}}\|_{H^s_x L^2_p} \le \delta_0 \nu^{\frac{3}{2}},
  \end{equation}
  system \eqref{SS3} has a unique global-in-time solution $\psi$ that satisfies
  \begin{equation}\label{eq:globalstab}
    \sup_{t \ge 0}  \|\psi(t)\|_{H^s_x L^2_p}^2
    + \nu \int_{0}^\infty  \|\nabla_p \psi(t)\|_{H^s_x L^2_p}^2\, \dd t
    \le C_0 \, \nu^{-1} \|\psi^{\mathrm{in}} \|^2_{H^s_x L^2_p}.
  \end{equation}
\end{theorem}
This theorem shows nonlinear stability of the solution $\Psi^{\mathrm{iso}}$ of \eqref{SS} (or equivalently of the trivial solution $\psi = 0$ of \eqref{SS3}) under a smallness condition on the initial datum that is explicit in terms of $\nu$. In other words, the nonlinear stability threshold \eqref{eq:stabthreshold} provides an estimate of the basin of attraction of the uniform equilibrium $\Psi^{\mathrm{iso}}$, depending on the strength of rotational diffusion. Determining the optimal threshold is an interesting open problem, recently resolved in other contexts, such as the stability of Couette flow in the Navier--Stokes equations (see \cites{WZ21,BGM20,MZ22} and references therein), or the stability of global Maxwellians in the Vlasov--Poisson--Landau system \cite{CLN21} and the Boltzmann equation \cite{BCZD24}. See also \cite{Wea} for an optimal stability criterion for nematic suspensions via an entropy method. We explain just before \cref{subsec:advection-diffusion} why addressing the question of the optimal threshold is still out of our reach.

Theorem~\ref{thm:nonstab} has two main features. First, it shows stability under condition \eqref{SC}, which is optimal in view of the linear analysis in \cites{saintillan-shelley-2008-instabilities,AO23,CZDGV22}. Second, it does not require any spatial diffusion term $\kappa \Delta_x \psi$ in the evolution equation \eqref{SS3}. One could add such a term and still prove stability under the same threshold \eqref{eq:stabthreshold}, \emph{independently of $\kappa$}. From this perspective, our result contrasts with \cite{AO23}, where such translational diffusion is added to the right-hand side of \eqref{SS3eq1}.

In the case of pullers ($\iota > 0$), the authors of \cite{AO23} prove nonlinear
stability of the incoherent state $\Psi^{\mathrm{iso}}$ both for $x \in \mathbb{T}^d$ and
$x \in \mathbb{R}^d$, under the condition
\[
  \|\psi^{\init}\|_{H^2_x L^2_p} \ll \min(\nu,\kappa).
\]
In the case of pushers ($\iota < 0$), they prove nonlinear stability of the
incoherent state under the stringent assumption $\Gamma = o(\nu^{1/2})$ instead of
$\Gamma < \Gamma_c$, and the initial constraint
\[
  \|\psi^{\init}\|_{H^2_x L^2_p} \ll \nu^{1/4} \min(\nu^{1/2},\kappa^{1/2}).
\]
These restrictions allow placing the terms $u \cdot \nabla_x \psi$ and $\frac{3\gamma}{4\pi} (p \otimes p) : E(u)$ on the right-hand side of \eqref{SS3eq1} and treating the model as a perturbation of the advection-diffusion equation
\[
\pa_t \psi + p \cdot \na_x \psi - \nu \Delta \psi = 0.
\]
Relaxing these smallness requirements compels us to adopt a different strategy, described below.

Our general approach relies on a classical bootstrap argument. We introduce the maximal time $T$ on which various smallness assumptions on $u$ and $\psi$ hold (see in particular \eqref{BA0}–\eqref{BA1}–\eqref{BA2} in Section~\ref{subsec_bootstrap}), and we show that improved smallness conditions hold up to time $T$ (see Section~\ref{sub:boot}). This implies that $T = \infty$, and the stability estimate follows.

To show improvement of these bootstrap assumptions, we derive from \eqref{SS3} a Volterra equation in $u$:
\begin{equation}\label{eq:Volter}
u(t) + \int_0^t K(t,\tau)\, u(\tau)\, \dd\tau = f(t)
\end{equation}
where the kernel $K(t,\tau) \in L(H^s_x, H^s_x)$ and the source term $f(t) \in H^s_x$ are defined via integral quantities over $p$, involving the solution operator $S_u(t,\tau)$ associated to the operator
\[
\pa_t + (p + u) \cdot \na_x - \nu \Delta_p
\]
on $\mathbb{T}^3 \times \mathbb{S}^2$. This reformulation and expressions for $K$ and $f$ are detailed in Section~\ref{sub:volterra}. We proceed in four main steps:
\begin{itemize}
\item The first key step, discussed further in Section~\ref{subsec:advection-diffusion}, establishes enhanced dissipation and mixing decay estimates for the solution of the advection-diffusion equation
  \[
  \pa_t g + (\bv+p) \cdot \na_x g = \nu \Delta_p g , \qquad g|_{t=0} =g^{\init}
  \]
  \emph{under a suitable smallness assumption on $v$} (to later be replaced by $u$). These mixing estimates rely on the vector field method.

\item Using these estimates, we obtain decay in time for the kernel $K$ and the source term $f$ in the Volterra equation (see Proposition~\ref{prop:estimates_K}).

\item We then analyze the Volterra equation \eqref{eq:Volter} under the bootstrap assumptions to improve smallness of $u$ and show integrable decay.

\item Finally, with decay and smallness of $u$ in hand, we return to the equation for $\psi$ to improve the smallness conditions on $\psi$.
\end{itemize}
We emphasize that this strategy, combining a Volterra formulation for $u$ with sharp analysis of the advection-diffusion operator, was already used in our linear analysis \cite{CZDGV22}. However, in the linear case, we only considered $u = 0$, enabling mode-by-mode Fourier analysis in $x$. The introduction of the advection term $u \cdot \nabla_x$ couples all modes, introducing significant analytical difficulties.

This affects the mixing estimates via the vector field method. These estimates are based on a family of vector fields indexed by Fourier modes $k$, of the form $\chi_k(p)\, J_k(t,\nabla_p)$, where $\chi_k$ is supported away from $p = -k/|k|$. This localization is not preserved under mode coupling, adding complexity compared to \cite{CZDGV22}. More generally, achieving good commutation properties of our vector fields with the advection-diffusion operator (and a fortiori with the full nonlinear operator) is challenging. These vector fields are tailored to address critical points of the advection field $V(p) = p$ on the sphere, but are much less flexible than in Euclidean settings. This accounts for the relatively strong stability threshold $O(\nu^{3/2})$ in Theorem~\ref{thm:nonstab}. We are currently unable to adapt the sharp stability techniques of \cite{CLN21} to propagate nonlinearly the full suite of linear enhanced dissipation and mixing estimates.

The estimates for the advection-diffusion problem are independent of the active suspension model, and are explained in more detail in the next section. They are likely of independent interest and are proved separately in Sections~\ref{sec2} and~\ref{sec3}. The remaining steps in the proof of Theorem~\ref{thm:nonstab} for the active suspension model are carried out in Section~\ref{sec4}.

\subsection{Advection-diffusion equations on the sphere}\label{subsec:advection-diffusion}

An important and independent part of our analysis are the mixing and
enhanced dissipation properties on the sphere. These are the
properties of solutions \(g=g(t,x,p)\) for \(x \in \T^3\) and \(p \in
\bbS^2\) solving the  advection-diffusion equation
\begin{equation}\label{transport-diffusion}
  \pa_t g + (\bv+p) \cdot \na_x g = \nu \Delta_p g , \qquad g|_{t=0} =g^{\init},
\end{equation}
where $g^{\init}$ is the assigned initial datum, $\nu\in(0,1)$ is a
diffusivity parameter, and \(\bv = \bv(t,x)\) is a divergence free
vector field satisfying adequate smallness and decay assumptions. We
are interested in semigroup-type estimates for
\eqref{transport-diffusion} that describe enhanced dissipation, as
well as mixing estimates for integrated quantities.

To state our results, we expand $g$ and $\bv$ in Fourier series in
$x$, calling $k\in \ZZ^3$ the corresponding Fourier variable.  For
$k\neq 0$, \eqref{transport-diffusion} becomes
\begin{equation}\label{eq:ADEforceU}
  \pa_t g_k + \ii p \cdot k\, g_k = \nu \Delta_p g_k
 + |k| \, \bV g_k, \qquad
  \bV g_k :=  -\ii \sum_{\ell\in \ZZ^3} \hat k \cdot \bv_{k-\ell} g_\ell,
\end{equation}
where $\hat{k} := k/|k|$. Without loss of generality, we will assume
$g_0=0$ initially, as this mode simply satisfies the standard heat
equation
\begin{equation}
  \pa_t g_0 = \nu \Delta_p g_0
\end{equation}
and hence such condition is preserved by the evolution. Besides
providing estimates on $ g $, we are interested in integrals of the
form
\begin{equation} \label{defHkVk}
  \sfV_k[g] := \int_{\bbS^2} g_k(p)\, Z_k(p)\, \na (p \cdot \hat{k})\, \dd p,
\end{equation}
for an arbitrary family $\{Z_k\}_{k \neq 0}$ of smooth, possibly
vector-valued, functions. Our main result is based on the assumption
that
\begin{equation} \label{BA} \tag{H}
\sup_{t \ge 0} \|v(t)\|_{H^s} + \left(  \int_0^{\infty} \|v(t)\|_{H^s}^2 \dd t \right)^{\frac12} \le \eps \nu^{\frac54},
\end{equation}
for some $s> \frac{5}{2}$ and $\eps\in (0,1)$ small enough.
\begin{theorem}\label{thm:linmain}
  Let $s> \frac{5}{2}$,  $0<s'<s+\frac 14$. There exists constants
  $C_0,\eps, \nu_0,\eta_1 > 0$ with the following properties. For all
  $\nu \le \nu_0$, if \eqref{BA} holds, then the solution to
  \eqref{transport-diffusion} satisfies the enhanced dissipation
  estimate
  \begin{equation}\label{eq:EnhDissipg}
    \| g(t)  \|_{H^s_xL^2_p}\le   C_0\e^{-\eta_1 \nu^{\frac12}  t}   \| g^{\init} \|_{H^s_xL^2_p},
  \end{equation}
  and the mixing estimate
  \begin{multline}\label{eq:mixDissipg}
    \sum_{k\neq 0} |k|^{2s'} \left|\sfV_k[g_k(t)]\right|^2\\
    \leq C_0\left( \frac{\nu^{\frac12}}{\min\{1,\nu^{\frac12} t\}}\right)^{3}
    \sup_{k} \left(  \|Z_k\|_{W_p^{1,\infty}}^2 + \|Z_k\|_{H^2_p}^2\right)
    \|g^{\init}, \na_p   g^{\init},\na_p^2    g^{\init}\|_{H^s_xL^2_p}^2   ,
  \end{multline}
 for all $t\geq 0$.
\end{theorem}
\begin{remark}
  Only the weaker condition
  \begin{equation} \label{Hrelaxed}
	\sup_{t \ge 0} \|v(t)\|_{H^s} + \left(  \int_0^{\infty} \|v(t)\|_{H^s}^2 \, \dd t \right)^{\frac{1}{2}} \leq \varepsilon \nu^{1/2}
  \end{equation}
  is required to establish the enhanced dissipation estimate
  \eqref{eq:EnhDissipg}; see Proposition~\ref{thm:density:control-convection}.

  Recall that the term \emph{enhanced dissipation} refers to the phenomenon
  where the timescale of exponential decay (in our case, of order
  \( \nu^{-1/2} \)) is much shorter than the usual diffusive timescale
  \( \nu^{-1} \) associated with the heat equation (see
  \cites{CZDE20,CKRZ08}). This effect is linked to mixing, which transfers
  energy to higher frequencies in the \( p \)-variable, where the viscous
  dissipation operator \( -\nu \Delta_p \) acts more effectively.  We do not know if
  condition~\eqref{Hrelaxed} is optimal for proving enhanced dissipation
  estimates.
\end{remark}

The proof of Theorem~\ref{thm:linmain} is rather involved and does not rely on any special properties of the sphere \(\bbS^2\).  It is based on two key ingredients: the derivation of hypocoercive estimates, inspired by
\cites{villani-2009-hypocoercivity,BCZ17}, and the use of the vector field method, following ideas from \cites{CLN21,CZ20}.

The central hypocoercivity estimate, given in \eqref{eq:EnhDissipg}, concerns the advection-diffusion operator
\[
L_\nu = \pa_t + p \cdot \na_x - \nu \Delta_p,
\]
i.e., the case with $u=0$. This estimate is established in Section~\ref{sec2}. Compared to the linear analysis in \cite{CZDGV22}, we present a more general version in Proposition~\ref{thm:density:control-convection}, where:
\begin{itemize}
\item the function $g$ is replaced by a general tensor $Y$ (which will later be taken as $g$, $Jg$ or $J^2g$ for some suitable vector field $J$);
\item a source term, such as the one in \eqref{eq:ADEforceU}, is included. This term will later originate either from the convection term (e.g., $- v \cdot \na_x g$,
  $- J(v \cdot \na_x g)$ or $- J^2 (v \cdot \na_x g)$), or from commutators between $J$ and $L_\nu$.
\end{itemize}
The estimate reveals improved decay properties for the quantity $\na (p \cdot k)Y$, which vanishes at $p = \pm \hat{k}$.

Although setting $Y=g$ and taking the source term to be $- v \cdot \na_x g$ is
already sufficient to obtain the exponential decay in \eqref{eq:EnhDissipg},
\emph{obtaining polynomial decay estimates for integral quantities}, uniformly
in $\nu$, requires the use of adapted vector fields.

These vector fields are introduced in Section~\ref{sec3}, where they are used to establish the mixing estimates \eqref{eq:mixDissipg}. As in \cite{CZDGV22}, after Fourier transform in $x$, the vector field takes the form
\[
J_k = \alpha_k \na_p + \ii \left(\frac{|k|}{\nu}\right)^{1/2} \beta_k \na_p (p \cdot \hat{k})
\]
where
\[
(\alpha_k,\beta_k) = (\alpha_{k,\nu}, \beta_{k,\nu})(t) =  (\alpha,\beta)(\nu^{1/2}|k|^{1/2}t)
\]
with $\alpha,\beta$ carefully chosen (slightly differently and with slight improvements over \cite{CZDGV22}). The evolution of $\alpha_k$ and $\beta_k$ is tailored so that their commutator with $L_\nu$ has better structure.  In particular, a critical term in the commutator vanishes near $p = \hat{k}$, which aligns with the hypocoercive estimate’s improved decay for quantities vanishing at $p = \pm \hat{k}$. However, this requires introducing an additional cut-off function $\chi_k(p)$, supported away from $p = -\hat{k}$.

A complication arises because the convection term is non-local in Fourier space
and tends to disrupt this localization. This necessitates two layers of
estimates: non-localized estimates leading to suboptimal decay and localized
ones leading to optimal decay.  This two-pronged analysis is carried out in
Sections~\ref{subsec:L2estimates} and \ref{subsec:hypocoercive}.  The usage of
the non-localized estimates is the main restriction on the threshold condition
in \eqref{BA}.

Once this framework is in place, the abstract source term $F$ is replaced with the actual convection term $-u \cdot \na_xg$. Due to mode coupling in Fourier space, only summing over all modes allows for closing the estimates, as shown in Section~\ref{sub:roleconvection}. Finally, the mixing estimates for integral quantities are presented in Section~\ref{sec:mixingestimates}.

\subsection{Notation}

We write \( a \lesssim b \) to indicate that \( a \leq C b \), where the constant \( C \) depends only on the fixed parameters \( \gamma \) and \( \iota \). Similarly, we write \( a \sim b \) if both \( a \lesssim b \) and \( b \lesssim a \) hold.

In Sections~\ref{sec2} and~\ref{sec3}, which focus on equation~\eqref{transport-diffusion} and its variations, the notation \( a \lesssim b \) will instead refer to an inequality with an \emph{absolute constant} \( C \), independent of any parameters.

For any real number \( s \), we use the following Sobolev space notation:
\begin{itemize}
  \item \( H^s = H^s_x(\mathbb{T}^3) \) denotes the standard Sobolev space in the spatial
  variable \( x \in \mathbb{T}^3 \) equipped with the norm
  \begin{equation*}
    \| u \|_{H^s}^2 = \sum_{k \in \mathbb{Z}^3} [\max(1,|k|)]^{2s} |\hat u_k|^2
  \end{equation*}
  defined over the Fourier modes \(u_k, k \in \mathbb{Z}^3\),
  \item \( \mathcal{H}^s = H^s_x(\mathbb{T}^3; L^2_p(\mathbb{S}^2)) \) denotes the corresponding space of \( L^2 \)-functions in \( p \in \mathbb{S}^2 \) with Sobolev regularity in \( x \),
  \item \( \mathbb{Z}^3_* = \mathbb{Z}^3 \setminus \{0\} \).
\end{itemize}

\section{Advection-diffusion with forcing: hypocoercivity} \label{sec2}

Equation \eqref{eq:ADEforceU} can be seen as a forced
advection-diffusion equation, as long as we impose mild assumptions on
the right-hand side.  In this section, we derive energy estimates for
a slight generalization of \eqref{eq:ADEforceU}.

\subsection{A general hypocoercivity setup and enhanced dissipation}
The basic starting block is to find the hypocoercive dissipation functional for
a tensor $Y$ solving the advection-diffusion equation \eqref{eq:ADEforceU}
on the sphere.  We will then first apply this estimate with \(Y=g\) and
then later with some vector fields \(Y=Jg\), \(Y=JJg\), for a suitable $J$.  As
a small generalization of \eqref{eq:ADEforceU}, we study a family of
$(0,n)$-tensors \((Y_k)_{k\in\ZZ^3_*}\) evolving as
\begin{equation}
  \label{eq:tensor:advection-diffusion}
  (\partial_t + \ii p \cdot k - \nu \Delta_p) Y_k = |k| F_k,
\end{equation}
for some forcing term $F_k$.  Here we take as the Laplacian the connection
  Laplacian defined as \(\nabla_p \cdot \nabla_{p}\) with the covariant derivative and the
  inner product from the metric.

From the analysis in \cite{CZDGV22},  we expect enhanced dissipation on the
time-scale \(O(\nu^{-\frac12}|k|^{-\frac12}) \) so that we set the rescaled
  time
\begin{equation}
  \label{eq:rescaled-time-h}
  h := \nu^{\frac12}|k|^{\frac12} t
\end{equation}
and define the time-dependent weights
\begin{equation}\label{eq:time-weights-def}
  (a_k,b_k,c_k) := (a,b,c)(h), \quad (a'_k,b'_k,c'_k) := (a',b',c')(h)
\end{equation}
for non-negative functions \(a,b,c\) to be specified below.  Denoting
by $\l\cdot,\cdot\r$ and $\|\cdot\|$ the $L^2$ inner product and norm,
for any smooth non-negative function \(\chi = \chi(p)\), we define the
sesquilinear form
\begin{align}\label{eq:tensor:e-mode}
    E_{\chi,k}(Y_k, \tilde Y_k)
    &:= \l Y_k \chi, \tilde Y_k \chi\r
      + \left( \frac{\nu}{|k|} \right)^{\frac12}  a_k \l \na_p Y_k \chi, \na_p
      \tilde Y_k \chi\r\notag\\
    &\qquad
      + b_k \l \ii \na_p(p \cdot \hat{k}) Y_k \chi, \na_p \tilde Y_k \chi \r
      + b_k \l \nabla_p Y_k \chi, \ii \na_p(p \cdot \hat{k}) \tilde Y_k \chi\r  \notag\\
    &\qquad +  \left( \frac{\nu}{|k|} \right)^{-\frac12} c_k
      \l \na_p(p \cdot \hat{k}) Y_k \chi,
      \na_p(p \cdot \hat{k}) \tilde  Y_k \chi \r
\end{align}
and the corresponding  energy functional
\begin{multline}
  E_{\chi,k}(Y_k)  := E_{\chi,k}(Y_k, Y_k)  = \|Y_k \chi \|^2  + \left( \frac{\nu}{|k|} \right)^{\frac12}  a_k \|\na_p Y_k \chi \|^2 \\
  + 2 b_k \Re   \l \ii \na_p(p \cdot \hat{k}) Y_k \chi, \na_p  Y_k \chi \r  +   \left( \frac{\nu}{|k|} \right)^{-\frac12} c_k
  \|\na_p(p \cdot \hat{k}) Y_k \chi \|^2.
\end{multline}
We also define  the dissipation functional
\begin{equation}\label{eq:tensor:dis-mode}
  \begin{split}
  D_{\chi,k}(Y_k)
  &:= \frac{\nu}{|k|}  \|\nabla_pY_k \chi\|^2
  + \frac{\nu}{|k|} a_k \left( \frac{\nu}{|k|} \right)^{\frac12} \| \na^2_p Y_k \chi\|^2\\
  &\quad + b_k \| \na_p(p \cdot \hat{k}) Y_k \chi\|^2
  + \frac{\nu}{|k|} c_k  \left( \frac{\nu}{|k|} \right)^{-\frac12}
  \| \na_p \left( \na_p(p \cdot \hat{k}) Y_k \right) \chi\|^2.
  \end{split}
\end{equation}
The first lemma shows that for a good choice of $a,b,c$,  these functionals provide a good estimate.
\begin{lemma}\label{thm:tensor:basic-dissipation}
 Let \(\chi = \chi(p)\) a smooth cut-off and set
  \begin{equation}\label{eq:choice-weights-abc}
    a(h) = A \min(h,1),  \quad b(h) = B \min(h^2,1), \quad c = C \min(h^3,1)
  \end{equation}
  for positive constants \(A,B,C\). There exist constants \(\nu_0,B_0,M>0\)
  such that for \(\nu \le \nu_0\),  \(B < B_0\) and
  \begin{equation}\label{eq:relation-constants-a-b-c}
    A = B^{2/3}, \qquad
    C = \frac{100B^2}{A}
  \end{equation}
any  \((0,n)\)-tensor solution $Y_k$ of  \eqref{eq:tensor:advection-diffusion} satisfies
  \begin{equation*}
    \begin{split}
      &\frac 12 \ddt E_{\chi,k}(Y_k)
        + \frac 34 |k|\, D_{\chi,k}(Y_k)
        - |k| \Re E_{\chi,k}(Y_k, F_k) \\
      &\qquad\le |k| M c_k \left(\frac{\nu}{|k|}\right)^{\frac12}
        \| Y_k \chi \|^2 \\
      &\qquad\quad
        + |k| M \left[ \frac{\nu}{|k|} \| Y_k \nabla \chi \|^2
        +  \left(\frac{\nu}{|k|}\right)^{\frac32} a_k
        \| \nabla Y_k \nabla \chi \|^2
        +  \left(\frac{\nu}{|k|}\right)^{\frac12} c_k
        \| \nabla (p\cdot \hat k) Y_k \nabla \chi \|^2 \right].
    \end{split}
  \end{equation*}
\end{lemma}
\begin{remark}
As we shall see, $a,b,c$ satisfy the condition that
$b_k^2 < \frac{1}{2} a_k c_k$. This makes the quadratic form
$E_{\chi,k}$ is coercive, in the sense that
\begin{equation*}
  \|Y_k \chi \|^2
  + \left( \frac{\nu}{|k|} \right)^{\frac12}  a_k \| \na_p Y_k \chi \|^2
  + \left( \frac{\nu}{|k|} \right)^{-\frac12} c_k
  \| \na_p(p \cdot \hat{k}) Y_k  \chi\|^2
  \lesssim  E_{\chi,k}(Y_k).
\end{equation*}
\end{remark}
\begin{remark}
  The statement of the lemma would still be true on any Riemannian
  manifold for which \((0,n)\)-tensors $Z$ satisfy
  \(|[\nabla,\Delta] Z| \lesssim |\nabla Z|\). To illustrate this, we
  will not use the fact that on $\bbS^2$ the commutator
  $[\nabla_p,\Delta_p] Z = -\nabla_p Z$ gives some extra coercivity property.
\end{remark}
\begin{proof}[Proof of Lemma \ref{thm:tensor:basic-dissipation}]
  Through the change of variables
  \begin{equation}\label{eq:mode-time-rescaling}
    \nu' := \frac{\nu}{|k|}, \quad t' := |k| t, \quad k' := \frac{k}{|k|},
  \end{equation}
  we can restrict to the case $|k| = 1$, so that $k = \hat{k}$. To lighten notations, we drop the subscript $k$, writing $Y$ instead of $Y_k$, $F$ instead of $F_k$, and so on, as well as the subscript $p$ on the various differential operators. We find through standard estimates  that
  \begin{equation}\label{eq:basicL2}
    \frac 12 \ddt \| Y \chi \|^2
    + \nu \| \nabla Y \chi \|^2
    \le \Re \l Y \chi, F \chi \r+  2 \nu \| \nabla Y \chi \|\, \| Y \nabla \chi \|.
  \end{equation}
  Similarly, as
  \begin{equation*}
    \pa_t \na Y + \ii  p \cdot k  \nabla Y - \nu \Delta \na Y  =   \na F - \ii \nabla(p\cdot k)  Y + \nu [\na, \Delta] Y,
  \end{equation*}
  we get
  \begin{equation}\label{eq:basicH1}
    \begin{aligned}
      \frac 12 \ddt \| \nabla Y \chi \|^2
      + \nu \| \nabla \nabla Y \chi \|^2
      & \le  \Re \l \na Y \chi, \na F \chi \r +  2 \nu \| \nabla \nabla Y \chi\| \, \|\nabla Y \nabla \chi\|    \\
      & \quad +   \| \nabla Y \chi \|\, \| \nabla(p\cdot k) Y \chi \| + \nu \|[\na, \Delta] Y \chi\|\, \| \na Y \chi\|.
    \end{aligned}
  \end{equation}
  Also, using
  \begin{align*}
    & \pa_t (\ii \na(p \cdot k) Y)+ \ii  p \cdot k  ( \ii \na(p \cdot k) Y)  - \nu   \ii \na(p \cdot k)  \Delta Y   = \ii \na(p \cdot k)    F, \\
    & \pa_t \na Y + \ii  p \cdot k \na  Y - \nu \na  \Delta Y  =   \na F - \ii \nabla(p\cdot k)  Y,
  \end{align*}
  we find
  \begin{equation*}
    \begin{split}
      &  \ddt \Re \l \ii \nabla(p\cdot k) Y \chi, \nabla Y \chi \r    + \| \nabla(p\cdot k) Y \chi \|^2 \\
      & \qquad\le \nu \l \ii \na(p \cdot k) \Delta Y \chi, \na Y \chi \r + \nu \l \ii \na(p \cdot k) Y \chi, \na \Delta Y \chi \r \\
      &\qquad\quad
      + \Re \l \ii \nabla(p\cdot k) F \chi, \nabla Y \chi \r
      + \Re \l \ii \nabla(p\cdot k) Y \chi, \nabla F \chi \r \\
      &\qquad \le \nu \| \nabla \nabla Y \chi \|
      \Big[
        2 \| \nabla(\nabla(p\cdot k) Y) \chi \| + \|Y \chi\|
        + 2 \| \nabla (p\cdot k) Y \nabla \chi \|
      \Big]\\
      &\qquad\quad
      + \Re \l \ii \nabla(p\cdot k) F \chi, \nabla Y \chi \r
      + \Re \l \ii \nabla(p\cdot k) Y \chi, \nabla F \chi \r  .
    \end{split}
  \end{equation*}
  Finally using
  \[  \pa_t (\ii \na(p \cdot k) Y)+ \ii  p \cdot k  ( \ii \na(p \cdot k) Y)  - \nu   \Delta  (\ii \na(p \cdot k)  Y)   = \ii \na(p \cdot k)    F  +  \nu [\ii \na(p \cdot k) , \Delta] Y \]
  together with the commutator formula (see \cite[Lemma~A.1]{CZDGV22})
  \begin{equation}
    \Delta(\nabla(p\cdot k) \otimes Y)
    = - \nabla(p\cdot k) \otimes Y
    - 2 (p\cdot k) \nabla Y
    + \nabla(p\cdot k) \otimes \Delta Y,
  \end{equation}
  we get
  \begin{equation*}
    \begin{split}
      &\frac 12 \ddt \| \nabla(p\cdot k) Y \chi \|^2
      + \nu \| \nabla( \nabla(p\cdot k) Y ) \chi \|^2
      + \nu \| \nabla(p\cdot k) Y \chi \|^2 \\
      &\qquad\le 2 \nu \| \na (\nabla(p\cdot k) Y) \chi \|\,
      \| \nabla(p\cdot k) Y \na \chi \|
      + \Re \l \nabla(p\cdot k) Y \chi, \nabla(p\cdot k) F \chi \r  \\
      &\qquad\quad+ 2 \nu \| \nabla(p\cdot k) Y \chi \|\,   \| \nabla Y \chi \|.
    \end{split}
  \end{equation*}
  Hence, gathering these estimates leads to
  \begin{equation*}
    \begin{split}
      &\frac 12 \ddt E_\chi(Y) + D_\chi(Y) - \Re E_\chi(Y,F) \\
      \le & \: \nu a' \| \nabla Y \chi \|^2
      + \nu^{\frac12} a \| \nabla Y \chi \|\, \| \nabla(p\cdot k) Y \chi \| +  \nu^{\frac32} a  \|[\na, \Delta] Y \chi\| \| \na Y \chi\| \\
      &\quad
      + \nu^{\frac12} b' \Re \l \ii \nabla(p\cdot k) Y \chi, \nabla Y
		\chi \r\\
	  &\quad
      + 2 \nu b
      \| \nabla \nabla Y \chi \|
      \big( \| \nabla(\nabla(p\cdot k) Y) \chi \| + \| Y \chi \|  +     \| \nabla(p\cdot k) Y \na \chi \|  \big)
      \\
      &\quad
      + c' \| \nabla(p\cdot k) Y \chi \|^2
      + 2 \nu^{\frac12} c
      \| \nabla(p\cdot k) Y \chi \|\,
      \| \nabla Y \chi \|
      \\
      &\quad + 2 \nu \| \nabla Y \chi \|\, \| Y \nabla \chi \|
      + 2 \nu^{\frac32} a \| \nabla \nabla Y \chi\| \, \|\nabla Y \nabla
      \chi\| \\
      &\quad+ 2 \nu^{\frac12} c  \| \na (\nabla(p\cdot k) Y) \chi \|\,      \| \nabla(p\cdot k) Y \na \chi \|.
    \end{split}
  \end{equation*}
  We are now ready to list the constraints on the coefficients $a,b,c$ in order to absorb most of the terms of the right-hand side.
  It is suffices that
  \begin{itemize}
  \item $a' \le \frac{1}{100}$ to absorb the first term at the right-hand side by the term $\nu \|\na Y \chi\|^2$ in  the dissipation functional $D_\chi(Y)$ at the left-hand side.
  \item $a^2 \le \frac{b}{100}$ to absorb the second term by the term $\nu \|\na Y \chi\|^2 + b \|\na (p \cdot  k) Y \chi\|^2$.
  \item $a \le 1$ and $\nu$ small to absorb the third term by  $\nu \|\na Y \chi\|^2$.
  \item $(b')^2 \le \frac{b}{100}$ to absorb the fourth term by $\nu \|\na Y \chi\|^2 + b \|\na (p \cdot  k) Y \chi\|^2$.
  \item $b^2 \frac{\nu^{\frac12}}{a} \le \frac{\nu^{\frac12} c}{100}$ to bound the fifth term by
    \[ \frac{1}{10} \left( \nu^{\frac32} a  \|\na^2 Y \chi\|^2 + \nu^{\frac12} c \|\na \left( \na(p \cdot k) Y \chi \right) \|^2 +  \nu^{\frac12} c  \|Y \chi\|^2 + \nu^{\frac12} c   \| \nabla (p\cdot k) Y \nabla \chi \|^2 \right)\]
    the first two terms being absorbed by $D_\chi(Y)$ at the left-hand side.
  \item $c' \le \frac{b}{100}$ to absorb the sixth term by $b  \|\na (p \cdot  k) Y \chi\|^2$.
  \item $c^2 \le \frac{b}{100}$ to absorb the seven term by $\nu \|\na Y \chi\|^2 + b \|\na (p \cdot  k) Y \chi\|^2$.
  \end{itemize}
  These conditions are satisfied with our choice of $a,b,c$ in
    ~\eqref{eq:choice-weights-abc}. The last three terms involving $\na \chi$ can
  be treated classically using Young's inequality, without further constraints
  on $a,b,c$.
\end{proof}

\subsection{The energy inequality for $g$}
As a first application of Lemma \ref{thm:tensor:basic-dissipation}, we deduce an energy inequality for the solution $g$ to \eqref{eq:tensor:advection-diffusion}. For the trivial cutoff \(\chi \equiv 1\), we use the shorthand
notation \(E_{k} = E_{\chi,k}\) and \(D_{k} = D_{\chi,k}\). By the
previous lemma, taking
\(Y_k=g_k\), we get:
\begin{lemma}\label{thm:density:basic-dissipation}
  Assume that the scalar function \(g\) satisfies
  \eqref{eq:tensor:advection-diffusion}.  Then with the same functions as in
  \cref{thm:tensor:basic-dissipation} we find (for possible
  smaller \(\nu_0\) and \(B_0\)) that
  \begin{equation}
    \frac 12 \ddt E_{k}(g_k)
    + |k| a_k \left(\frac{\nu}{|k|}\right)^{\frac12} \| g_k \|^2
    + \frac 58 |k|\, D_{k}(g_k)
    - |k| \Re E_{k}(g_k, F_k)
    \le 0.
  \end{equation}
\end{lemma}
For the proof, we first recall  the following interpolation result, see \cite{CZDGV22}*{Lemma 4.2}.
\begin{lemma}\label{thm:hypo-interpolation}
  For all $\sigma\in (0,1]$, all vectors $e \in \bbS^2$ and all
  complex-valued $g\in H^1(\bbS^2)$, the following inequality holds
  \begin{equation}\label{eq:spectralgap}
    \sigma^{\frac12}\|g\|^2\le \frac{\sigma}{2} \| \nabla g\|^2
    + 2 \| \nabla(p\cdot e) \,g  \|^2.
  \end{equation}
\end{lemma}

We now can perform the proof.
\begin{proof}[Proof of \cref{thm:density:basic-dissipation}]
  Applying \cref{thm:tensor:basic-dissipation} yields
  \begin{equation}\label{eq:density:basic-dissipation:proof-start}
    \frac 12 \ddt E_{k}(g_k) + \frac 34 |k|\, D_{k}(g_k)
    - |k| \Re E_{k}(g_k, F_k)
    \le |k| M c_k \left(\frac{\nu}{|k|}\right)^{\frac12}
    \| g_k \|^2.
  \end{equation}
  By the change of variables \eqref{eq:mode-time-rescaling}, we can
  again restrict to \(|k| = 1\). We distinguish between long times and
  short times. For long times $t$ such that
  \(\nu \le b_k = b(\nu^{\frac12} t)\), we apply the interpolation
  result with $\sigma = \frac{\nu}{b_k}$. It gives
  \begin{equation*}
    D_k(g_k)
    \ge \nu \|\na g_k\|^2 + b_k \|\na (p \cdot k) g_k\|^2
    \ge \frac{1}{2} \nu^{\frac12} b_k^{\frac12} \|g_k\|^2.
  \end{equation*}
  As \(b_k^{\frac12} \gg a_k \gg c_k\) this gives the claimed control and
  allows the absorption of the right-hand side in
  \eqref{eq:density:basic-dissipation:proof-start}.

  For the smaller times $t$ such that  \(\nu \ge b_k\),  we find by Poincaré inequality that
  \begin{equation*}
    c_k \nu^{\frac12}\| g_k \|^2
    \lesssim c_k \nu^{\frac12} \| \nabla g_k \|^2
    + c_k \nu^{\frac12} \| \nabla (p\cdot k) g_k \|^2.
  \end{equation*}
  As for such times  \(c_k \ll \nu^{\frac12}\)and $c_k \ll b_k$, we can also absorb it in
  the dissipation and control the right-hand side of
  \eqref{eq:density:basic-dissipation:proof-start}.
\end{proof}

\section{Advection-diffusion with forcing: vector fields} \label{sec3}

Adepting the strategy in \cite{CZDGV22}, we use vector fields to deduce mixing
estimates for \eqref{transport-diffusion} that are uniform in the diffusivity parameter $\nu$.  Proving
enhanced dissipation for $g$ essentially relies on the energy inequality of
Lemma \ref{thm:density:basic-dissipation} together with a suitable treatment of the forcing term given by
advection (see Section \ref{sub:roleconvection}), as done in Section \ref{subsub:densityg}. However, the proof of
uniform-in-$\nu$ mixing estimates is more involved: the use of vector fields
requires cut-offs to localize and improve basic $L^2$ estimates near the south
pole $p=-\hat k$, see Remark \ref{rem:explanationJ} and Lemma \ref{thm:j:basic-l2}.

To treat the convection term in a sharp way, we need to introduce two hypocoercivity schemes in Section \ref{subsub:applicationvector}: a non-localized version (Lemma \ref{thm:j:noncutoff-dissipation}), without losses with respect to the convection term, and a localized one (Lemma \ref{thm:j:basic-dissipation}), allowing the derivation of a sharp energy inequality. These results are then applied in Section \ref{sec:mixingestimates} to obtain mixing estimates that are uniform in $\nu$.

Our goal is to construct vector fields that have good commutation properties with the advection-diffusion equation  \eqref{transport-diffusion}.  We look for vector fields of the form
\begin{equation}
  \label{eq:generic-viscosity-vector-fields}
  J_{k} = \alpha_k \nabla_p
  + \ii \left(\frac{|k|}{\nu}\right)^{\frac12} \beta_k \nabla (p\cdot \hat k)
\end{equation}
where we again introduce the good time-scale \(h=\nu^{\frac12}|k|^{\frac12}t\) and
set
\begin{equation*}
  \alpha_k(t) = \alpha(h),\qquad
  \beta_k(t) = \beta(h)
\end{equation*}
for functions \(\alpha,\beta\) to be specified. The factor in front of
\(\beta_k\) is inspired by the fact that as $\nu \rightarrow 0$, we
want to recover \(\alpha_k \equiv 1\) and
\(\left(\frac{|k|}{\nu}\right)^{\frac12} \beta_k \equiv |k| t\). The
reason is that the vector field
$\nabla_p + \ii |k| t \na (p \cdot \hat{k})$ commutes to the free
transport part.  For the full operator in \eqref{transport-diffusion}
we find the commutators
\begin{equation*}
  [\partial_t + \ii p\cdot k - \nu \Delta_p, \alpha_k \nabla_p]
  = \nu^{\frac12} |k|^{\frac12} \alpha' \nabla_p
  - \ii |k| \alpha \nabla(p\cdot \hat k)
  - \alpha \nu [\Delta_p,\nabla_p]
\end{equation*}
and
\begin{equation*}
  \begin{multlined}
    \left[\partial_t + \ii p\cdot k - \nu \Delta_p, \ii
      \left(\frac{|k|}{\nu}\right)^{\frac12} \beta_k \nabla(p\cdot \hat k)\right]\\
    = \ii |k| \beta' \nabla(p\cdot \hat k)
    - \ii \nu^{\frac12} |k|^{\frac12} \beta \nabla(p\cdot \hat k)
    + 2 \ii \nu^{\frac12} |k|^{\frac12} \beta \nabla\left[(p\cdot \hat k) \cdot \right]
  \end{multlined}
\end{equation*}
so that
\begin{equation*}
  \begin{multlined}
    (\partial_t + \ii p\cdot k - \nu \Delta_p) J_k Y_k   =   - \ii |k| (\alpha-\beta') \nabla(p\cdot \hat k) Y_k
    + \nu^{\frac12} |k|^{\frac12} (\alpha'  + 2 \ii \beta p \cdot \hat{k}) \nabla_p Y_k  \\
    - \alpha \nu [\Delta_p,\nabla_p] Y_k  + \ii \nu^{\frac12} |k|^{\frac12} \beta \nabla(p\cdot \hat k) Y_k.
  \end{multlined}
\end{equation*}
Clearly, for the $O(1)$ term at the right-hand side to disappear, the condition
$\alpha = \beta'$ is necessary. A key idea from \cite{CZDGV22} is to
complete this condition in such a way that the second term at the right-hand side
vanishes at the \emph{north pole} $p = \hat{k}$.  The goal is to
benefit from estimates for quantities $\na (p \cdot \hat{k}) Y_k$ or
$\na (\na (p \cdot \hat{k}) Y_k)$ better than those for $Y_k$ or
$\na Y_k$ alone. In \cite{CZDGV22}, we achieve this by imposing
\begin{equation*}
  \beta'(h) = \alpha(h),\qquad
  \alpha'(h) = - 2 \ii \beta(h).
\end{equation*}
Together with the  initial data $\alpha(0) = 1$, $\beta(0) = 0$, this provides
\begin{equation*}
  \alpha(h) = \cosh( (1-\ii) h ),\qquad
  \beta(h) = \frac{1+\ii}{2} \sinh( (1-\ii) h ).
\end{equation*}
This is a good choice  up to the critical time
\(\nu^{-\frac12} |k|^{-\frac12}\). However, afterwards both terms behave
asymptotically with the factor \(\e^{(1-\ii)h}\), which is growing fast. Hence, we rather
consider here
\begin{equation}\label{eq:def-alpha}
  \alpha(h)
  = \e^{-(1-\ii)h} \cosh( (1-\ii) h )
  = \frac{1}{2}
  \left(1 + \e^{-2h} \e^{2\ii h}\right)
\end{equation}
and
\begin{equation}\label{eq:def-beta}
  \beta(h) = \e^{-(1-\ii)h} \frac{1+\ii}{2} \sinh( (1-\ii) h )
  = \frac{(1+\ii)}{4}
  \left(1 - \e^{-2h} \e^{2\ii h}\right).
\end{equation}
\begin{remark} \label{rem:alphabeta}
With this new choice of $\alpha$ and $\beta$, we have the important property that
\[ \alpha \sim 1, \quad \beta^2 \sim b\]
where $b$ is the function appearing in Lemma \ref{thm:tensor:basic-dissipation} (and now fixed).
\end{remark}
In light of \eqref{eq:def-alpha}-\eqref{eq:def-beta}, we  compute the commutator between $J_{k}$ and the advection-diffusion equation as
\begin{equation*}
  \begin{split}
    &\left[\partial_t + \ii p\cdot k - \nu \Delta_p, J_{k}\right]
    + |k| (1-\ii) \left(\frac{\nu}{|k|}\right)^{\frac12}\\
    &\qquad= |k| 2 \ii \beta \left(\frac{\nu}{|k|}\right)^{\frac12}
      \nabla_p\left[(p\cdot \hat k -1) \cdot \right]
      + |k| \left[
      - \frac{\nu}{|k|} \alpha [\Delta_p,\nabla_p]
      - \ii \left(\frac{\nu}{|k|}\right)^{\frac12} \beta
      \nabla(p\cdot \hat k)
      \right].
  \end{split}
\end{equation*}
Hence we find for a tensor \(Y_k\) solution of
\eqref{eq:tensor:advection-diffusion} that \(J_k Y_k\) satisfies
\begin{equation}\label{eq:j:evolution}
  (\partial_t + \ii p \cdot k - \Delta_p) J_k Y_k
  + |k| (1-\ii) \left(\frac{\nu}{|k|}\right)^{\frac12} J_k Y_k
  = |k| J_k F_k + |k| R_k Y_k,
\end{equation}
where
\begin{equation}\label{eq:j:commutator-error}
  R_k Y_k = 2 \ii \beta_k \left(\frac{\nu}{|k|}\right)^{\frac12}
  \nabla_p[ (p\cdot \hat k-1) Y_k ]
  - \frac{\nu}{|k|} \alpha_k [\Delta_p,\nabla_p] Y_k
  - \ii \left(\frac{\nu}{|k|}\right)^{\frac12} \beta_k
  \nabla(p\cdot \hat k) Y_k
\end{equation}
is a remainder term.
\begin{remark} \label{rem:explanationJ} We insist that through our
  choice of vector field $J_k$, the worst term at the right-hand side of
  \eqref{eq:j:commutator-error}, namely
  $2 \ii \beta_k \left(\frac{\nu}{|k|}\right)^{\frac12} \nabla_p[
  (p\cdot \hat k-1) Y_k ]$, vanishes at $p = \hat{k}$. This will allow
  us to benefit from the better controls that we have on
  $\na (p \cdot k) Y_k$ or $\na (\na (p \cdot k) Y_k)$, see the
  dissipation functional \eqref{eq:tensor:dis-mode}. Still, this
  better control will only be achieved away from the \emph{south pole}
  $p = -\hat{k}$. Indeed, near the south pole,
  $\na (p \cdot \hat{k}) Y_k$ vanishes while our term at the right-hand side does
  not, so that the former cannot control the latter. Namely, we will
  achieve optimal estimates only for $J_k g_k \chi_k$, where $\chi_k$
  is zero near the south pole $p = - \hat{k}$. To obtain control near
  this south pole, we need to define another vector field $H_k$,
  replacing the relation $\alpha' = - 2 \ii \beta$ which was our
  starting point for $J_k$ by $\alpha' = + 2 \ii \beta$. All the
  estimates we obtain for $J_k Y_k \chi_k$, with $\chi_k = 0$ near the
  south pole, extend readily to $H_k Y_k g_k$, with $g_k = 0$
  near the north pole.
\end{remark}

\subsection{$L^2$ estimates} \label{subsec:L2estimates}
We start with a basic estimate for $J_k Y_k$ without localization.
\begin{lemma}\label{thm:j:basic-l2-withoutcutoff}
  There exists $M > 0$ such that for all $k \neq 0$ it holds that
  \begin{equation} \label{estimateJwithoutcutoff}
    \begin{multlined}
      \frac 12 \ddt \| J_k Y_k\|^2
      + |k| \frac 34 \frac{\nu}{|k|} \| \nabla J_k Y_k\|^2
      + |k| \frac 34 \left(\frac{\nu}{|k|}\right)^{\frac 12}
      \| J_k Y_k \|^2
      - |k| \Re\l J_k Y_k , J_k F_k  \r \\
      \le |k| M \left[ b_k \|Y_k \|^2 + D_{k}(Y_k) \right].
    \end{multlined}
  \end{equation}
\end{lemma}
\begin{proof}
 As before, we can take the time rescaling from
  \eqref{eq:mode-time-rescaling} and assume \(|k| = 1\). We also get rid of the subscript $k$ in the notation when no confusion arises. Performing an \(L^2\) estimate on \eqref{eq:j:evolution}, we find
 \begin{equation*}
    \begin{multlined}
      \frac 12 \ddt \| JY \|^2
      + \nu \| \nabla JY \|^2
      + \nu^{\frac 12}
      \| JY  \|^2
      -\Re \l JY ,  J F  \r   \le  \Re\l
      JY, RY
      \r.
    \end{multlined}
  \end{equation*}
To bound the right-hand side, we integrate by parts the first term in the definition \eqref{eq:j:commutator-error} of $R$, and  get
 \begin{equation*}
   \begin{aligned}
    \Re\l
    JY , RY
    \r
& \lesssim
   |\beta| \nu^{\frac12}    \| \nabla J Y \| \|(p \cdot k - 1) Y\| + \nu |\alpha|  \|J Y \|  \|\na Y\| +   |\beta| \nu^{\frac12}  \|J Y \| \|\na (p \cdot k) Y\| \\
   & \lesssim   \nu^{\frac12} |b|^{\frac12}  \| \nabla J Y \| \|Y\| +  \nu   \|J Y \|  \|\na Y\| +   \nu^{\frac12} |b|^{\frac12}  \|J Y \| \|\na (p \cdot k) Y\| ,
   \end{aligned}
 \end{equation*}
 where we used Remark \ref{rem:alphabeta} for the last bound. The result follows then classically from Young's inequality.
\end{proof}
We now show an improved $L^2$ bound, when localized away from the \emph{south pole} $p = -\hat{k}$.
\begin{lemma}\label{thm:j:basic-l2}
  Let $k \neq 0$, let
  \(\chi_k = \chi_k(p), \tilde\chi_k = \tilde \chi_k(p)\) be two
  smooth functions, which are zero near \(p = - \hat k\), and such
  that \(|\chi_k| + |\nabla \chi_k| \lesssim \tilde \chi_k\). Then,
  for some \(M > 0\) independent of $k$, we have
  \begin{align*}
      &\frac 12 \ddt \| J_k Y_k \chi_k \|^2
      + |k| \frac 34 \frac{\nu}{|k|} \| \nabla J_k Y_k \chi_k \|^2
      + |k| \frac 34 \left(\frac{\nu}{|k|}\right)^{\frac 12}
      \| J_k Y_k \chi_k \|^2
      - |k| \l J_k Y_k \chi_k, J_k F_k \chi_k \r \\
       &\qquad\le |k| M \left[
        \frac{\nu}{|k|} \| Y \tilde \chi_k \|^2 + D_{\tilde\chi_k,k}(Y_k) \right].
  \end{align*}
\end{lemma}
\begin{proof}
  As before, we can through
  \eqref{eq:mode-time-rescaling}  assume that \(|k| = 1\). We further omit  the subscript $k$ from our notations.  The localized \(L^2\) estimate yields
  \begin{multline*}
	\frac 12 \ddt \| JY \chi \|^2
	+ \nu \| \nabla JY \chi \|^2
	+ \nu^{\frac 12}
	\| JY \chi \|^2
	- \l JY \chi ,  J F \chi  \r\\
	\le
	2 \nu \| \nabla JY \chi \|
	\| JY \nabla \chi \|
	+ \Re\l
	JY \chi, RY \chi
	\r.
  \end{multline*}
Thanks to the definition of $J$ and  Remark \ref{rem:alphabeta}, we find that
 \begin{equation*}
   \| JY \nabla \chi \|^2  \lesssim   |\alpha|^2  \| \na Y  \nabla \chi \|^2 + \nu^{-1} |\beta|^2  \|\na (p \cdot k) Y \na \chi \|^2  \lesssim  \| \na Y \tilde\chi \|^2 + \nu^{-1} b \| \na (p \cdot k) Y \tilde \chi\|^2.    \end{equation*}
It follows that the first term at the right-hand side satisfies for some $M > 0$:
\[  2 \nu \| \nabla JY \chi \|
      \| JY \nabla \chi \| \le \frac{1}{8}  \nu \| \nabla JY \chi \|^2 +  M  \left(\nu  \| \na Y \tilde\chi \|^2 +  b \| \na (p \cdot k) Y \tilde \chi\|^2 \right) .  \]
  For the second term, we integrate by parts the first term at the right-hand side of \eqref{eq:j:commutator-error}, and find
   \begin{equation*}
   \begin{aligned}
    \Re\l
    JY  \chi, RY \chi
    \r
  &  \lesssim
    |\beta| \nu^{\frac12}    \| \nabla J Y \chi \|     \| (p \cdot k - 1) Y \chi\|  +    |\beta| \nu^{\frac12}   \|  J Y \chi \|     \| (p \cdot k - 1) Y \na \chi\|  \\
& \qquad      + \nu |\alpha| \| JY \chi\|  \|\na Y \chi\| +   |\beta| \nu^{\frac12}  \|  J Y \chi \|  \| \na (p \cdot k) Y \chi\|   \\
& \lesssim  \left(
      \| \nabla J Y \chi \|
      + \| J Y \chi \| \right) \left(   b^{\frac12} \nu^{\frac12}    \| \na (p \cdot k)Y \tilde\chi\| + \nu    \|\na Y \chi\|  \right) .
\end{aligned}
 \end{equation*}
  Note that besides Remark \ref{rem:alphabeta},  we have used crucially that $|(p \cdot k - 1)| \lesssim |\na(p \cdot k)|$ on the support of $\tilde\chi$, a property that was missing in the proof of \cref{thm:j:basic-l2-withoutcutoff}.  One can conclude using Young's inequality  that for some $M > 0$:
  \begin{equation*}
   \begin{aligned}
    \Re\l
    JY  \chi, RY \chi
    \r
  &  \lesssim   \frac{1}{8}  \nu \| \nabla JY \chi \|^2 +  \frac{1}{8}  \nu \|  JY \chi \|^2 + M \left(  \nu \| \na Y \tilde\chi \|^2 +   b \| \na (p \cdot k) Y \tilde \chi\|^2 \right) .
\end{aligned}
 \end{equation*}
The lemma follows.
\end{proof}

\subsection{Hypocoercive estimates} \label{subsec:hypocoercive}
To control \(JY\) in the hypocoercivity functional, we first show a
non-localized bound, where we do not benefit from the localization and
thus lose in terms of scaling in \(\nu\).

\begin{lemma}\label{thm:j:noncutoff-dissipation}
  Assume the setup of \cref{thm:tensor:basic-dissipation}. For
  possibly smaller \(\nu_0,B_0\) and a new constant \(M\) we find that
  \begin{align*}
      &\frac 12 \ddt E_{k}(J_kY_k)
        + \frac 58 |k|\, D_{k}(J_kY_k)
        + \frac 34 |k|\, \left(\frac{\nu}{|k|}\right)^{\frac12}
        E_{k}(J_k Y_k)
        - |k| \Re E_{k}(J_k Y_k, JF_k) \\
        &\qquad\le |k| M \left(\frac{\nu}{|k|}\right)^{-\frac12}  \left[ D_k(Y_k) +  \left(\frac{\nu}{|k|}\right)^{\frac12} \frac{c^2_k}{b_k}  \| Y_k \|^2   \right] .
  \end{align*}
\end{lemma}
\begin{proof}
  By the change of variables \eqref{eq:mode-time-rescaling}, we can
  restrict to \(|k| = 1\). We omit as before the subscript $k$.  Then,
  \eqref{eq:j:evolution} and \cref{thm:tensor:basic-dissipation} imply
  (taking into account the extra term $(1-\ii) \nu^{\frac12} JY$)
  \begin{equation*}
    \begin{multlined}
      \frac 12 \ddt E(J Y)
      + \frac 34 \, D(J Y)
      + \nu^{\frac12} E(JY)
      -  \Re E(JY, JF)     \le \Re E(RY,JY) + M c \nu^{\frac12} \| JY \|^2.
    \end{multlined}
  \end{equation*}
  The last term on the right-hand side can directly be absorbed by the third term
  at the left-hand side, as $c \ll 1$.  For \( \Re E(RY,JY)\), first note that
  \begin{equation}\label{thm:j:noncutoff-r-bound}
    |RY|^2 \lesssim  (\nu  |\beta|^2 + \nu^2 |\alpha|^2) |\nabla Y|^2 + \nu  |\beta|^2 |\na (p \cdot k)Y|^2 \lesssim \nu |\nabla Y|^2 +  \nu b   |\na (p \cdot k) Y|^2,
  \end{equation}
  where we used Remark \ref{rem:alphabeta} for the last bound.  Then split \(  \Re E_k(RY,JY)\) as
  \begin{equation*}
    \Re E_{k}(RY,JY) = I^0 + I^a + I^b  + I^c,
  \end{equation*}
  where
  \begin{align*}
    I^0
    &= \Re \l RY, JY \r,\\
    I^a
    &= a \nu^{\frac12} \Re \l \nabla RY, \nabla JY \r,\\
    I^b
    &= b \Re \l \ii \nabla(p\cdot k) RY, \nabla J Y \r
            + b \Re \l  \ii  \nabla(p\cdot k) J Y, \nabla RY \r,
    \\
    I^c
    &= c \nu^{-\frac12}
      \Re \l \nabla (p\cdot k) RY, \nabla (p\cdot k) JY \r.
  \end{align*}
Clearly
  \begin{equation*}
    I^0 \le \| RY \|\, \| JY \|,
  \end{equation*}
  so that for a constant \(\delta>0\) it holds that
  \begin{equation*}
    I^0 - \delta \nu^{\frac12} \| JY \|^2
    \lesssim \delta^{-1} \nu^{-\frac12} \| RY \|^2 \lesssim  \delta^{-1} \nu^{-\frac12}  \left( \nu \| \nabla Y \|^2 + b \nu \| \na (p \cdot k)Y \|^2 \right) \lesssim  \delta^{-1} \nu^{-\frac12}  D(Y).
  \end{equation*}
  For \(I^a\) we find after an integration by parts that
  \begin{equation*}
    I^a - \delta a \nu^{\frac32} \| \na^2 JY \|^2
    \lesssim \delta^{-1} a \nu^{-\frac12} \| RY \|^2 \lesssim  \delta^{-1} \nu^{-\frac12} \| RY \|^2 \lesssim \delta^{-1} \nu^{-\frac12}  D(Y).
  \end{equation*}
  Recalling the choice~\eqref{eq:choice-weights-abc} of the weights, we
    get for \(I^b\) that
  \begin{multline*}
    I^b - \delta \nu^{\frac12} c
    \left[
      \| \nabla( \nabla(p\cdot k) JY) \|^2
      + \| JY \|^2
    \right]\\
    \lesssim \delta^{-1} \frac{b^2}{c} \nu^{-\frac12} \| RY \|^2   \lesssim \delta^{-1}  \nu^{-\frac12} \| RY \|^2     \lesssim  \delta^{-1} \nu^{-\frac12}  D(Y).
  \end{multline*}
Finally, we need to bound \(I^c\).  We find that
  \begin{equation*}
    \begin{split}
      I^c - \delta b \|\na (p \cdot k) J Y\|^2
      &\lesssim  \delta^{-1} \nu^{-1} \frac{c^2}{b} \|\na (p \cdot k) RY\|^2 \\
      & \lesssim  \delta^{-1} \frac{c^2}{b} \left(  \|\na (p \cdot k)\na Y\|^2 +  b \|\na (p \cdot k) Y\|^2\right) \\
           &\lesssim   \delta^{-1}  \nu^{-\frac12} \Big(  \frac{c^2}{b}  \nu^{\frac12}  \|\na (\na (p \cdot k) Y)\|^2 +  \frac{c^2}{b} \nu^{\frac12}  \| Y\|^2 \\
           &\qquad\quad\qquad+   b \|\na (p \cdot k) Y\|^2\Big) \\
              &\lesssim   \delta^{-1}  \nu^{-\frac12}  \left( \frac{c^2}{b} \nu^{\frac12}  \| Y\|^2  + D(Y) \right).
    \end{split}
  \end{equation*}
The lemma follows.
\end{proof}

Using a cutoff, we can obtain a sharp estimate.
\begin{lemma}\label{thm:j:basic-dissipation}
Assume the setup of Lemma \ref{thm:tensor:basic-dissipation} and   $\chi_k, \tilde \chi_k$ as in Lemma \ref{thm:j:basic-l2}. For possible
  smaller \(\nu_0,B_0\) and a new constant \(M\) we find that
  \begin{align*}
      &\frac 12 \ddt E_{\chi_k,k}(J_kY_k)
        + \frac 58 |k|\, D_{\chi_k,k}(J_kY_k)\\
      &\qquad
        + \frac 34 |k|\, \left(\frac{\nu}{|k|}\right)^{\frac12}
        E_{\chi_k,k}(J_kY_k)
        - |k| \Re E_{\chi_k,k}(J_kY_k, JF_k) \\
      &\qquad\quad \le |k| M
        \left[
        \left(\frac{\nu}{|k|}\right) \frac{c_k^2}{b_k} \| Y_k \tilde \chi \|^2 + D_{\tilde\chi_k,k}(Y_k)
+        \left(\frac{\nu}{|k|}\right)^{\frac12} c_k
        \| \nabla (p\cdot \hat k) J_k Y_k \nabla \chi \|^2
        \right].
  \end{align*}
\end{lemma}
\begin{proof}
  Again, we can make the change of variables
  \eqref{eq:mode-time-rescaling} as in the proof of
  \cref{thm:tensor:basic-dissipation}, and assume that \(|k|=1\).
  Starting from \cref{thm:tensor:basic-dissipation}, we find that
  (omitting as before the subscript $k$)
  \begin{equation}\label{eq:start-estimate-j}
    \begin{split}
      &\frac 12 \ddt E_{\chi}(JY)
        + \frac 34 \, D_{\chi}(JY)
        + \, \nu^{\frac12}
        E_{\chi}(JY)
        - \Re E_{\chi}(JY, JF) \\
      &\qquad\le \Re E_{\chi}(RY,JY)
        +  M
        c \nu^{\frac12}
        \| JY \chi \|^2 \\
      &\qquad\quad +  M
        \left[
        \nu \| JY \nabla \chi \|^2
        + \nu^{\frac32} a
        \| \nabla JY \nabla \chi \|^2
        +  \nu^{\frac12} c
        \| \nabla (p\cdot k) JY \nabla \chi \|^2 \right].
    \end{split}
  \end{equation}
  We first treat the term $\Re E_{\chi}(RY,JY)$.  A preliminary remark
  is that $|p \cdot k - 1| \sim |\na (p \cdot k)|^2$ on the support of
  $\tilde \chi$, from where it is easily deduced that, pointwise in
  the support of $\tilde \chi$:
  \begin{equation}
    \label{eq:estimate-j-r}
    \begin{aligned}
      |RY|^2
      &  \lesssim \nu |\beta|^2
      \left[
        |\nabla(\nabla(p\cdot k) Y)|^2
        +
        |\nabla(p\cdot k) Y|^2
      \right]
      + \nu^2 |\alpha|^2 |\nabla Y|^2 \\
      & \lesssim \nu b
      \left[
        |\nabla(\nabla(p\cdot k) Y)|^2
        +
        |\nabla(p\cdot k) Y|^2
      \right]
      + \nu^2  |\nabla Y|^2.
    \end{aligned}
  \end{equation}
  The first term splits as
  \begin{equation*}
    \Re E_{\chi}(RY,JY) = I^0 + I^a + I^b + I^c,
  \end{equation*}
  where
  \begin{align*}
    I^0
    &= \Re \l RY \chi, JY \chi \r,\\
    I^a
    &= a \nu^{\frac12} \Re \l \nabla RY \chi, \nabla JY \chi \r,\\
    I^b
    &= b \Re \l \ii \nabla(p\cdot k) RY\chi, \nabla J Y \chi \r
            + b \Re \l \ii \nabla(p\cdot k) J Y\chi, \nabla RY \chi \r,
    \\
    I^c
    &= c \nu^{-\frac12}
      \Re \l \nabla (p\cdot k) RY \chi, \nabla (p\cdot k) JY \chi \r.
  \end{align*}
  For \(I^0\), after integration by parts of the first term in the
  definition \eqref{eq:j:commutator-error} of $RY$, we get
  \begin{equation*}
    \begin{split}
      |I^0|
      &\le 2 |\beta| \nu^{\frac12}
        \left\{
        \| \nabla  JY \chi \|\, \| (p\cdot k-1) Y \chi \|
        + 2 \| JY \chi \|\, \| (p\cdot k-1) Y \nabla \chi \|
        \right\} \\
      &\quad +
        \nu |\alpha| \| JY \chi \|\, \| [\Delta,\nabla] Y \chi \|
        + \nu^{\frac12} |\beta| \| JY \chi \|\, \| \nabla(p\cdot k) Y \chi \|.
    \end{split}
  \end{equation*}
  Hence we find for a constant \(\delta > 0\) that
  \begin{equation*}
    \begin{multlined}
      |I^0| - \delta \nu \| \nabla JY \chi \|^2
      - \delta \nu^{\frac12} \| JY \chi \|^2\\
      \lesssim \delta^{-1}
      \left\{
        |\beta|^2 \| \nabla(p\cdot k) Y \chi \|^2
        + |\beta|^2 \nu^{\frac12} \| \nabla(p\cdot k) Y \nabla \chi \|^2
        + \nu^{\frac32} \| \nabla Y \chi \|^2
      \right\}  \lesssim  \delta^{-1} D_{\tilde\chi}(Y),
    \end{multlined}
  \end{equation*}
  taking into account that $|\beta|^2 \sim b$ from the
    Definition~\eqref{eq:def-beta} of \(\beta\) and the
    choice~\eqref{eq:choice-weights-abc} for \(b\).  For \(I^a\) note that
  \begin{equation*}
    |I^a|
    \le a \nu^{\frac12}
    \left[
      \| \nabla \nabla JY \chi \|\, \|RY \chi \|
      +  2 \| \nabla JY \chi \|\, \|RY \nabla \chi \|
    \right]
  \end{equation*}
  so that
  \begin{equation*}
    |I^a| - \delta \nu^{\frac32} a \| \nabla \nabla JY \chi \|^2
    - \delta \nu \| \nabla JY \chi \|^2
    \lesssim
    \delta^{-1} \nu^{-\frac12} a \|RY \chi \|^2
    + \delta^{-1} a^2  \| RY \nabla \chi \|^2.
  \end{equation*}
  Hence, from \eqref{eq:estimate-j-r}, we find (using that the
    choice~\eqref{eq:choice-weights-abc} implies $a b \sim c$)
  \begin{equation*}
    |I^a| - \delta \nu^{\frac32} a \| \nabla \nabla JY \chi \|^2
    - \delta \nu \| \nabla JY \chi \|^2 \lesssim \delta^{-1} D_{\tilde\chi}(Y).
\end{equation*}
  For \(I^b\) note that
  \begin{equation*}
    |I^b| \le b \| RY \chi \|
    \Big[
      2 \| \nabla(\nabla(p\cdot k) JY) \chi \|
      + \| JY \chi \|
    \Big]
    + 2b \| RY \nabla \chi \|\,
    \|\nabla(p\cdot k) JY \chi \|,
  \end{equation*}
  so that
  \begin{equation*}
    \begin{multlined}
      |I^b|
      - \delta \nu^{\frac12} c \| \nabla(\nabla(p\cdot k) JY) \chi \|^2
      - \delta \nu^{\frac12} \| JY \chi \|^2
      - \delta b \| \nabla(p\cdot k) JY \chi \|^2\\
      \lesssim
      \delta^{-1} \nu^{-\frac12} \frac{b^2}{c} \| RY \chi \|^2
      + \delta^{-1} b \| RY \nabla \chi \|^2 \lesssim \delta^{-1} D_{\tilde\chi}(Y),
    \end{multlined}
  \end{equation*}
  where the last inequality comes again from \eqref{eq:estimate-j-r}
  (note that $\frac{b^2}{c} \sim a$, $b \sim a^2$, so that the right-hand side is
  similar to the one for $I_a$).  For \(I^c\), the treatment is more
  involved.  Note that
  \begin{equation*}
    \begin{multlined}
      \nabla(p\cdot k) R
      = 2 \nu \nabla_p
      \left[ (p\cdot k-1) (JY - \alpha \nabla Y) \right]
      - 2 \ii \beta \nu^{\frac12} \nabla\nabla(p\cdot k)
      (p\cdot k-1)Y \\
      - \nu \alpha \nabla(p\cdot k) [\Delta,\nabla] Y
      - \ii \nu^{\frac12} \beta \nabla(p\cdot k) \nabla(p\cdot k) Y.
    \end{multlined}
  \end{equation*}
  We remind that on the support of $\tilde \chi$,
  $|p \cdot k - 1| \sim |\na (p \cdot k)|^2$ so that
  \[|(p \cdot k - 1) \na Y| \lesssim |\na (\na(p \cdot k) Y)| + |\na(p \cdot k) Y|.\]
  Hence,  we find that
  \begin{equation*}
    \begin{split}
      |I^c|
      &\lesssim \nu^{\frac12} c \| \nabla(\nabla(p\cdot k) JY) \chi \|
        \left\{
        \| \nabla(p\cdot k) JY \chi \|
        + \| \nabla(\nabla(p\cdot k) Y) \chi \|
        + \| \nabla(p\cdot k)  Y \chi \|
        \right\} \\
      &\quad + \nu^{\frac12} c \| \nabla(p\cdot k) JY \chi \|
        \left\{
        \| \nabla(p\cdot k) JY \nabla \chi \|
        + \| \nabla(\nabla(p\cdot k) Y) \nabla \chi \|
        + \| \nabla(p\cdot k)  Y \nabla \chi \|
        \right\} \\
      &\quad + \nu^{-\frac12} c \| \nabla(p\cdot k) JY \chi \|
        \left\{
        \beta \nu^{\frac12} \| \nabla(p\cdot k) Y \chi \|
        + \nu \| \nabla( \nabla(p\cdot k) Y ) \chi \|
        + \nu \|  Y \chi \|
        \right\}.
    \end{split}
  \end{equation*}
  Hence we can bound \(I^c\) as
  \begin{equation*}
    \begin{split}
      &|I^c|
      - \delta \nu^{\frac12} c \| \nabla(\nabla(p\cdot k) JY) \chi \|^2
      - \delta b \| \nabla(p\cdot k) JY \chi \|^2\\
      &\lesssim
      \delta^{-1} \nu^{\frac12} c
      \left\{
      \| \nabla(p\cdot k) JY \chi \|^2
      + \| \nabla(\nabla(p\cdot k) Y) \chi \|^2
      + \| \nabla(p\cdot k)  Y \chi \|^2
      \right\} \\
      &\quad + \delta^{-1} \nu \frac{c^2}{b}
      \big\{
      \| \nabla(p\cdot k) JY \nabla \chi \|^2
      + \| \nabla(\nabla(p\cdot k) Y) \nabla \chi \|^2
      + \| \nabla(p\cdot k)  Y \nabla \chi \|^2  \big\} \\
          &\quad  + \delta^{-1} \frac{c^2 |\beta|^2}{b}  \| \nabla(p\cdot k) Y \chi \|^2 + \delta^{-1} \nu \frac{c^2}{b}  \big\{ \| \nabla( \nabla(p\cdot k) Y ) \chi \|^2 +  \|Y \chi\|^2    \big\} \\
          & \lesssim  \delta^{-1} \nu^{\frac12} c   \| \nabla(p\cdot k) JY \chi \|^2 + \delta^{-1} D_{\tilde\chi}(Y) + \delta^{-1} \nu \frac{c^2}{b}   \| \nabla(p\cdot k) JY \nabla \chi \|^2    +  \delta^{-1} \nu \frac{c^2}{b}  \|Y \chi\|^2 .
    \end{split}
  \end{equation*}
  The first term at the right-hand side can be absorbed by $D_{\chi}(JY)$ at the
  left-hand side. Regarding the third term, note that by the definition \eqref{eq:generic-viscosity-vector-fields} we
  can estimate directly
  \begin{equation*}
  \begin{aligned}
 \nu \frac{c^2}{b}    \| \nabla(p\cdot k) JY \nabla \chi \|^2
    & \lesssim  \nu \frac{c^2}{b} \big(
    \| \nabla( \nabla(p\cdot k) Y) \nabla \chi \|^2
    + \| Y \nabla \chi \|^2 \big)\\
    &\quad + \frac{c^2}{b}  |\beta|^2 \| \nabla(p\cdot k) Y \nabla \chi \|^2 \\
    & \lesssim D_{\tilde\chi}(Y) +  \nu \frac{c^2}{b}  \|Y \na \chi\|^2 .
    \end{aligned}
  \end{equation*}
  The last step of the proof is to control the second to fourth term in the
  right-hand side of \eqref{eq:start-estimate-j}. The second one can be absorbed by the term
  $\nu^{\frac12} E_{\chi}(JY)$ at the left-hand side, as $c \ll 1$. For the next one,
  we note that by definition \eqref{eq:generic-viscosity-vector-fields}
  \begin{equation*}
    \nu \| JY \nabla \chi \|^2
    \lesssim \nu \| \nabla Y \nabla \chi \|^2
    + b_k \| \nabla(p\cdot k) Y \nabla \chi \|^2 \lesssim D_{\tilde\chi}(Y).
  \end{equation*}
  For the fourth term, again by \eqref{eq:generic-viscosity-vector-fields}
  \begin{equation*}
    \nu^{\frac32} a \| \nabla JY \nabla \chi \|^2
    \lesssim \nu^{\frac32} a \| \nabla \nabla Y \nabla \chi \|^2
    + \nu^{\frac12} a b
    \| \nabla(\nabla(p\cdot k) Y) \nabla \chi \|^2 \lesssim  D_{\tilde\chi}(Y).
  \end{equation*}
 Putting all the above estimates together, we conclude the proof.
\end{proof}

\subsection{The role of convection}\label{sub:roleconvection}
In this section, we now specialize the discussion to the case in which the forcing term \(F_k\)
is of the form
\begin{equation}\label{eq:tensor:convection-term}
  \bV Y_k = - \ii \hat{k} \sum_{\ell} \bv_{k-\ell} Y_{\ell},
\end{equation}
for a divergence-free velocity field \(\bv\).  This is precisely the setting in \eqref{eq:ADEforceU}.
As we do not have a strong enough gain of regularization in \(x\), we cannot neglect the gains
arising from  the divergence-free property. This translates into losing the possibility of working mode-by-mode
in $k$, as all the $x$-modes are coupled, and instead we work with energy functionals that include all the
non-zero $x$-modes.

Therefore, we look at
\begin{equation}\label{eq:definition-summed-energy}
  E_{\chi,s}(Y) = \sum_{k} |k|^{2s} E_{\chi_k,k}(Y_k)
\end{equation}
for a family of cutoffs \((\chi_k)_{k \neq 0}\)  (that we can think of as \(\chi_k=\chi(p\cdot \hat k)\)).
The divergence-free cancellations can be seen through the following lemma.
\begin{lemma}\label{thm:abstract-divergence-cancellation}
  Consider a divergence free velocity field $\bv$, so that  \(\sum_k k\cdot \bv_k = 0\), and
  weights $W : \RR^3 \to \RR_+$ and \(\tilde W : \RR^3 \to X\) for
  some normed space \(X\) satisfying
  \begin{align*}
    & \forall x,y,   \quad
      |y| \ge \frac{|x|}{2} \: \Rightarrow \:   W(x)  \lesssim W(y),
      \quad
      \|\tilde{W}(x)\|_X \lesssim \|\tilde{W}(y)\|_X, \\
    & \forall  x,y,  \quad
      \|\tilde{W}(x) - \tilde{W}(y)\|_X
      \lesssim  \left( \frac{\|\tilde{W}(x)\|_X}{|x|} +  \frac{\|\tilde{W}(y)\|_X}{|y|} \right) |x-y|.
  \end{align*}
  Then for any non-negative sequences $\{H_k\}_{k \in \ZZ^3}, \{G_k\}_{k \in \ZZ^3}$, we have
  \begin{align*}
    &\sum_{k,\ell} W(k) \|\tilde{W}(k) - \tilde{W}(l)\|_X  \,
      |\bv_{k-\ell} \cdot \ell| \,  H_\ell \, G_k \\
    &\qquad\lesssim \Big( \sum_k |k| H_k  \Big)
      \Big(  \sum_k  W(k) \|\tilde{W}(k)\|_X |\bv_k|^2 \Big)^{\frac12}
      \Big(  \sum_k W(k)  \|\tilde{W}(k)\|_X G_k^2 \Big)^{\frac12}   \\
    &\qquad\quad + \Big( \sum_k |k| G_k  \Big)
      \Big(  \sum_k  W(k) \|\tilde{W}(k)\|_X |\bv_k|^2 \Big)^{\frac12}
      \Big(  \sum_k W(k) \|\tilde{W}(k)\|_X H_k^2 \Big)^{\frac12}  \\
    &\qquad\quad + \Big( \sum_k |k| |\bv_k| \Big)
      \Big( \sum_k W(k) \|\tilde{W}(k)\|_X  H_k^2  \Big)^{\frac12}
      \Big( \sum_k W(k) \|\tilde{W}(k)\|_X  G_k^2  \Big)^{\frac12}.
  \end{align*}
\end{lemma}
\begin{proof}
We split the sum in three:
  \begin{equation*}
    \sum_{k,\ell} = \sum_{|k| \ge 2|\ell|} + \sum_{|\ell| \ge 2|k|} + \sum_{ |\ell|/2 \le |k| \le 2 |\ell|}.
  \end{equation*}
  We find
  \begin{align*}
    &\sum_{|k| \ge 2|\ell|} W(k) \|\tilde{W}(k) - \tilde{W}(\ell)\|_X   \,
    |\bv_{k-\ell} \cdot \ell| \,  H_\ell \, G_k\\
    &\qquad\lesssim  \sum_{|k| \ge 2|\ell|}   \,
      \big(\sqrt{W(k-\ell) \|\tilde{W}(k-\ell)\|_X}\, |\bv_{k-\ell}|\big)  \,
      \big(\sqrt{W(k)\|\tilde{W}(k)\|_X}\, G_k\big) \, |\ell H_\ell|  \\
    &\qquad\lesssim \Big( \sum_k |k| H_k  \Big)
      \Big(  \sum_k W(k) \|\tilde{W}(k)\|_X |\bv_k|^2 \Big)^{\frac12}
      \Big(  \sum_k W(k) \|\tilde{W}(k)\|_X G_k^2 \Big)^{\frac12}.
  \end{align*}
  Similarly, as $\bv_{k-\ell} \cdot \ell = \bv_{k-\ell} \cdot k$,
  \begin{align*}
    &\sum_{|\ell| \ge 2|k|} W(k) \|\tilde{W}(k) - \tilde{W}(\ell)\|_X   \,
    |\bv_{k-\ell} \cdot \ell| \,  H_\ell \, G_k\\
    &\qquad\lesssim \sum_{|\ell| \ge 2|k|} W(\ell) \|\tilde{W}(k) - \tilde{W}(\ell)\|_X \,
      |\bv_{k-\ell} \cdot k| \,  H_\ell \, G_k  \\
    &\qquad\lesssim  \Big( \sum_k |k| G_k  \Big)
      \Big(  \sum_k  W(k)\|\tilde{W}(k)\|_X |\bv_k|^2 \Big)^{\frac12}
      \Big(  \sum_k W(k) \|\tilde{W}(k)\|_X H_k^2 \Big)^{\frac12}.
  \end{align*}
  Eventually,
  \begin{align*}
    &\sum_{ |\ell|/2 \le |k| \le 2 |\ell|}  W(k) \|\tilde{W}(k) - \tilde{W}(\ell)\|_X\,
    |\bv_{k-\ell} \cdot \ell| \,  H_k \, G_\ell\\
    &\qquad\lesssim \sum_{ |l|/2 \le |k| \le 2 |l|}  W(k)
      \Big(\frac{\|\tilde{W}(k)\|_X}{|k|}+\frac{\|\tilde{W}(\ell)\|_X}{|\ell|}\Big)
      |k-\ell| |\bv_{k-\ell} \cdot \ell| H_k \, G_\ell \\
    &\qquad\lesssim \sum_{ |\ell|/2 \le |k| \le 2 |\ell|}
      \sqrt{|W(k)| \|\tilde{W}(k)\|_X} \,
      \sqrt{|W(\ell)| \|\tilde{W}(\ell)\|_X} \,
      |k-\ell| |\bv_{k-\ell}|  H_k \, G_\ell \\
    &\qquad\lesssim \Big( \sum_k |k| |\bv_k| \Big)
      \Big( \sum_k W(k) \|\tilde{W}(k)\|_X  H_k^2  \Big)^{\frac12}
      \Big( \sum_k W(k) \|\tilde{W}(k)\|_X  G_k^2  \Big)^{\frac12}.
  \end{align*}
  The proof is over.
\end{proof}

Using the cancellation from the divergence-free condition, we control
the error in the hypoelliptic estimate by the following bound.
\begin{lemma}\label{thm:tensor:basic-convection}
  Assume the setup of  \cref{thm:tensor:basic-dissipation}, and let  $(\chi_k)_{k \neq 0}$ a family of smooth functions.  Then,  the convection operator defined in   \eqref{eq:tensor:convection-term} obeys the following estimate,
 for all \(s > \frac52\), for all $\delta > 0$:
  \begin{align*}
      &\sum_{k} |k|^{2s+1} \Re E_{\chi_k,k}(\bV Y_k,Y_k)
        - \delta \nu^{\frac 12} \sum_k |k|^{2s} \| Y_k \chi_k \|^2
        - \delta \nu \sum_k |k|^{2s} \| \nabla Y_k \chi_k \|^2 \\
      &\quad - \delta \nu^{\frac 12} \sum_k |k|^{2s} b_k^2
        \| \nabla(\nabla(p\cdot \hat k) Y_k) \chi_k \|^2\\
      &\quad
        - (\delta + \| v \|_{H^s} \nu^{-\frac12})
        \sum_k |k|^{2s+1} c_k \| \nabla(p\cdot \hat k) Y_k \chi_k \|^2 \\
      &\qquad \lesssim  \delta^{-1}     \nu^{-1} \| v \|_{H^s}^2  \sum_k |k|^{2s} \| Y_k \|^2
        + \delta^{-1}   \| v \|_{H^s}^2 \sum_k |k|^{2s} \| \nabla Y_k \|^2    .
  \end{align*}
\end{lemma}
\begin{remark}
  Note that the localization through a cut-off is lost in the last two terms. This will  be the main constraint on the size of $v$. Note also that the requirement $s > \frac{5}{2}$ ensures that $\sum_{k \in 2\pi \ZZ^3} |k| |v_k| \lesssim \|v\|_{H^s}$, an inequality that will be used implicitly each time we apply    \cref{thm:abstract-divergence-cancellation}.
  \end{remark}

\begin{proof}[Proof of \cref{thm:tensor:basic-convection}]
We use for short notation $\bV_k$ instead of $\bV  Y_k$.   We decompose
  \begin{equation*}
    \sum_{k} |k|^{2s+1} \Re E_{\chi_k,k}(\bV_k,Y_k)
    = I^0 + I^a + I^b + I^c,
  \end{equation*}
  where
  \begin{align*}
    I^0
    &= \sum_{k} |k|^{2s+1} \Re \l \bV_k \chi_k, Y_k \chi_k \r, \\
    I^a
    &= \nu^{\frac12} \sum_{k} |k|^{2s-\frac 12} a_k
                \Re \l \nabla_p \bV_k \chi_k, \nabla_p Y_k \chi_k \r,\\
    I^b
    &= \sum_{k} |k|^{2s+1} b_k
      \left[
      \Re \l \ii \na(p \cdot \hat{k}) \bV_k \chi_k, \na Y_k \chi_k \r
      + \Re  \l \ii \na(p \cdot \hat{k}) Y_k \chi_k, \na \bV_k \chi_k \r
      \right],\\
    I^c
    &= \nu^{-\frac12} \sum_{k} |k|^{2s+\frac 32} c_k
      \Re \l \nabla(p\cdot \hat k)\bV_k \chi_k, \nabla(p\cdot \hat k)Y_k \chi_k \r.
  \end{align*}
  Using that \(v\) is real, we find from the definition that
  \begin{align*}
    I^0
    &= \sum_{k,\ell} |k|^{2s}
      \Re\l \ii k \bv_{k-\ell} Y_\ell \chi_k, Y_k \chi_k \r \\
    &= \frac 12 \sum_{k,\ell}
      \Re \left\l
      \ii \left(|k|^{2s} k \chi_k^2 - |\ell|^{2s} \ell \chi_\ell^2\right) \cdot
      \bv_{k-\ell} Y_\ell, Y_k
      \right\r \\
      & = \frac 12 \sum_{k,\ell}
      \Re \left\l
      \ii \left(|k|^{2s}  \chi_k^2 - |\ell|^{2s} \chi_\ell^2\right)
      \bv_{k-\ell} \cdot \ell Y_\ell, Y_k
      \right\r ,
  \end{align*}
  where the last equality comes from the divergence-free property: $\bv_{k-\ell} \cdot k =  \bv_{k-\ell} \cdot \ell$.
  We can split the difference as
  \begin{equation*}
    |k|^{2s}  \chi_k^2 - |\ell|^{2s}  \chi_\ell^2
    =
    |k|^{2s}  \chi_k (\chi_k-\chi_\ell)
    + \left( |k|^{2s}  - |\ell|^{2s}  \right) \chi_k \chi_\ell
    + |\ell|^{2s}  \chi_\ell (\chi_k-\chi_\ell).
  \end{equation*}
  Hence, using the symmetry in $k$ and $\ell$,
    \begin{multline*}
    I^0
    \le \sum |k|^{2s} \|\chi_k - \chi_\ell\|_\infty |\bv_{k-\ell} \cdot \ell|
    \| Y_\ell \|\, \| Y_k \chi_k \|\\
    + \frac 12 \sum \left[ |k|^{2s} - |\ell|^{2s} \right] |\bv_{k-\ell} \cdot \ell|
    \| Y_\ell \chi_\ell \|\, \| Y_k \chi_k \|.
  \end{multline*}
We apply  \cref{thm:abstract-divergence-cancellation}, taking  $W = |k|^{2s}$ and $\tilde W(k) = \chi_k$ for the first sum, and  $W = 1$, $\tilde W(k) = |k|^{2s}$ for the second sum. We get
  \begin{equation*}
    I^0
    \lesssim \| v \|_{H^s}
    \left(
      \sum |k|^{2s} \| Y_k \|^2
    \right)^{\frac12}
    \left(
      \sum |k|^{2s} \| Y_k \chi_k \|^2
    \right)^{\frac12}.
  \end{equation*}
  For \(I^a\) we find in the same way (that is replacing $\chi_k(p)$, function of $p$,  by $\chi_k(p) a( \nu^{1}|k|^{\frac12} t)$, function of $p$ and $t$) that
  \begin{align*}
    I^a
    &= \frac{\nu^{\frac12}}{2} \sum
      \Re \l
      \ii \left[
      |k|^{2s-\frac 12}  a_k \chi_k^2
      - |\ell|^{2s-\frac 12}   a_\ell \chi_{\ell}^2
      \right] \bv_{k-\ell} \cdot \ell \nabla Y_\ell, \nabla Y_k
      \r\\
    &\le
      \nu^{\frac12} \sum |k|^{2s-\frac 12}
      \| a_k^{\frac12} \chi_k - a_\ell^{\frac12} \chi_\ell \|_\infty
      |\bv_{k-\ell} \cdot  \ell|
      \| \nabla Y_\ell \|\,
      a_k^{\frac12} \| \nabla Y \chi_k \|\\
    &\qquad +
      \frac{\nu^{\frac12}}{2} \sum
      \left| |k|^{2s-\frac 12} - |\ell|^{2s-\frac 12} \right|
      |\bv_{k-\ell} \cdot \ell|
      a_\ell^{\frac12} \| \nabla Y_\ell \chi_\ell \|\,
      a_k^{\frac12} \| \nabla Y \chi_k \|.
  \end{align*}
  Hence we find that
  \begin{equation*}
    I^a
    \lesssim
    \nu^{\frac12}
    \| v \|_{H^s}
    \left(
      \sum |k|^{2s-\frac 12} \| \nabla Y_k \|^2
    \right)^{\frac12}
    \left(
      \sum |k|^{2s-\frac 12} a_k \| \nabla Y_k \chi_k \|^2
    \right)^{\frac12}.
  \end{equation*}
  For \(I^b\), we find by the definition
  \begin{equation*}
    I^b
    = \sum \Re
    \left\l \ii \left[
        b_k |k|^{2s}  \nabla(p\cdot \hat k) \chi_k^2
        - b_\ell |\ell|^{2s}  \nabla(p\cdot \hat \ell) \chi_\ell^2
        \right] \bv_{k - \ell}  \cdot \ell Y_\ell, \nabla Y_k
    \right\r,
  \end{equation*}
  where we split the bracket as
  \begin{equation*}
    \begin{split}
      &b_k |k|^{2s}  \nabla(p\cdot \hat k) \chi_k^2
        - b_\ell |\ell|^{2s}  \nabla(p\cdot \hat \ell) \chi_\ell^2 \\
      &\qquad= |k|^{2s}  \nabla(p\cdot \hat k) b_k \chi_k (\chi_k-\chi_\ell)
        + \left[
        |k|^{2s}  \nabla(p\cdot \hat k) b_k
        -
        |\ell|^{2s}  \nabla(p\cdot \hat \ell) b_\ell
        \right] \chi_k \chi_\ell\\
       &\qquad\quad + |\ell|^{2s}  \nabla(p\cdot \hat \ell) b_\ell \chi_\ell (\chi_k - \chi_\ell).
    \end{split}
  \end{equation*}
  Hence we find that
  \begin{equation*}
    \begin{split}
      I^b
      &\le
        \sum |k|^{2s} \| \chi_k - \chi_\ell \|_\infty
        |\bv_{k-\ell} \cdot \ell|\,
        \| Y_\ell \|\,
        b_k \| \nabla(p\cdot \hat k) \nabla Y_k \chi_k \| \\
      &\quad +
        \sum
        \| |k|^{2s}  \nabla(p\cdot \hat k) b_k - |\ell|^{2s}  \nabla(p\cdot \hat \ell) b_\ell \|_\infty
        |\bv_{k-\ell} \cdot  \ell|\,
        \| Y_\ell \chi_\ell \|\,
        \| \nabla Y_k \chi_k \| \\
      &\quad +
        \sum |k|^{2s} \| \chi_k - \chi_\ell \|_\infty
        |\bv_{k-\ell} \cdot \ell|\,
        \|  \nabla Y_\ell  \|\,
        b_k  \| \nabla(p\cdot \hat k) Y_k \chi_k \|.
    \end{split}
  \end{equation*}
  This yields that
  \begin{align*}
      I^b
      &\lesssim \| v \|_{H^s}
        \left(
        \sum_k |k|^{2s} \| Y_k \|^2
        \right)^{\frac12}
        \left(
        \sum_k |k|^{2s} b_k^2 \| \nabla(p\cdot \hat k) \nabla Y_k \chi_k\|^2
        \right)^{\frac12}\\
      &\quad + \| v \|_{H^s}
        \left(
        \sum_k |k|^{2s} \| Y_k \chi_k \|^2
        \right)^{\frac12}
        \left(
        \sum_k |k|^{2s} \| \nabla Y_k \chi_k\|^2
        \right)^{\frac12} \\
      &\quad + \| v \|_{H^s}
      \left(
      \sum_k |k|^{2s} \| \nabla Y_k \|^2
      \right)^{\frac12}
      \left(
      \sum_k |k|^{2s} b_k^2 \| \nabla(p\cdot \hat k) Y_k \chi_k\|^2
      \right)^{\frac12}.
  \end{align*}
  For \(I^c\) we find
  \begin{align*}
    I^c = \frac{\nu^{-\frac12}}{2} \sum \Re
    \Big\l
      \ii \Big[& |k|^{2s+\frac 12}  c_k
        \nabla(p\cdot \hat k) \nabla(p\cdot \hat k) \chi_k^2\\
        &-
        |\ell|^{2s+\frac 12}  c_\ell
        \nabla(p\cdot \hat \ell) \nabla(p\cdot \hat \ell)
        \chi_\ell^2
      \Big] \bv_{k-\ell} \cdot \ell Y_\ell, Y_k
    \Big\r.
  \end{align*}
  We split the difference as
  \begin{align*}
      &|k|^{2s+\frac 12} c_k
        \nabla(p\cdot \hat k) \nabla(p\cdot \hat k) \chi_k^2
        -
        |\ell|^{2s+\frac 12}  c_\ell
        \nabla(p\cdot \hat \ell) \nabla(p\cdot \hat \ell)
        \chi_\ell^2 \\
      &\qquad=
        |k|^{2s-\frac 12}  c_k^{\frac12} |k|^{\frac 12} \nabla(p\cdot \hat k)
        \left[
        c_k^{\frac12} |k|^{\frac 12} \nabla(p\cdot \hat k) \chi_k
        -
        c_\ell^{\frac12} |\ell|^{\frac 12} \nabla(p\cdot \hat \ell)\chi_l
        \right] \chi_k \\
      &\qquad\quad +
        \left[
        |k|^{2s-\frac 12}  - |\ell|^{2s-\frac 12}
        \right]
        c_k^{\frac12} |k|^{\frac 12} \nabla(p\cdot \hat k) \chi_k
        c_\ell^{\frac12} |\ell|^{\frac 12} \nabla(p\cdot \hat \ell)
        \chi_\ell \\
      &\qquad\quad +
        |\ell|^{2s-\frac 12} c_\ell^{\frac12} |\ell|^{\frac 12} \nabla(p\cdot \hat \ell)
        \left[
        c_k^{\frac12} |k|^{\frac 12} \nabla(p\cdot \hat k) \chi_k
        -
        c_\ell^{\frac12} |\ell|^{\frac 12} \nabla(p\cdot \hat \ell) \chi_\ell
        \right] \chi_\ell.
  \end{align*}
  Hence we find that
  \begin{equation*}
    \begin{split}
      I^c
      &\le \nu^{-\frac12} \sum |k|^{2s-\frac 12}
        \left\|
        c_k^{\frac12} |k|^{\frac 12} \nabla(p\cdot \hat k)
        -
        c_\ell^{\frac12} |\ell|^{\frac 12} \nabla(p\cdot \hat \ell)
        \right\|_\infty
        |\bv_{k-\ell} \cdot \ell| \| Y_\ell \|
        c_k^{\frac12} |k|^{\frac 12} \| \nabla(p\cdot \hat k) Y_k \chi_k \| \\
      &\quad + \frac{\nu^{-\frac12}}{2}
        \sum
        \left||k|^{2s-\frac 12} - |\ell|^{2s-\frac 12}\right|
        |\bv_{k-\ell} \cdot \ell|
        c_\ell^{\frac12} |\ell|^{\frac 12} \| \nabla(p\cdot \hat \ell) Y_\ell \chi_\ell \|
        c_k^{\frac12} |k|^{\frac 12} \| \nabla(p\cdot \hat k) Y_k \chi_k \|.
    \end{split}
  \end{equation*}
  Applying \cref{thm:abstract-divergence-cancellation}, we obtain  that
  \begin{equation*}
    \begin{split}
      I^c
      &\lesssim \nu^{-\frac12}\| v \|_{H^s}
        \left(
        \sum |k|^{2s} \|Y_k\|^2
        \right)^{\frac12}
        \left(
        \sum |k|^{2s+1} c_k \|\nabla(p\cdot \hat k) Y_k \chi_k\|^2
        \right)^{\frac12} \\
      &\quad +
        \nu^{-\frac12} \| v \|_{H^s}
        \left(
        \sum |k|^{2s+\frac 12} c_k \|\nabla(p\cdot \hat k) Y_k \chi_k\|^2
        \right).
    \end{split}
  \end{equation*}
  Collecting the estimates, we therefore find
  \begin{equation*}
    \begin{split}
      \sum_{k} &|k|^{2s+1} \Re E_{\chi_k,k}(\bV_k,Y_k)\\
      &\lesssim \| v \|_{H^s}
        \left(
        \sum_{k} |k|^{2s} \| Y_k \|^2
        \right)^{\frac12}
        \left(
        \sum_{k} |k|^{2s} \| Y_k \chi_k \|^2
        \right)^{\frac12} \\
      &\quad +
        \nu^{\frac12}
        \| v \|_{H^s}
        \left(
        \sum_k |k|^{2s-\frac 12} \| \nabla Y_k \|^2
        \right)^{\frac12}
        \left(
        \sum_k |k|^{2s-\frac 12} a_k \| \nabla Y_k \chi_k \|^2
        \right)^{\frac12} \\
      &\quad + \| v \|_{H^s}
        \left(
        \sum_k |k|^{2s} \| Y_k \|^2
        \right)^{\frac12}
        \left(
        \sum_k |k|^{2s} b_k^2 \| \nabla(p\cdot \hat k) \nabla Y_k \chi_k\|^2
        \right)^{\frac12}\\
      &\quad + \| v \|_{H^s}
        \left(
        \sum_k |k|^{2s} \| Y_k \chi_k \|^2
        \right)^{\frac12}
        \left(
        \sum_k |k|^{2s} \| \nabla Y_k \chi_k\|^2
        \right)^{\frac12} \\
      &\quad + \| v \|_{H^s}
        \left(
        \sum_k |k|^{2s} \| \nabla Y_k \|^2
        \right)^{\frac12}
        \left(
        \sum_k |k|^{2s} b_k^2 \| \nabla(p\cdot \hat k) Y_k \chi_k\|^2
        \right)^{\frac12} \\
      &\quad + \nu^{-\frac12}\| v \|_{H^s}
        \left(
        \sum_k |k|^{2s} \|Y_k\|^2
        \right)^{\frac12}
        \left(
        \sum_k |k|^{2s+1} c_k \|\nabla(p\cdot \hat k) Y_k \chi_k\|^2
        \right)^{\frac12} \\
      &\quad + \nu^{-\frac12}
        \| v \|_{H^s}
        \left(
        \sum_k |k|^{2s+\frac 12} c_k \|\nabla(p\cdot \hat k) Y_k \chi_k\|^2
        \right).
    \end{split}
  \end{equation*}
  Splitting the right-hand side with Young's inequality then gives the claimed control.
\end{proof}

Using the cancellation from the divergence-free condition, we control
the error for the energy by the following estimate.
\begin{lemma}\label{thm:tensor:basic-convection-energy}
  Let $(\chi_k = \chi_k(p))_{k \neq 0}$ a family of smooth functions.  Then, the
  convection operator defined in \eqref{eq:tensor:convection-term} obeys the
  following estimate, for all \(s > \frac52\):
  \begin{equation*}
    \sum_k |k|^{2s+1} \Re\l \bV Y\chi_k, Y\chi_k \r
    \lesssim \| v \|_{H^s}
    \left(
      \sum_k |k|^{2s} \| Y_k \|^2
    \right)^{\frac12}
    \left(
      \sum_k |k|^{2s} \| Y_k \chi_k \|^2
    \right)^{\frac12}.
  \end{equation*}
\end{lemma}
\begin{proof}
  See the computation for \(I^0\) in the previous lemma.
\end{proof}

\subsubsection{Commutator with the vector fields}\label{subsub:CommVect}

In order to control \(Jg\) and \(JJg\), we need to understand
the commutator with the convection term, that is we need to understand the
influence of
\begin{equation}\label{eq:advection:commutator-definition}
  SY_k = \sum_\ell  J_k (-\ii \hat k \bv_{k-\ell}) Y_\ell
  - \sum_\ell (-\ii \hat k \bv_{k-\ell}) J_\ell Y_\ell.
\end{equation}
To control \(Jg\) in \(L^2\), we  prove a first
commutator estimate.
\begin{lemma}\label{thm:commutator-convection-energy}
Let  $(\chi_k)_{k \neq 0}$ a family of smooth functions.  We find
  \begin{align*}
      &\sum_k |k|^{2s+1} \Re \l SY_k \chi_k , J_k Y_k \chi_k\r
      -\delta \nu^{\frac 12}
      \sum_k |k|^{2s+\frac 12} \| J_kY_k \chi_k \|^2 \\
      &\qquad\lesssim
      \| v \|_{H^s}
      \left(
        \sum_k |k|^{2s} \| \nabla Y_k \|^2
      \right)^{\frac 12}
      \left(
        \sum_k |k|^{2s} \| J_kY_k \chi_k \|^2
      \right)^{\frac 12}
      \\
      &\qquad\quad+ \delta^{-1} \nu^{-\frac 32}
      \| v \|_{H^{s+\frac 14}}^2
      \sum_k |k|^{2s+\frac 12} \| Y_k \|^2.
  \end{align*}
\end{lemma}
\begin{proof}
  We estimate the \(\alpha\) term in \(SY\) as
  \begin{equation*}
  \begin{aligned}
    &\sum |k|^{2s}
    \Re \l
    (\alpha_k-\alpha_\ell) (-\ii \bv_{k-\ell} \cdot \ell) \nabla Y_\ell
    \chi_k,
    J_kY_k \chi_k
    \r\\
   &\qquad  \lesssim  \sum |k|^{2s}
    \| \alpha_k-\alpha_\ell\|_{\infty} |\bv_{k-\ell} \cdot \ell|  \|\nabla Y_\ell\|
    \|JY_k \chi_k\|  \\
    &\qquad  \lesssim \| v \|_{H^s}
    \left(
      \sum |k|^{2s} \| \nabla Y_k \|^2
    \right)^{\frac 12}
    \left(
      \sum |k|^{2s} \| J_kY_k \chi_k \|^2
    \right)^{\frac 12}.
    \end{aligned}
  \end{equation*}
  For the \(\beta\) term we find
  \begin{align*}
      &\nu^{-\frac 12} \sum |k|^{2s}\l
        (\beta_k |k|^{\frac 12} \nabla(p\cdot \hat k)
        - \beta_\ell |\ell|^{\frac 12} \nabla(p\cdot \hat \ell)  \ii \bv_{k-\ell} \cdot \ell  Y_\ell \chi_k ,
        J_kY_k \chi_k \r  \\
       &\qquad  \lesssim \nu^{-\frac 12} \sum |k|^{2s}
        \|\beta_k |k|^{\frac 12} \nabla(p\cdot \hat k)
        - \beta_\ell |\ell|^{\frac 12} \nabla(p\cdot \hat \ell) \|_\infty
         |\bv_{k-\ell} \cdot \ell| \|Y_\ell\|   \|J_kY_k \chi_k\| \\
      &\qquad\lesssim
        \nu^{-\frac 12} \| v \|_{H^{s+\frac 14}}
        \left(
        |k|^{2s+\frac 12} \| Y_k \|^2
        \right)^{\frac 12}
        \left(
        |k|^{2s+\frac 12} \| J_kY_k \chi_k \|
        \right)^{\frac 12}.
  \end{align*}
  Splitting gives the required estimate.
\end{proof}

\begin{remark}\label{remark:commutator-convection-energy}
By a slight modification,  the \(\beta\) term can also be bounded by
  \begin{equation*}
    \nu^{-\frac 12} \| v \|_{H^{s+\frac 12}}
    \left(
      |k|^{2s+1} \| Y_k \|^2
    \right)^{\frac 12}
    \left(
      |k|^{2s} \| J_kY_k \chi_k \|
    \right)^{\frac 12}.
  \end{equation*}
\end{remark}

For the hypoelliptic functional, we control the commutator by the
following estimate.

\begin{lemma}\label{thm:commutator-convection-hypoelliptic}
Let  $(\chi_k)_{k \neq 0}$ a family of smooth functions. Then,
  \begin{align*}
      &\sum |k|^{2s+1} \Re E_{\chi_k,k}(SY_k, J_kY_k)
        - \delta
        \left[
          \sum |k|^{2s+1} D_{\chi_k,k}(J_kY_k)
          + \nu^{\frac 12}
          \sum |k|^{2s+\frac 12}
          \| J_kY_k \chi_k \|^2
        \right] \\
        &\qquad\lesssim
        \delta^{-1} \| v \|_{H^{s+\frac 12}}^2
        \left[
          \nu^{-1} \sum |k|^{2s+1} \| \nabla Y_k \|^2
          +
          \nu^{-2} \sum |k|^{2s+1} \| Y_k \|^2
        \right].
  \end{align*}
\end{lemma}
\begin{proof}
  From the definition we find that
  \begin{equation*}
    \begin{split}
      &\sum |k|^{2s+1} \Re E_{\chi_k,k}\left(SY_k, J_kY_k \right) \\
      &\quad= \sum |k|^{2s} \Re E_{\chi_k,k}\left(
        (\alpha_k -\alpha_\ell) (-\ii \bv_{k-\ell} \cdot  \ell)
        \nabla Y_\ell,
        J_kY_k
        \right) \\
      &\quad\quad+ \nu^{-\frac 12}
        \sum |k|^{2s} \Re E_{\chi_k,k}\left(
        \ii(\beta_k |k|^{\frac12} \nabla(p\cdot \hat k)-\beta_\ell |l|^{\frac12} \nabla(p\cdot \hat \ell))
        (-\ii \bv_{k-\ell} \cdot \ell)
        Y_\ell,
        J_kY_k
        \right).
    \end{split}
  \end{equation*}
  We split as before the contributions in \(I^0\), \(I^a\), \(I^b\) and
  \(I^c\). For \(I^0\) we find
  \begin{equation*}
    \begin{split}
      I^0
      &= \sum |k|^{2s} \Re
        \l
        (\alpha_k-\alpha_\ell) (-\ii \bv_{k-\ell} \cdot \ell) \nabla Y_\ell
        \chi_k,
        J_kY_k \chi_k
        \r \\
      &\quad + \nu^{-\frac 12} \sum |k|^{2s} \Re
        \l
        (\beta_k |k|^{\frac12}  \nabla(p\cdot \hat k)
        - \beta_\ell |\ell|^{\frac12} \nabla(p\cdot \hat \ell))
        (\bv_{k-\ell} \cdot \ell) Y_\ell \chi_k,
        J_kY_k \chi_k
        \r.
    \end{split}
  \end{equation*}
  The first sum can directly be estimated by
  \cref{thm:abstract-divergence-cancellation} as
  \begin{align*}
    &\sum |k|^{2s} \Re
    \l
    (\alpha_k-\alpha_\ell) (-\ii \bv_{k-\ell} \cdot \ell) \nabla Y_\ell
    \chi_k,
    J_kY_k \chi_k
    \r\\
    &\qquad\lesssim \| v \|_{H^s}
    \left(
      \sum |k|^{2s} \| \nabla Y_k \|^2
    \right)^{\frac 12}
    \left(
      \sum |k|^{2s} \| J_kY_k \chi_k \|^2
    \right)^{\frac 12}.
  \end{align*}
  For the second sum, we find by
  \cref{thm:abstract-divergence-cancellation} that
  \begin{equation*}
    \begin{split}
      &\nu^{-\frac 12} \sum |k|^{2s} \Re
        \l
        (\beta_k |k|^{\frac 12} \nabla(p\cdot \hat k)
        - \beta_\ell |\ell|^{\frac 12} \nabla(p\cdot \hat \ell))
        ( \bv_{k-\ell} \cdot \ell) Y_\ell \chi_k,
        J_kY_k \chi_k
        \r \\
      &\qquad\lesssim
        \nu^{-\frac 12} \| v \|_{H^{s+\frac 14}}
        \left( \sum
        |k|^{2s+\frac 12} \| Y_k \|^2
        \right)^{\frac 12}
        \left( \sum
        |k|^{2s+\frac 12} \| J_kY_k \chi_k \|
        \right)^{\frac 12}.
    \end{split}
  \end{equation*}
  For \(I^a\), we find
  \begin{equation*}
    \begin{split}
      I^a &= \nu^{\frac12}  \sum |k|^{2s-\frac 12} \Re a_k
            \l (\alpha_k-\alpha_\ell) (-\ii \bv_{k-\ell} \cdot \ell) \nabla \nabla
            Y_\ell \chi_k,
            \nabla J_k Y_k \chi_k \r \\
          &\quad+ \sum |k|^{2s-\frac 12} \Re a_k
            \l \nabla (\ii \beta_k |k|^{\frac 12} \nabla(p\cdot \hat k)
            \\
          &\quad\qquad\qquad\qquad\qquad- \ii \beta_\ell |\ell|^{\frac 12} \nabla(p\cdot \hat \ell))
            (-\ii \bv_{k-\ell} \cdot \ell)
            Y_\ell \chi_k, \nabla J_k Y_k \chi_k
            \r.
    \end{split}
  \end{equation*}
  In the first term, we integrate  by parts and then find by
  \cref{thm:abstract-divergence-cancellation} that
  \begin{equation*}
    \begin{split}
      I^a
      &\lesssim
        \nu^{\frac 12} \| v \|_{H^{s-\frac 14}}
        \left(
        \sum |k|^{2s-\frac 12} \| \nabla Y_k \|^2
        \right)^{\frac12}
        \left(
        \sum |k|^{2s-\frac 12} a_k^2 \| \nabla \nabla J_kY_k \chi_k \|^2
        \right)^{\frac12}\\
      &\quad +
        \| v \|_{H^{s}}
        \left(
        \sum |k|^{2s} \left( \| \nabla Y_k \|^2 + \| Y_k \|^2 \right)
        \right)^{\frac12}
        \left(
        \sum |k|^{2s} \| \nabla J_kY_k \chi_k \|^2
        \right)^{\frac12}.
    \end{split}
  \end{equation*}
  For \(I^b\) we find
  \begin{equation*}
    \begin{split}
      I^b
      &= \sum |k|^{2s} \Re
        \l \ii  (\alpha_k-\alpha_\ell) (-\ii
        \bv_{k-\ell} \cdot \ell) \nabla Y_\ell \chi_k, b_k
       \nabla(p\cdot \hat k) \nabla J_kY_k \chi_k
        \r \\
      &\quad +
        \sum |k|^{2s} \Re
        \l \nabla (\alpha_k-\alpha_\ell) (-\ii
        \bv_{k-\ell} \cdot \ell) \nabla Y_\ell \chi_k,
        \ii b_k \nabla(p\cdot \hat k) J_kY_k \chi_k
        \r \\
      &\quad + \nu^{-\frac 12} \sum |k|^{2s+\frac12}  \Re
        \Big\l \ii
        (\ii \beta_k |k|^{\frac 12}  \nabla(p\cdot \hat k)
        - \ii \beta_\ell |\ell|^{\frac 12} \nabla(p\cdot \hat \ell))
        (-\ii \bv_{k-\ell} \cdot \ell) Y_\ell \chi_k,\\
	  &\qquad\qquad\qquad\qquad\qquad b_k |k|^{-\frac12}
      \nabla(p\cdot \hat k)  \nabla J_kY_k \chi_k
        \Big\r \\
      &\quad +
        \nu^{-\frac 12}
        \sum |k|^{2s+\frac12}  \Re
        \Big\l \nabla
        (\ii \beta_k |k|^{\frac 12} \nabla(p\cdot \hat k)
        - \ii \beta_\ell |\ell|^{\frac 12} \nabla(p\cdot \hat \ell))
        (-\ii \bv_{k-\ell} \cdot \ell) Y_\ell \chi_k,\\
	  &\qquad\qquad\qquad\qquad\qquad
        \ii \nabla(p\cdot \hat k) b_k |k|^{-\frac12} J_kY_k \chi_k
        \Big\r.
    \end{split}
  \end{equation*}
  so that we find (after integration by parts of the second and fourth terms)
  \begin{equation*}
    \begin{split}
      I^b
      &\lesssim
        \| v \|_{H^s}
        \left(
        \sum |k|^{2s} \| \nabla Y_k \|^2
        \right)^{\frac 12}
        \left(
        \sum |k|^{2s} b_k^2
        \left(
        \| \nabla(\nabla(p\cdot \hat k) J_kY_k) \chi_k \|^2
        + \| JY_k \chi_k \|^2
        \right)
        \right)^{\frac 12} \\
      &\quad +
        \nu^{-\frac 12}
        \| v \|_{H^{s+\frac 12}}
        \left(
        \sum |k|^{2s+1} \| Y_k \|^2
        \right)^{\frac 12}\\
&\qquad  \qquad  \qquad  \qquad        \times\left(
        \sum |k|^{2s} b_k^2
        \left(
        \| \nabla(\nabla(p\cdot \hat k) JY_k) \chi_k \|^2
        + \| J_kY_k \chi_k \|^2
        \right)
        \right)^{\frac 12} .
    \end{split}
  \end{equation*}
  For \(I^c\) we find
  \begin{equation*}
    \begin{split}
      I^c
      &= \nu^{-\frac 12} \sum |k|^{2s}  \Re
        \l \nabla(p\cdot \hat k)
        (\alpha_k-\alpha_\ell) (-\ii \bv_{k-\ell} \cdot \ell) \nabla Y_\ell
        \chi_k, c_k |k|^{\frac12}
        \nabla(p\cdot \hat k) J_kY_k \chi_k
        \r \\
      &\quad + \nu^{-1} \sum |k|^{2s+\frac12}  \Re
        \Big\l \nabla(p\cdot \hat k)
        (\beta_k |k|^{\frac 12} \nabla(p\cdot \hat k)
        - \beta_\ell |\ell|^{\frac 12} \nabla(p\cdot \hat \ell))
        (-\ii \bv_{k-\ell} \cdot \ell) Y_\ell \chi_k,\\
	  &\qquad\qquad\qquad\qquad\qquad c_k
        \nabla(p\cdot \hat k) J_kY_k \chi_k
        \Big\r.
    \end{split}
  \end{equation*}
  Hence we find that
  \begin{equation*}
    \begin{split}
      I^c
      &\lesssim \nu^{-\frac 12}
        \| v \|_{H^s}
        \left(
        \sum |k|^{2s} \| \nabla Y_k \|^2
        \right)^{\frac12}
        \left(
        \sum |k|^{2s+1} c_k^2 \| \nabla(p\cdot \hat k) J_kY_k \chi_k \|^2
        \right)^{\frac12} \\
      &\quad +
        \nu^{-1}
        \| v \|_{H^{s+\frac 12}}
        \left(
        \sum |k|^{2s+1} \| Y_k \|^2
        \right)^{\frac12}
        \left(
        \sum |k|^{2s+1} c_k^2 \| \nabla(p\cdot \hat k) J_kY_k \chi_k \|^2
        \right)^{\frac12},
    \end{split}
  \end{equation*}
  where the first term comes from the \(\alpha\)-term and the second
  from the \(\beta\)-term.

  The inequality of the lemma can then be established as in the proof of \cref{thm:tensor:basic-convection}, through splitting and Young's inequality.
\end{proof}


To control \(J_kJ_kg\), we need the following estimate.
\begin{lemma}\label{thm:commutator-convection-j-energy}
Let  $(\chi_k)_{k \neq 0}$ a family of smooth functions.
  \begin{align*}
      &\sum_k |k|^{2s+1} \Re\l J_kSY_k \chi_k, J_kJ_kY_k \chi_k \r
      - \delta \nu^{\frac 12}
      \sum_k |k|^{2s+\frac 12} \| J_kJ_kY_k \chi_k \|^2\\
      &\qquad- \delta \nu
      \sum_k |k|^{2s} \| \nabla J_kJ_kY_k \chi_k \|^2 \\
      &\qquad\quad\lesssim
      \nu^{-\frac 52} \| v \|_{H^{s+\frac 34}}^2
      \sum_k |k|^{2s+\frac 32} \| Y_k \|^2
      + \nu^{-\frac 32} \| v \|_{H^{s+\frac 14}}^2
      \sum_k |k|^{2s+\frac 12} \| \nabla Y_k \|^2.
  \end{align*}
\end{lemma}
\begin{proof}
  Expanding the definition of \(J_kSY_k\) we find
  \begin{align*}
    |k| J_kSY_k
    &= \sum_\ell \alpha_k (\alpha_k-\alpha_\ell)
      (-\ii \bv_{k-\ell} \cdot \ell)
      \nabla \nabla Y_\ell \\
    &\quad + \nu^{-\frac 12} \sum_\ell
      \alpha_k \nabla \Big(
      (\ii \beta_k |k|^{\frac 12} \nabla(p\cdot \hat k)
      - \ii \beta_\ell |\ell|^{\frac 12} \nabla(p\cdot \hat \ell))
      (-\ii \bv_{k-\ell} \cdot \ell)
      Y_\ell \Big) \\
    &\quad + \nu^{-\frac 12} \sum_\ell
      \ii \beta_k |k|^{\frac 12} \nabla(p\cdot \hat k)
      (\alpha_k-\alpha_\ell)
      (-\ii \bv_{k-\ell} \cdot \ell)
      \nabla Y_\ell \\
    &\quad + \nu^{-1} \sum_\ell
      \ii \beta_k |k|^{\frac 12} \nabla(p\cdot \hat k)
      (\ii \beta_k |k|^{\frac 12} \nabla(p\cdot \hat k)
      - \ii \beta_\ell |\ell|^{\frac 12} \nabla(p\cdot \hat \ell))
      (-\ii \bv_{k-\ell} \cdot \ell)
      Y_\ell.
  \end{align*}
  The contribution of the first term can be bounded, after integration by parts,  by
  \begin{equation*}
    \begin{aligned}
      & \sum |k|^{2s}  \|\alpha_k-\alpha_\ell\|_\infty
        |\bv_{k-\ell} \cdot \ell| \, \| \nabla Y_\ell \| \,
        \big(\|\na J_kJ_kY_k \chi^2_k\| + 2 \|J_kJ_kY_k \chi_k |\na \chi_k|\|\big)
      \\
      & \qquad     \lesssim
        \| v \|_{H^s}
        \left( \sum_k |k|^{2s} \| \nabla Y_k \|^2 \right)^{\frac 12}
        \left( \sum_k |k|^{2s} (\| \nabla J_kJ_kY_k \chi_k\|^2
        + \| J_kJ_kY_k \chi_k\|^2 )\right)^{\frac 12}.
    \end{aligned}
  \end{equation*}
  The contribution of the second term can be bounded by
  \begin{equation*}
    \begin{aligned}
      &    \sum |k|^{2s}    \nu^{-\frac 12}
        \|\beta_k |k|^{\frac 12} \nabla(p\cdot \hat k)
        -  \beta_\ell |\ell|^{\frac 12}\nabla(p\cdot \hat \ell) \|_\infty
        \,|\bv_{k-\ell} \cdot \ell|\,
        \|\na Y_\ell\| \, \| J_kJ_kY_k \chi^2_k\| \\
       &\qquad+  \sum |k|^{2s}    \nu^{-\frac 12}
          \|\beta_k |k|^{\frac 12} \nabla^2(p\cdot \hat k)
          -  \beta_\ell |\ell|^{\frac 12}\nabla^2(p\cdot \hat \ell) \|_\infty
          \,|\bv_{k-\ell} \cdot \ell|\,
          \|Y_\ell\| \, \| J_kJ_kY_k \chi^2_k\| \\
      &\qquad \quad   \lesssim
        \nu^{-\frac 12}
        \| v \|_{H^{s+\frac 14}}
        \left( \sum_k |k|^{2s+\frac 12}
        (\| \na Y_k \|^2 +   \| Y_k \|^2)  \right)^{\frac 12}
        \left( \sum_k |k|^{2s+\frac 32}
        \|J_kJ_kY_k \chi_k \|^2 \right)^{\frac 12}.
    \end{aligned}
  \end{equation*}
  The contribution of the third term can be bounded by
  \begin{equation*}
    \begin{aligned}
      &    \sum |k|^{2s+\frac12}    \nu^{-\frac 12}
        \|\alpha_k-\alpha_\ell\|_\infty \, |\bv_{k-\ell} \cdot \ell|\,
        \|\nabla Y_\ell\| \, \|J_kJ_kY_k \chi^2_k\| \\
      & \qquad  \lesssim
        \nu^{-\frac 12}
        \| v \|_{H^{s+\frac 14}}
        \left( \sum_k |k|^{2s+\frac 12}
        \| \nabla Y_k \|^2 \right)^{\frac 12}
        \left( \sum_k |k|^{2s+\frac 12}
        \| J_kJ_kY_k \chi_k \|^2 \right)^{\frac 12}.
    \end{aligned}
  \end{equation*}
  The last term can be bounded by
  \begin{equation*}
    \begin{aligned}
      &   \sum |k|^{2s+1}    \nu^{-1}
        \| \beta_k |k|^{\frac 12} \nabla(p\cdot \hat k)
        - \beta_\ell |\ell|^{\frac 12} \nabla(p\cdot \hat \ell) \|_\infty
        \, |\bv_{k-\ell} \cdot \ell|\,
        \|Y_\ell\|\, |k|^{-\frac12}\|J_kJ_kY_k \chi_k\|
      \\
      & \qquad \lesssim
        \nu^{-1}
        \| v \|_{H^{s+\frac 34}}
        \left(
        \sum_k |k|^{2s+\frac 32}
        \| Y_k \|^2
        \right)^{\frac 12}
        \left(
        \sum_k |k|^{2s+\frac 12} \| J_kJ_kY_k \chi_k \|^2
        \right)^{\frac 12}.
    \end{aligned}
  \end{equation*}
  The result follows from all these bounds and Young's inequality.
\end{proof}

\subsubsection{Application to the density $g$}\label{subsub:densityg}
Under suitable assumptions on the velocity field \(v\), we can now obtain
controls on \(g\) and its vector fields. Notice that we request something less
stringent than \eqref{BA}. We begin from \(g\).

\begin{proposition}\label{thm:density:control-convection}
There exists $\eps_0>0$ with the following property: if $\eps\in (0,\eps_0)$  and
\[ \sup_{t \ge 0} \|v(t)\|_{H^s} + \Big(  \int_0^{+\infty} \|v(t)\|_{H^s}^2 \dd t \Big)^{\frac12} \le \eps \nu^{\frac12} \]
then for any $T>0$ there holds
  \begin{align*}
    \sup_{0\le t \le T} \sum_k |k|^{2s} E_k(g_k)  &+  \nu^{\frac12} \int_0^T \sum_k |k|^{2s+\frac12} a_k \|g_k \|^2\\
    &+ \int_0^T \sum_k |k|^{2s+1} D_k(g_k)  \lesssim
   \| g^\init\|_{H^s_xL^2_p}^2.
  \end{align*}
  As a consequence, \eqref{eq:EnhDissipg} holds.
\end{proposition}
\begin{proof}
  We start from the inequality in \cref{thm:density:basic-dissipation}: we take
  $F_k = \bV g_k$, multiply by $|k|^{2s}$ and sum over $k$. To bound the term
  $\sum |k|^{2s} \Re E_k(g_k, \bV g_k)$, we apply \cref{thm:tensor:basic-convection}
  with $Y_k = g_k$, $\chi=1$.  We find
  \begin{align*}
    &\frac{\dd}{\dd t} \sum_k |k|^{2s} E_k(g_k)  + \nu^{\frac12} \sum_k |k|^{2s+\frac12} a_k \| g_k \|^2
      + \sum_k |k|^{2s+1} D_k(g_k)  \\
    &\quad       \lesssim ( \delta \nu^{\frac 12} +  \delta^{-1}   \nu^{-1} \| \bv \|_{H^s}^2)  \sum |k|^{2s} \| g_k  \|^2
      +  (\delta \nu  +  \delta^{-1} \| \bv \|_{H^s}^2)  \sum |k|^{2s} \| \nabla g_k  \|^2 \\
    &\qquad +  \delta \nu^{\frac 12} \sum |k|^{2s} b_k^2
      \| \nabla(\nabla(p\cdot \hat k) g_k) \|^2
      + (\delta + \| \bv \|_{H^s} \nu^{-\frac12} ) \sum |k|^{2s+1} c_k \| \nabla(p\cdot \hat k) g_k  \|^2,
  \end{align*}
  for any $\delta > 0$. We take $\delta=\eps$. If $\eps_0\in (0,1)$ is small enough, most
  terms at the right-hand side can be absorbed in the dissipation functional,
  and we end up with
  \begin{equation} \label{equationEk}
    \begin{aligned}
      &\frac{\dd}{\dd t} \sum_k |k|^{2s} E_k(g_k)  + \sum_k |k|^{2s+1} a_k \left(\frac{\nu}{|k|}\right)^{\frac12} \| g_k \|^2
        + \sum_k |k|^{2s+1} D_k(g_k)  \\
      &\qquad       \lesssim \left( \eps \nu^{\frac 12} + \eps^{-1}   \nu^{-1} \| \bv \|_{H^s}^2\right)  \| g  \|^2_{H^s_xL^2_p} .
    \end{aligned}
  \end{equation}
  In particular,
  \begin{equation}
    \frac{\dd}{\dd t}  \| g  \|_{H^s_xL^2_p}
    \lesssim \left( \eps \nu^{\frac 12} + \eps^{-1}   \nu^{-1} \| \bv \|_{H^s}^2\right)   \| g  \|_{H^s_xL^2_p}
  \end{equation}
  so that our assumptions and Gronwall lemma yield
  \[
  \|g(t)\|_{H^s_xL^2_p}    \leq M\e^{M\eps \nu^{\frac12}t}   \| g^\init\|_{H^s_xL^2_p}, \qquad \forall t\geq 0
  \]
  for an absolute constant $M>0$.  If $t\le \nu^{-\frac12}$, injecting
  this bound in the right-hand side of \eqref{equationEk} and
  integrating on $(0,T)$ proves the lemma. If $t \ge \nu^{-\frac12}$,
  we first integrate \eqref{equationEk} on $(0,\nu^{-\frac12})$,
  resulting in
  \[
  \sum_k |k|^{2s} E_k(g_k)\vert_{t = \nu^{-\frac12}} \le   \e^{M \eps}  \| g^\init\|_{H^s_xL^2_p}^2 .
  \]
  Using again the interpolation inequality \eqref{eq:spectralgap} with
  $\sigma = \frac{4\nu a_k }{ |k| c_k}$, we deduce that at time
  $t=\nu^{-\frac12}$ (time for which $\sigma \le 1$ if $\nu$ is small
  enough):
\begin{align*}
E_k(g_k)
&\ge \|g_k\|^2+  \frac12 a_k \nu^{\frac12} |k|^{-\frac12} \|\na_p g_k\|^2 + \frac12c_k \nu^{-\frac12} |k|^{\frac12} \|\na_p(p \cdot \hat{k}) g_k\|^2 \\
&\ge \left(1+ \frac{1}{2} (a_k c_k)^{\frac12}\right) \|g_k\|^2 = \left(1+ \frac{1}{2}(A C)^{\frac12}\right) \|g_k\|^2,
\end{align*}
and eventually
\[
\|g( \nu^{-\frac12})\|_{H^s_xL^2_p}    \le \lambda \| g^\init\|_{H^s_xL^2_p}^2 ,  \qquad \lambda := \frac{    \e^{M \eps_0} }{1+ \frac{1}{2}(A C)^{\frac12}} < 1
\]
taking $\eps<\eps_0$ small enough.
Iterating this bound gives exactly \eqref{eq:EnhDissipg} and concludes the proof.
\end{proof}

\subsubsection{Application to the vector fields}\label{subsub:applicationvector}
The analogous of Proposition \ref{thm:density:control-convection} for vector fields is contained in \cref{thm:j-noncutoff-convection}  (estimate without cut-off) and \cref{thm:j-cutoff-convection} (estimate with cut-off) below. We start with an $L^2$ estimate.
\begin{lemma}\label{thm:j-convection-energy}
 Let  $(\chi_k)_{k \neq 0}$ a family of smooth functions  with $\chi_k = 0$ near $p = -\hat{k}$.
 There exists $\eps_0>0$ with the following property: if $\eps\in (0,\eps_0)$  and
 \[ \sup_{t \ge 0} \|v(t)\|_{H^s} + \left(  \int_0^{\infty} \|v(t)\|_{H^s}^2 \dd t \right)^{\frac12} \le \eps \nu^{\frac34} \]
 then for any $T>0$ there holds
  \begin{align*}
      &\sup_{0\le t\le T} \sum_k |k|^{2s-\frac 12} \| J_kg_k \chi_k \|^2
      + \int_0^T \nu \sum_k |k|^{2s-\frac 12} \| \nabla J_kg_k \chi_k \|^2
      + \nu^{\frac 12} \sum_k |k|^{2s} \| J_kg_k \chi_k \|^2 \\
      &\qquad \qquad \qquad \lesssim
        \| g^\init \|_{H^s_xL^2_p}^2 + \| \nabla g^\init \|_{H^s_xL^2_p}^2 .
  \end{align*}
\end{lemma}
\begin{proof}
We first establish an estimate without cut-off, starting from  \cref{thm:j:basic-l2-withoutcutoff}. We take $Y_k = g_k$, $F_k = \bV  Y_k = \bV  J_k g_k$,  multiply by $|k|^{2s-\frac 12}$ and sum over $k$. We find
\begin{equation}
  \begin{multlined}
	\ddt \sum |k|^{2s-\frac12}\| J_k g_k\|^2
	+  \nu  \sum |k|^{2s-\frac12} \| \nabla J_k g_k\|^2
	+   \nu^{\frac12}  \sum |k|^{2s} \| J_k g_k \|^2 \\
	\lesssim \sum_k |k|^{2s+\frac12} \l J_k g_k , \bV  J_k g_k  \r +  \sum_k |k|^{2s+\frac12} \l J_k g_k, S g_k  \r \\
	+  \sum_{k}  |k|^{2s+\frac12} b_k  \|g_k \|^2 + \sum_{k}  |k|^{2s+\frac12} D_{k}(g_k) .
  \end{multlined}
\end{equation}
The first term at the right-hand side is bounded thanks to \cref{thm:tensor:basic-convection-energy} (applied with $Y = J g$, $\chi=1$):
  \begin{equation*}
    \sum_k |k|^{2s+\frac12} \Re\l \bV  J_k g_k , J_k g_k  \r
    \lesssim \| \bv \|_{H^s}   \sum_k |k|^{2s-\frac12} \| J_k g_k \|^2.
  \end{equation*}
  It can be absorbed by the left-hand side, as  $\|\bv\|_{H^s} \ll \nu^{\frac12}$.
The second term in the right-hand side is estimated thanks to \cref{thm:commutator-convection-energy}  (applied with $Y =  g$, $\chi=1$):
\begin{equation*}
    \begin{aligned}
  &    \sum |k|^{2s+\frac12} \Re \l SJ_k g_k, J_k g_k \r
  - \delta \nu^{\frac 12}   \sum_k |k|^{2s} \| J_k g_k \|^2 \\
    & \qquad   \lesssim
      \| v \|_{H^{s-\frac14}}
      \left(
        \sum_k |k|^{2s-\frac12} \| \nabla g_k \|^2
      \right)^{\frac 12}
      \left(
        \sum_k |k|^{2s-\frac12} \| J_k g_k \|^2
      \right)^{\frac 12} \\
	  &\qquad\qquad
      + \delta^{-1} \nu^{-\frac 32}
      \| v \|_{H^s}^2
      \sum_k |k|^{2s} \| g_k \|^2    \\
  & \qquad    \lesssim \delta \nu^{\frac12}  \sum_k |k|^{2s-\frac12} \| J_k g_k \|^2 + \delta^{-1} \| v \|_{H^{s-\frac14}}^2 \nu^{-\frac12}   \sum_k |k|^{2s-\frac12} \| \nabla g_k \|^2 \\
	  &\qquad\qquad+ \delta^{-1} \nu^{-\frac 32}
      \| v \|_{H^s}^2
      \sum_k |k|^{2s} \| g_k \|^2 .
      \end{aligned}
  \end{equation*}
We take $\delta=\eps$. For $\eps<\eps_0$ small enough, the first  term can  be absorbed in the left-hand side, while the second term can be included in $\sum_{k}  |k|^{2s+\frac12} D_{k}(g_k)$. We end up with
\begin{equation*}
    \begin{multlined}
      \ddt \sum |k|^{2s-\frac12}\| J_k g_k\|^2
      +  \nu  \sum |k|^{2s-\frac12} \| \nabla J_k g_k\|^2
      +   \nu^{\frac12}  \sum |k|^{2s} \| J_k g_k \|^2 \\
\qquad \lesssim        \sum_{k}  |k|^{2s+\frac12} b_k  \|g_k \|^2 + \sum_{k}  |k|^{2s+\frac12} D_{k}(g_k) +      \eps^{-1} \nu^{-\frac 32}
      \| v \|_{H^s}^2    \| g  \|_{H^s_xL^2_p}^2.
    \end{multlined}
\end{equation*}
The time integral of the  second term at the right-hand side. could be controlled directly  thanks to \cref{thm:density:control-convection}. Similarly, the time integral of the last term could be controlled thanks to the assumption on $v$ and \eqref{eq:EnhDissipg}. But  the limiting term, due to the absence of cut-off, is
$ \sum_{k}  |k|^{2s+\frac12} b_k  \|g_k \|^2 \lesssim  \sum_{k}  |k|^{2s+\frac12} a_k  \|g_k \|^2$, where we could not squeeze a factor $\na (p \cdot k)$. This forces us to multiply the previous inequality by $\nu^{\frac12}$. We get, with  \cref{thm:density:control-convection}  that
\begin{equation} \label{estimateJwithoutcutoff2}
    \begin{multlined}
      \sup_{0 \le t \le T} \nu^{\frac12} \sum |k|^{2s-\frac12}\| J_k g_k\|^2
      + \int_0^T   \nu^{\frac32}  \sum |k|^{2s-\frac12} \| \nabla J_k g_k\|^2
      +  \int_0^T   \nu  \sum |k|^{2s} \| J_k g_k \|^2 \\
 \lesssim       \| g^{\init}  \|_{H^s_xL^2_p}^2 + \nu^{\frac12} \| \nabla_p g^{\init}  \|_{H^s_xL^2_p}^2.
    \end{multlined}
\end{equation}
We can then obtain a better bound using the cut-off $\chi$, replacing our starting point \eqref{estimateJwithoutcutoff} by the improved \cref{thm:j:basic-l2}. After similar manipulations, we get
\begin{equation*}
    \begin{aligned}
&      \ddt \sum |k|^{2s-\frac12}\| J_k g_k \chi_k \|^2
      +  \nu  \sum |k|^{2s-\frac12} \| \nabla J_k g_k \chi_k \|^2
      +   \nu^{\frac12}  \sum |k|^{2s} \| J_k g_k \chi_k\|^2 \\
&\qquad \lesssim    \| v \|_{H^{s-\frac14}}
    \left(
      \sum |k|^{2s-\frac12} \| J_k g_k \|^2
    \right)^{\frac12}
    \left(
      \sum |k|^{2s-\frac12} \| J_k g_k \chi_k \|^2   \right)^{\frac12}  \\
   & \qquad \quad +  \| v \|_{H^{s-\frac14}}
      \left(
        \sum_k |k|^{2s-\frac12} \| \nabla g_k \|^2
      \right)^{\frac 12}
      \left(
        \sum_k |k|^{2s-\frac12} \| J_k g_k \chi_k \|^2
      \right)^{\frac 12} \\
      &\qquad\quad
 +   \sum_{k}  |k|^{2s+\frac12} D_{k}(g_k) +    \nu \sum_{k}  |k|^{2s-\frac12} \|g_k \tilde{\chi}_k\|^2     +    \eps^{-1} \nu^{-\frac 32}
      \| v \|_{H^s}^2  \| g  \|_{H^s_xL^2_p}^2
    \end{aligned}
\end{equation*}
which implies, using the smallness assumption on $v$:
\begin{equation*}
    \begin{aligned}
&      \ddt \sum |k|^{2s-\frac12}\| J_k g_k \chi_k \|^2
      +  \nu  \sum |k|^{2s-\frac12} \| \nabla J_k g_k \chi_k \|^2
      +   \nu^{\frac12}  \sum |k|^{2s} \| J_k g_k \chi_k\|^2 \\
&\qquad \lesssim   \nu^{-\frac12} \| v \|_{H^{s-\frac14}}^2
      \sum |k|^{2s-\frac12} \| J_k g_k \|^2
 +   \sum_{k}  |k|^{2s+\frac12} D_{k}(g_k)     \\
 &\qquad\quad+(\nu +   \eps^{-1} \nu^{-\frac 32}
      \| v \|_{H^s}^2 )   \| g  \|_{H^s_xL^2_p}^2  .
    \end{aligned}
\end{equation*}
Using \cref{thm:density:control-convection},  \eqref{estimateJwithoutcutoff2} and  \eqref{eq:EnhDissipg} concludes the proof of the lemma.
\end{proof}


For the control of the hypocoercive functional of \(J_kY\), we first
show an estimate with a loss in \(\nu\), but without localization. This is
\begin{lemma}\label{thm:j-noncutoff-convection}
 There exists $\eps_0>0$ with the following property: if $\eps\in (0,\eps_0)$  and
 \[ \sup_{t \ge 0} \|v(t)\|_{H^s} + \left(  \int_0^{+\infty} \|v(t)\|_{H^s}^2 \dd t \right)^{\frac12} \le \eps \nu^{\frac34} \]
then for any $T>0$ there holds
  \begin{equation*}
    \begin{multlined}
      \nu^{\frac12}
      \left[
        \sup_{0\le t\le T} \sum_k |k|^{2s-1} E_k(J_kg_k)
        + \int_0^T \sum_k |k|^{2s} D_k(J_kg_k)
        + \nu^{\frac 12} \sum_k |k|^{2s-\frac 12} E_k(J_kg_k)
      \right] \\
      \lesssim
      \|g^{\init}, \na_p   g^{\init} \|_{H^s_xL^2_p} ^2.
    \end{multlined}
  \end{equation*}
\end{lemma}
\begin{proof}
  The proof is very close to proof of \eqref{estimateJwithoutcutoff2}. One
  starts from \cref{thm:j:noncutoff-dissipation}, that is applied with
  $Y_k = g_k$, $F_k = \bV g_k$. One multiplies by $|k|^{2s-1}$, and sum over
  $k$. The nonlinear term
  \begin{multline*}
	\sum |k|^{2s-1} \Re E_k(J_k g_k, J_k \bV  g_k) = \\
	\sum |k|^{2s-1} \Re E_k(J_k g_k, \bV  J_k g_k) +
	\sum |k|^{2s-1} \Re E_k(J_k g_k, S g_k)
  \end{multline*}
  is then treated with \cref{thm:tensor:basic-convection} and
  \cref{thm:commutator-convection-hypoelliptic}. We leave the details to the
  reader.
\end{proof}

We then get an improved estimate with cut-off.
\begin{lemma}\label{thm:j-cutoff-convection}
Let  $(\chi_k)_{k \neq 0}$ as in \cref{thm:j-convection-energy}.    There exists $\eps_0>0$ with the following property: if $\eps\in (0,\eps_0)$  and
 \[ \sup_{t \ge 0} \|v(t)\|_{H^s} + \Big(  \int_0^{+\infty} \|v(t)\|_{H^s}^2\dd t \Big)^{\frac12} \le \eps \nu \]
 then for any $T>0$ there holds
  \begin{equation*}
    \begin{multlined}
      \sup_{0\le t\le T} \sum_k |k|^{2s-1} E_{\chi_k,k}(J_kg_k)
      + \int_0^T \sum_k |k|^{2s} D_{\chi_k,k}(J_kg_k)\\
      + \nu^{\frac 12}  \int_0^T \sum_k |k|^{2s-\frac 12} E_{\chi_k,k}(J_kg_k)
      \lesssim
       \|g^{\init}, \na_p   g^{\init} \|_{H^s_xL^2_p}^2.
    \end{multlined}
  \end{equation*}
\end{lemma}
\begin{proof}
The proof is again very close to the one of \cref{thm:j-convection-energy}, once the estimate without cut-off has been established. One needs to combine \cref{thm:j:basic-dissipation}, \cref{thm:j-noncutoff-convection} as well as
  \cref{thm:tensor:basic-convection} and
  \cref{thm:commutator-convection-hypoelliptic} to control the convection term.
\end{proof}

The last control that we will need to prove mixing estimates is an $H^{s-\frac34}$ estimate on   $JJg$.
\begin{lemma}\label{thm:j-j-energy-convection}
Let  $\chi, \chi_k$ as in \cref{thm:j-convection-energy}.   There exists $\eps_0>0$ with the following property: if $\eps\in (0,\eps_0)$  and
 \[ \sup_{t \ge 0} \|v(t)\|_{H^s} + \Big(  \int_0^{+\infty} \|v(t)\|_{H^s}^2\dd t \Big)^{\frac12} \le \eps \nu^{\frac54} \]
 then for any $T>0$ there holds
 \begin{equation*}
    \begin{multlined}
      \sup_{0\le t\le T} \sum_k |k|^{2s-\frac 32} \| J_kJ_kg_k \chi_k \|^2
      + \int_0^T \nu \sum_k |k|^{2s-\frac 32} \| \nabla J_kJ_kg_k \chi_k \|^2\\
      + \nu^{\frac 12} \sum_k |k|^{2s-1} \| J_kJ_kg_k \chi_k \|^2
      \lesssim
 \|g^{\init}, \na_p   g^{\init},\na_p^2    g^{\init}\|_{H^s_xL^2_p}^2.
    \end{multlined}
  \end{equation*}
\end{lemma}
\begin{proof}
The starting point is \cref{thm:j:basic-l2} applied with $Y = J g$, $F = J \bV  g$. We find.
\begin{align*}
      &\frac 12 \ddt \sum |k|^{2s-\frac32} \| J_k J_k g_k \chi_k \|^2
      + \frac 34  \nu \sum |k|^{2s-\frac32} \| \nabla J_k J_k g_k \chi_k \|^2\\
 &\qquad\qquad    +  \frac 34 \nu^{\frac12}  \sum |k|^{2s-1}  \| J_k J_k g \chi_k \|^2\\
&\qquad \lesssim       \sum |k|^{2s-\frac12} |\l J_k Y_k \chi, J_k J_k \bV  g_k \chi \r |
      +  \nu \sum |k|^{2s-\frac32}  \| J_k g_k  \tilde \chi_k \|^2 \\
      &\qquad\quad+  \sum |k|^{2s-\frac12}  D_{\tilde\chi_k,k}(J_k g_k) .
  \end{align*}
The last two terms are controlled thanks to \cref{thm:j-cutoff-convection}, replacing $\chi$ by $\tilde \chi$, after integration in time. The convection term is decomposed into
\begin{align*}
  \sum |k|^{2s-\frac12} \l J_k J_k g_k \chi, J_k J_k \bV  g_k \chi \r & = \sum |k|^{2s-\frac12} \l J_k J_k g_k \chi_k, J_k S  g_k \chi_k \r
  \\
 &\quad + \sum |k|^{2s-\frac12} \l J_k J_k g_k \chi_k, S J_k   g_k \chi \r \\
 &\quad+  \sum |k|^{2s-\frac12} \l J_k J_k g_k \chi_k, \bV  J_k J_k   g_k \chi_k \r .
  \end{align*}
  We use  \cref{thm:tensor:basic-convection-energy} to write
  \begin{align*}
 &    \sum_k |k|^{2s-\frac12} \Re  \l J_k J_k g_k \chi_k, \bV  J_k J_k   g_k \chi_k \r \\
  & \qquad \lesssim \| v \|_{H^{s-\frac34}}
    \left(
      \sum |k|^{2s-\frac32} \| J_k J_k g_k \|^2
    \right)^{\frac12}
    \left(
      \sum |k|^{2s-\frac32}  \|  J_k J_k g_k \chi_k \|^2
    \right)^{\frac12} \\
    &\qquad \lesssim \delta \nu^{\frac12}    \sum |k|^{2s-\frac32} \|  J_k J_k g_k \chi_k \|^2 + \delta^{-1} \| v \|_{H^{s-\frac34}}^2  \nu^{-\frac12}\sum |k|^{2s-\frac32} \| J_k J_k g_k \|^2.
  \end{align*}
Replacing the first operator $J_k$ by its definition, we get
\begin{align*}
 &\| v \|_{H^{s-\frac34}}^2  \nu^{-\frac12}\sum |k|^{2s-\frac32} \| J_k J_k g_k \|^2 \\
 & \qquad \lesssim \| v \|_{H^{s-\frac34}}^2  \nu^{-\frac32} \Big( \nu \sum |k|^{2s-\frac32} \| \na_p  J_k g_k \|^2  +  \sum |k|^{2s-\frac12} b_k  \| \na(p \cdot \hat{k}) J_k g_k \|^2 \Big) .
\end{align*}
The right-hand side can be controlled thanks to  \cref{thm:j-noncutoff-convection}, after integration in time.
The term
\[
\sum |k|^{2s-\frac12} \l J_k J_k g_k \chi_k, S J_k   g_k \chi_k \r
\]
can be controlled in an easier way  with   \cref{thm:commutator-convection-energy}  and  \cref{thm:j-noncutoff-convection}.
Finally, the term
\[
\sum |k|^{2s-\frac12} \l J_k J_k g_k \chi_k, J_k S  g_k \chi_k \r
\]
can be controlled thanks to
  \cref{thm:commutator-convection-j-energy}. This concludes the proof.
\end{proof}

\subsection{Mixing estimates}\label{sec:mixingestimates}
We now consider integral quantities of the form \eqref{defHkVk}, namely
\begin{equation}
  \sfV_k[g] = \int_{\bbS^2} g_k(p) Z_k(p)  \na (p \cdot \hat{k}) \dd p,
\end{equation}
and establish decay estimates  as  a consequence of orientation mixing.
\begin{proposition}\label{thm:mixing-j}
  Let $k \neq 0$, $Z_k = Z_k(p)$ a smooth function, and
  $\chi_k = \chi_k(p)$ a smooth function which is $1$ on the support
  of $Z_k$. Then, for any $g\in H^1$, we have
 \begin{equation}
      \left|  \sfV_k[g]\right|
       \lesssim
        \frac{\nu^{\frac 12}}{|k|^{\frac 12} \beta_k}
        \Big[
        \| J_k g_k \chi_k \|
        +
        \| g_k \chi_k \|
        \Big]
        \| Z_k \|_{H^1}.
  \end{equation}
\end{proposition}
\begin{proof}
  We use the identity
  $\na (p \cdot \hat{k})g_k = \frac{1}{\ii} \frac{\nu^{\frac12}}{|k|^{\frac12} \beta_k} \left( J_k g_k - \alpha_k \na g_k\right)$,
  so that
\begin{align*}
 \left|  \sfV_k[g] \right| &\le  \frac{\nu^{\frac12}}{|k|^{\frac12} |\beta_k|}   \left( \left| \int_{\bbS^2} J_k g_k Z_k(p) \right| + \left|\int_{\bbS^2}  \alpha_k  \na g_k Z_k(p)  \right| \right) \\
 & \le  \frac{\nu^{\frac12}}{|k|^{\frac12} |\beta_k|}  \left( \left| \int_{\bbS^2} J_k g_k   Z_k(p) \right| + \left|\int_{\bbS^2}  \alpha_k   g_k   \na \cdot Z_k(p) \right| \right)\\
 & \lesssim  \frac{\nu^{\frac12}}{|k|^{\frac12} |\beta_k|}  \left(  \int_{\bbS^2} \left| J_k g_k  \chi_k   Z_k(p) \right|  + \int_{\bbS^2}  | g_k \chi_k   \na \cdot Z_k(p) | \right).
\end{align*}
The result follows from Cauchy-Schwarz inequality.
\end{proof}
When $g$ is the solution of the advection-diffusion equation \eqref{transport-diffusion}, suitable norms of the functions $\{\sfV_k\}_{k\in \ZZ}$ can be estimated in terms of the vector field $J$ as follows.
\begin{proposition}\label{prop:conclusion-j-mixing}
  Assume
   \[ \sup_{t \ge 0} \|v(t)\|_{H^s} + \left(  \int_0^{\infty} \|v(t)\|_{H^s}^2 \dd t \right)^{\frac12} \le \eps \nu^{\frac34}. \]
 Then, for $\eps < \eps_0$ small enough, we find
  \begin{equation*}
    \sum_k |k|^{2s+\frac 12} | \sfV_k[g(t)]|^2
    \lesssim
    \frac{\nu}{|\beta(\nu^{\frac 12} t)|^2}
    \sum_k |k|^{2s} \left(\| g^\init_k \|^2 + \| \nabla g^\init_k \|^2\right)  \left( \sup_\ell  \| Z_\ell \|_{H^1}^2 \right) .
  \end{equation*}
\end{proposition}
\begin{proof}
  For all $k \neq 0$, we introduce $\chi_k, \psi_k$ a smooth partition
  of unity on the sphere (\(\chi_k+\psi_k\equiv 1\)) such that
  $\chi_k = 0$ near $p = -\hat{k}$ and $ \psi_k= 0$ near
  $p = \hat{k}$.  We can decompose
 \[  \sfV_k[g(t)] = \int_{\bbS^2} g_k(t) (Z_k \chi_k)(p) \na (p \cdot \hat{k}) + \int_{\bbS^2} g_k(t) (Z_k \psi_k)(p) \na (p \cdot \hat{k}) .\]
 It is enough to prove that
 \begin{multline*}
   \sum_k |k|^{2s+\frac 12} \Big|  \int_{\bbS^2} g_k(t) (Z_k \chi_k)(p) \na (p \cdot \hat{k})\Big|^2  \lesssim\\
   \frac{\nu}{\beta(\nu^{\frac 12} t)^2}
   \sum_k |k|^{2s} (\| g^\init_k \|^2 + \| \nabla g^\init_k \|^2)  \left( \sup_\ell  \| Z_\ell \|_{H^1}^2 \right) .
 \end{multline*}
 as the other integral could be treated similarly, by exchanging the roles of
 the north and south poles, see Remark \ref{rem:explanationJ}. We then introduce
 another family $(\tilde \chi_k)_{k \neq 0}$ which is still zero near $p = -\hat{k}$,
 with $\tilde \chi_k = 1$ on the support of $\chi_k$. It follows from
 \cref{thm:mixing-j} that
\begin{align*}
& \sum_k |k|^{2s+\frac 12} \Big|  \int_{\bbS^2} g_k(t) (Z_k \chi_k)(p) \na (p \cdot \hat{k})\Big|^2  \\
& \qquad\lesssim  \sum_k |k|^{2s-\frac 12}    \frac{\nu}{|\beta_k|^2}
        \left[
        \| J_k g_k(t) \tilde \chi_k \|^2
        +
        \| g_k(t) \tilde \chi_k \|^2
        \right]
        \| Z_k \|_{H^1}^2 \\
       &\qquad \lesssim \frac{\nu}{|\beta(\nu^{\frac 12} t)|^2}  \sum_k |k|^{2s-\frac 12}    \left[
        \| J_k g_k(t) \chi_k \|^2  +    \| g_k(t) \chi_k \|^2
        \right]  \left( \sup_\ell  \| Z_\ell \|_{H^1}^2 \right).
\end{align*}
The result then follows from \cref{thm:density:control-convection} and  \cref{thm:j-convection-energy}.
\end{proof}

When dealing with \(JJg\), we follow the same spirit as in
\cite[Proposition~1.7]{CZDGV22} and obtain the following bound.

\begin{proposition} \label{prop_mixing1}
Let $k \neq 0$;  $Z_k = Z_k(p)$ a smooth function, and  $\chi_k = \chi_k(p)$ a smooth function  which is $1$ on the support of $Z_k$. Then,  for all  $r = r_k(t)$, the following bound holds
  \begin{equation*}
    | \sfV_k[g]|
    \lesssim
    \Big( A_{r,k} \|g_k \chi_k\| + B_{r,k} \|J_k g_k \chi_k\| + C_{r,k}  \|J_k J_k g_k \chi_k\| \Big) \big( \|Z_k\|_{H^2} +  \|Z_k\|_{W^{1,\infty}} \big),
  \end{equation*}
  where
  \begin{align*}
    A_{r,k}
    & :=
     \frac{\nu^{\frac 12}}{|k|^{\frac 12} |\beta_k| }
      \left(
      r+ r^{-1}
      \frac{\nu^{\frac 12}}{|k|^{\frac 12}  |\beta_k|  }
      \right),\\
    B_{r,k}
    & :=
      \frac{\nu^{\frac 12}}{|k|^{\frac 12}  |\beta_k|  }
      \left(
      r + (r^{-1} + \lvert \ln r \rvert^{\frac12})
      \frac{\nu^{\frac 12}}{|k|^{\frac 12}  |\beta_k|  }
      \right),\\
    C_{r,k}
    & := \frac{\nu}{|k|  |\beta_k| ^2 } \lvert \ln(r) \rvert^{\frac12}.
  \end{align*}
\end{proposition}

\begin{proof}
As a preliminary step, we introduce $\chi_r = \chi_r(p)$ a smooth function with $\chi_r = 1$ on an $r$-neighborhood of $\hat{k}$, and $\chi_r = 0$ outside a $2r$-neighborhood of $\hat{k}$. We then consider, for any smooth $G$ with $G=0$ near $p = -\hat{k}$ and $\chi_k = 1$ on the support of $G$, and for any tensor $Y_k$:
  \begin{equation*}
    \left|   \int_{\bbS^2} Y_k G  \right| \le \left|   \int_{\bbS^2} Y_k \chi_r G  \right| +  \left|   \int_{\bbS^2} Y_k  (1-\chi_r) G  \right| =:  I_1 + I_2.
  \end{equation*}
  We find directly by Cauchy-Schwarz
  \begin{equation*}
    I_1 \lesssim \|Y_k  \chi_k\| \, r \|G\|_{L^\infty}.
  \end{equation*}
  Then,
  \begin{align*}
    I_2 & = \frac{1}{\ii\beta_k} \left(\frac{\nu}{|k|}\right)^{\frac12}  \int_{\bbS^2} \na (p \cdot \hat{k})  \cdot  J_k Y_k  \frac{G}{|\na_p (p \cdot \hat{k})|^2} (1-\chi_r)  \\
    & \quad+  \frac{\alpha_k}{\ii\beta_k} \left(\frac{\nu}{|k|}\right)^{\frac12}  \int_{\bbS^2}   Y_k  \na_p \cdot \left( \na (p \cdot \hat{k})  \frac{G}{|\na_p (p \cdot \hat{k})|^2} (1-\chi_r) \right),
  \end{align*}
  from where
  \begin{align*}
    I_2 & \lesssim \frac{1}{|\beta_k|}  \left(\frac{\nu}{|k|}\right)^{\frac12}  \|G\|_{L^\infty} \|J_k Y_k  \chi_k\| \lvert \ln(r) \rvert^{\frac12}
      +  \frac{|\alpha_k|}{|\beta_k|} \left(\frac{\nu}{|k|}\right)^{\frac12} \|Y_k  \chi_k\| \frac{1}{r} \left( \|G\|_{L^\infty} + \|G\|_{H^1} \right).
  \end{align*}
  We conclude that
  \begin{equation}
    \begin{aligned}
      \left|   \int_{\bbS^2} Y_k G  \right| &\lesssim \left( \left(  \frac{1}{|\beta_k|} \left(\frac{\nu}{|k|}\right)^{\frac12} \frac{1}{r} + r \right) \|Y_k  \chi_k \| +  \frac{1}{|\beta_k|}  \left(\frac{\nu}{|k|}\right)^{\frac12} \lvert \ln(r) \rvert^{\frac12} \|J_k Y_k  \chi_k\|  \right)  \\
      &\qquad\times \left( \|G\|_{L^\infty} + \|G\|_{H^1} \right).
    \end{aligned}
  \end{equation}
  We then turn to
  \begin{align*}
     | \sfV[g_k(t)]|  \le \left| \frac{1}{\ii \beta_k}  \left(\frac{\nu}{|k|}\right)^{\frac12} \int_{\bbS^2} Z_k J_k g_k  \right| + \left| \frac{\alpha_k}{\ii \beta_k}  \left(\frac{\nu}{|k|}\right)^{\frac12}   \int_{\bbS^2} (\na_p \cdot Z_k) g_k \right|.
  \end{align*}
  Applying the previous formula, we end up with
  \begin{equation}
    \begin{aligned}
     | \sfV[g_k(t)]|
      & \lesssim \frac{1}{|\beta_k|} \left(\frac{\nu}{|k|}\right)^{\frac12}  \left( \Big(  \frac{1}{|\beta_k|} \left(\frac{\nu}{|k|}\right)^{\frac12} \frac{1}{r} + r \Big) \|J_k g_k \chi_k\| \right. \\
      &\quad\left. +  \frac{1}{|\beta|}  \left(\frac{\nu}{|k|}\right)^{\frac12} \lvert \ln(r) \rvert|^{\frac12} \|J_k J_k g_k \chi_k\|  \right) (\|Z_k\|_{L^\infty} + \|Z_k\|_{H^1}) \\
      & \quad+  \frac{1}{|\beta_k|}  \left(\frac{\nu}{|k|}\right)^{\frac12}  \left( \Big(  \frac{1}{|\beta_k|} \left(\frac{\nu}{|k|}\right)^{\frac12} \frac{1}{r} + r \Big) \|g_k \chi_k\| \right.\\
      &\quad \left. +  \frac{1}{|\beta_k|}  \left(\frac{\nu}{|k|}\right)^{\frac12} \lvert \ln(r) \rvert^{\frac12} \|J_k g_k \chi_k\|  \right)   \left( \|\na Z_k\|_{L^\infty} + \|\na Z_k\|_{H^1} \right).
    \end{aligned}
  \end{equation}
  The result follows.
\end{proof}
We can now conclude this linear analysis, proving \cref{thm:linmain}.

\begin{proof}[Proof of \cref{thm:linmain}]
  For $t \le 1$, \cref{prop:conclusion-j-mixing} implies
  \begin{equation*}
    \sum_{k\neq 0} |k|^{2s'} \left|\sfV_k[g_k(t)]\right|^2
    \lesssim t^{-2}
    \sup_{k} \left(  \|Z_k\|_{W_p^{1,\infty}}^2 + \|Z_k\|_{H^2_p}^2\right)
    \|g^{\init}, \na_p   g^{\init},\na_p^2    g^{\init}\|_{H^s_xL^2_p}^2   ,
  \end{equation*}
  so that the claimed conclusion follows as in this time range
  \(t^{-2} \le t^{-3}\).

  In the range \(t \ge 1\), we apply \cref{prop_mixing1} with the
  choice
  \[ r = r_k(t) =  \left(\frac{\nu^{\frac 12}}{|k|^{\frac 12} |\beta_k|}\right)^{\frac 12} \lesssim 1. \]
  Then the bounds satisfy
  \begin{equation*}
    A_{r,k}, B_{r,k}, |k|^{\frac 14}C_{r,k}
    \lesssim  \left(\frac{\nu^{\frac 12}}{|k|^{\frac 12} |\beta_k|}\right)^{\frac 32},
  \end{equation*}
  which implies the claimed result.
\end{proof}

\section{Nonlinear stability} \label{sec4}
The proof of the  nonlinear stability Theorem \ref{thm:nonstab} relies on two main ingredients: the analysis of a Volterra equation that allows
the computation of $u$ from $\psi$ through the relations specified in \eqref{SS3}, and the use of the linear estimates of Theorem \ref{thm:linmain},
via a proper bootstrapping scheme. These two points will be carried out in Sections \ref{sub:volterra} and \ref{sub:boot}, respectively.

\emph{For notational convenience, we will prove Theorem \ref{thm:nonstab} with shifted index $s-1$ instead of $s$. As $u$ has one more degree of regularity than $\psi$ in $x$, $u$ will then have $H^s$ regularity in $x$.}

Let $s > \frac{9}{2}$, $T \in (0,\infty]$, and a field $\bv$ defined
on $[0,T)$ and satisfying
\begin{equation} \label{BA'} \tag{H'}
  \sup_{0 \le t \le  T} \|v(t)\|_{H^s} + \Big(  \int_0^{T} \|v(t)\|_{H^s}^2 \dd t \Big)^{\frac12} \le \eps \nu^{\frac54}.
\end{equation}
We introduce
\begin{equation}
S_v(t,\tau) : L^2 \rightarrow L^2, \qquad  0 \le  \tau \le t < T
\end{equation}
the linear two-parameter process arising as the solution operator of
the (non-autonomous) advection-diffusion equation in
\eqref{transport-diffusion}, considered on $(0,T)$. Namely, for
$0 \le \tau \le t < T$ we set $S_v(t,\tau) g^{\init} = g_\tau(t)$ where
$g_\tau$ is the solution on $[\tau,T)$ of
\[   \pa_t g_\tau + (v+p) \cdot \na_x g_\tau - \nu \Delta_p g_\tau = 0, \qquad g_\tau|_{t=\tau} =g^{\init}. \]
The results in Theorem \ref{thm:linmain} were stated under \eqref{BA}
for the time interval $(0,\infty)$ but extend straightforwardly under
\eqref{BA'} to an arbitrary interval $(\tau, T)$. In particular, given $s> \frac{5}{2}$,  $0<s'<s+\frac 14$, there
exist  constants $C_0, \eps, \nu_0, \eta_1 > 0$ such that for
all $\nu \le \nu_0$ the condition~\eqref{BA'} implies
\begin{equation} \label{estim_Sv1}
  \| (S_v(t,\tau) g^{\init})_{\neq 0} \|_{H^s_xL^2_{p}}\le   C_0\e^{-\eta_1 \nu^{\frac12}  (t-\tau)}  \| g^{\init}_{\neq 0}  \|_{H^s_xL^2_{p}}
\end{equation}
and
\begin{equation} \label{estim_Sv2}
  \begin{multlined}
  \sum_{k\neq 0} |k|^{2s'} \left|\sfV_k[S_v(t,\tau)
    g^{\init}]\right|^2  \\
  \leq C_0\left(
    \frac{\nu^{\frac12}}{\min\{1,\nu^{\frac12} (t-\tau)\}}\right)^{3}
  \sup_{k} \left(  \|Z_k\|_{W_p^{1,\infty}}^2 + \|Z_k\|_{H^2_p}^2\right)
  \|g^{\init}_{\neq 0}, \na_p   g^{\init}_{\neq 0},\na_p^2    g^{\init}_{\neq 0}\|_{H^s_xL^2_p}^2   ,
\end{multlined}
\end{equation}
 for all $t \in [0, T)$. We will now focus on the analysis of the full system \eqref{SS3}.

\subsection{Bootstrap assumptions} \label{subsec_bootstrap}
Let $\psi^{\init} \in \mH^{s-1}$. Existence and uniqueness of a local in time solution to \eqref{SS3} satisfying
\[  \psi \in C_{loc}([0, T_*), \mH^{s-1}), \quad \na_p \psi  \in L^2_{loc}([0, T_*), \mH^{s-1}) \]
is standard, see e.g.\ \cite[Theorem~B.2]{AO23} for \(s=2\).  Moreover, if $T_*$
is the maximal time of existence, one has
\begin{equation*}
  \limsup_{t \rightarrow T_*} \|u(t)\|_{H^s} = \infty.
\end{equation*}
Let $\eta_0 = \frac{\eta_1}{10}$, with $\eta_1$ the absolute constant
in \eqref{estim_Sv1}. Let $\delta$, $\delta'$, $\delta''$ positive
constants to be specified later, only depending on $\gamma$ and
$\iota$. For each $\nu$, let $T = T(\nu,\gamma, \iota) > 0$ the
maximal time upon which the following three bootstrap assumptions
hold:
\begin{align}
  \label{BA0} \tag{BA0}
  \sup_{0 \le t \le T}
  \e^{\eta_0 \nu^{\frac12} t} \|u(t)\|_{H^{s}} & \le \delta  \nu^{\frac32},\\
  \label{BA1} \tag{BA1}
  \int_{0}^T  \|\na_p \psi(t)\|^2_{\mH^{s-1}}\dd t  & \le \delta'^2 \nu^{\frac12}, \\
  \label{BA2} \tag{BA2}
  \sup_{0 \le t \le T}\|\psi(t)\|_{\mH^{s-1}}  & \le \delta'' \nu^{\frac12}.
\end{align}
For any $\delta, \delta', \delta''$, the existence of a positive $T$
is guaranteed by the smallness assumption on $\psi^{\init}$, taking
$\delta_0$ small enough compared to $\delta, \delta', \delta''$. Note
also that \eqref{BA0} implies \eqref{BA'}, for
$\delta  \le \min\big(\frac{\eps}{\eta_0^{1/4}},1\big)$.  The point is
to show that there exists $\nu_0$ such that for all $\nu \le \nu_0$,
all three bounds are satisfied with improved constants
$\delta/2, \delta'/2, \delta''/2$ instead of
$\delta, \delta', \delta''$, and that moreover, on $(0,T)$:
\begin{equation}  \label{stab_estim_T}
  \sup_{0 \le t < T}  \|\psi\|_{H^s_x L^2_p}^2
  + \nu \int_{0}^T  \|\na_p \psi\|_{H^s_x L^2_p}^2\, \dd t
  \lesssim \nu^{-1} \|\psi^{\init} \|^2.
\end{equation}
Improvement of the constants will imply that $T = T_* = \infty$, which
combined with \eqref{stab_estim_T} will conclude the proof of
Theorem~\ref{thm:nonstab}.

\begin{remark}
  Condition \eqref{BA0} is of enhanced dissipation type, and holds on the
  velocity field $u$ solving the nonlinear system. If we were able to propagate
  instead some mixing estimate of the form
    \[ \sup_{0 \le t \le T} (1+t)^\alpha \|u(t)\|_{H^s} \le \delta \nu^{\frac54}\]
    for some $\alpha > 1$, we could lower the stability threshold \eqref{eq:stabthreshold}  from $\nu^{\frac32}$ down to
    $\nu^{\frac54}$, because it would be enough to ensure that $u$ satisfies \eqref{BA'}. However, deriving such nonlinear mixing estimates, in the spirit of \cite{CLN21}, seems out of reach in our context.
\end{remark}

\subsection{Analysis of the Volterra equation}\label{sub:volterra}
From \eqref{SS3}, we can deduce a Volterra-like equation for $u$. On
one hand, from the second relation in \eqref{SS3}, $u$ is obtained
from $\psi$ through application of the linear operator \(U\) defined
by
\begin{equation*}
  U \psi =   \iota \text{St}^{-1} \na_x \cdot \int_{\mathbb{S}^2}  p \otimes p\, \psi
\end{equation*}
where $\text{St}$ is the Stokes operator $-\mathbb{P} \Delta$ on $\T$
(with $\mathbb{P}$ the Leray projector). On the other hand, the
solution $\psi$ to the first relation in \eqref{SS3} obeys the
Duhamel's formula
\begin{equation*}
  \psi(t) = S_u(t,0) \psi^{\init} + \int_0^t S_u(t,\tau)\, F(\tau)\, \dd\tau
\end{equation*}
with
\[ F(t) =   \frac{3\gamma}{4\pi} (p \otimes p) : E(u)  -  \nabla_p\cdot \Big(\mathbb{P}_{p^\perp}   \left[(\gamma E(u) + W(u))p\right] \psi\Big). \]
Applying operator $U$ to both sides of this Duhamel's formula, we get
\[
u(t) + \int_0^t K(t,\tau) u(\tau)\dd\tau = f(t)
\]
where
\begin{align}
 & K(t,\tau) u_0  = - \frac{3 \gamma \iota}{4\pi} \text{St}^{-1} \na_x \cdot  \int_{\mathbb{S}^2} p \otimes p  \, S_u(t,\tau) ( p \otimes p :  E(u_0) )     \\
 &f(t)  =  \iota \text{St}^{-1} \na_x \cdot \int_{\mathbb{S}^2}  p \otimes p \,S_u(t,0)  \psi^{\init}\notag \\
 &\quad - \iota \text{St}^{-1} \na_x \cdot \int_{\mathbb{S}^2} p \otimes p \int_0^t S_u(t,\tau)  \nabla_p\cdot \Big(\mathbb{P}_{p^\perp}   \left[(\gamma E(u(\tau) ) + W(u(\tau) ))p\right] \psi(\tau)  \Big) \dd\tau.
 \end{align}
Our first result concerns the integrability properties of the kernel $K$ and only requires the bootstrap assumption  \eqref{BA'}.
\begin{proposition}\label{prop:estimates_K}
  There exists an absolute constant $\nu_0$ such that for
  $\nu \le \nu_0$ the condition \eqref{BA'} implies that kernel $K$
  satisfies the following estimates
  \begin{enumerate}[label=(\alph*), ref=(\alph*)]
    \item\label{item:short}  $\|K(t,\tau)\|_{L(H^s,H^s)} \lesssim 1$ for all $\tau \le t$ with $t-\tau \le 1$.
    \item\label{item:enhanced}  $\|K(t,\tau)\|_{L(H^s,H^s)} \lesssim \e^{-\eta_1 \nu^{\frac12} (t-\tau)}$, for all $\tau \le t$, with $\eta_1$ the rate given in the hypocoercive estimate \eqref{estim_Sv1}.
    \item\label{item:mix}  $\|K(t,\tau)\|_{L(H^s,H^r)}  \lesssim  \frac{1}{(t-\tau)^{\frac32}} + \nu^{\frac34}$,  for all $\tau < t$, for any $r <  s + \frac14$.
  \end{enumerate}
\end{proposition}
\begin{proof}
  We can restrict to the case $\tau=0$, as all the arguments that we
  will use are translation invariant. Given $u_0 \in H^s$, We denote
  by $\psi_0$ the solution of
  \begin{equation*}
    \pa_t \psi_0 + (u+p) \cdot \na_x \psi_0 -\nu \Delta_p \psi_0 =0,
    \quad
    \psi_0\vert_{t=0} = - \frac{3 \gamma \iota}{4\pi} p \otimes p: E(u_0).
  \end{equation*}
Note that the initial data is mean-free in $x$, a property that is propagated through time: $\psi_0 = \psi_{0, \neq 0}$. From the definition of $\psi_0$,
 \[ K(t,0) u_0 =   \text{St}^{-1} \na_x \cdot  \int_{\mathbb{S}^2} p \otimes p \, \psi_0.  \]
 By a standard energy estimate on $\psi_0$ in $\mH^{s-1}$, under
 \eqref{BA'} for \(u\) we find
 \[ \|\psi_0(t)\|_{\mH^{s-1}} \lesssim  \|\psi_0(0)\|_{\mH^{s-1}} \lesssim \|u_0\|_{H^s}, \quad \forall t \le 1\]
 Hence,
 \[ \|K(t,0) u_0\|_{H^s} \lesssim \|\psi_0(t)\|_{\mH^{s-1}} \lesssim \|u_0\|_{H^s}, \quad \forall t \le 1\]
 and estimate \ref{item:short} follows. Also, the hypocoercive
 estimate \eqref{estim_Sv1}, which is valid under \eqref{BA'}, implies
 for $\delta$ small enough that
\[  \|\psi_0(t)\|_{\mH^{s-1}} \lesssim  \e^{-\eta_1 \nu^{\frac12} t} \|\psi_0(0)\|_{\mH^{s-1}} \lesssim  \e^{-\eta_1 \nu^{\frac12} t}\|u_0\|_{H^s}, \quad \forall t \ge 0, \]
which implies estimate \ref{item:enhanced}.
It remains to prove estimate \ref{item:mix}. Expressing the operator $ \text{St}^{-1} \na_x \cdot $ in Fourier, one checks that
\[ \|K(t,0) u_0\|^2_{H^r}  = \sum_{k \neq 0} |k|^{2r-2}  \left| \int_{\bbS^2} (p \cdot \hat{k}) P_{\hat{k}^\perp} p \, \psi_{0,k} \right|^2.\]
To a given $k$, we can associate a cartesian frame $(e_x,e_y,e_z := \hat{k})$ and spherical coordinates $(\theta,\varphi)$  with $\theta$ the colatitude and $\varphi$ the longitude. In particular, $p = \sin \theta \cos \varphi e_x + \sin \theta \sin \varphi e_y + \cos \theta \hat{k}$. We compute
\begin{align*}
 p \cdot \hat{k} = \cos \theta, \quad  \na (p \cdot \hat{k}) & =  -\sin \theta e_\theta = -\sin \theta ( \cos \theta \cos \phi e_x  + \cos \theta \sin \phi e_y  -\sin \theta \hat{k}) ,\\
P_{\hat{k}^\perp} p & = p - (p \cdot \hat{k})  \hat{k} = \sin \theta \cos \phi e_x + \sin \theta \sin \phi e_y ,
\end{align*}
so that
\[ (p \cdot \hat{k}) P_{\hat{k}^\perp} p = - P_{\hat{k}^\perp}  \na (p \cdot \hat{k}) .\]
Hence
\[  \|K(t,0) u_0\|^2_{H^r}  = \sum_{k \neq 0} |k|^{2r-2}   \left| \int_{\bbS^2}  P_{\hat{k}^\perp}  \na (p \cdot \hat{k})\, \psi_{0,k} \right|^2.\]
From there, we apply inequality \eqref{estim_Sv2} (see \eqref{defHkVk} for the definition of  $\sfV_k$), with  $Z = Z_k = P_{\hat{k}^\perp}$ and $s = r - \frac54 +$  to get
\begin{align*}
\|K(t,0) u_0\|^2_{H^r} &   \lesssim
    \left(\frac{\nu^{\frac 12}}{\min(1,\nu^{\frac 12} t)}\right)^3
\Big( \|\psi_0(0)\|_{\mH^{r-\frac{5}{4}+}}^2 + \|\na_p  \psi_0(0) \|_{\mH^{r-\frac{5}{4} +}}^2 \\
&\qquad\qquad\qquad\qquad + \|\na_p^2\psi_0(0) \|_{\mH^{r-\frac{5}{4} +}}^2  \Big) \\
& \lesssim  \Big( \frac{1}{t^3} + \nu^{\frac32} \Big) \|u_0\|^2_{H^{r-\frac14+}}  .
\end{align*}
The result follows.
  \end{proof}
We can now state a stability estimate for $u$.
\begin{proposition}\label{prop_volterra}
  Assume \eqref{SC}.  Let $\eta_0 = \frac{\eta_1}{10}$, with $\eta_1$
  the absolute constant given in \eqref{estim_Sv1}. There exists
  $\nu_0 > 0$ depending on $\gamma$ and $\iota$ such that for
  $\nu \le \nu_0$ the condition \eqref{BA'} implies
  \begin{align*}
    \sup_{0 \le t \le T} \e^{\eta_0 \nu^{\frac12} t} \|u(t)\|_{H^s}
    \lesssim \sup_{0 \le t \le T} \e^{\eta_0 \nu^{\frac12} t} \|f(t)\|_{H^s}.
  \end{align*}
\end{proposition}

\begin{proof}
The proof is an adaptation of the reasoning in
\cite{CZDGV22}*{Section 5.2}. We first extend
$K(t,\tau)$ by zero for $t \ge T$ or $\tau > t$, and extend $u(t)$ and $f(t)$ by zero for $t \ge T$. Setting
\[
k(t,\tau) = \e^{\eta_0 \nu^{\frac12} (t-\tau)} K(t,\tau),\qquad \tilde{u}(t) = \e^{\eta_0 \nu^{\frac12} t} u(t),\qquad \tilde{f}(t) = \e^{\eta_0 \nu^{\frac12} t} f(t),
\]
the Volterra equation is equivalent to
\begin{equation}  \label{Volterra_non_conv}
 \tilde{u}(t) + \int_{0}^\infty k(t,\tau)\, \tilde{u}(\tau)\, \dd\tau =
 \tilde{f}(\tau)
 \qquad\text{for all $t  \in \R_+$}
\end{equation}
 and the point is to show that $\sup_{t \ge 0} \|\tilde{u}(t)\|_{H^s}
 \lesssim  \sup_{t \ge0} \|\tilde{f}(t)\|_{H^s}$. We introduce
 \begin{equation*}
   \mK_s := \left\{ k : \R_+ \times \R_+ \rightarrow L(H^s,H^s), \  k(t,\tau) = 0 \: \text{ for } \tau \le t, \quad \sup_{t \in \R_+} \int_{0}^\infty |k(t,\tau)|  \dd\tau < +\infty \right\}
 \end{equation*}
 equipped with
 $\|k\|_{\mK_s} := \sup_t \int_{\R_+} |k(t,\tau)| \,\dd\tau$. This
 space is an analogue of the space of \emph{Volterra kernels of
   bounded type} introduced in
 \cite{gripenberg-londen-staffans-1990-volterra}, replacing kernels
 with values in $\CC^n$ by kernels with values in $L(H^s,H^s)$. One
 can show exactly as in
 \cite{gripenberg-londen-staffans-1990-volterra} that $\mK_s$ is a
 Banach algebra for the product
\[ (k_1 \star k_2)(t,\tau) := \int_{0}^\infty k_1(t,\tau')\, k_2(\tau',\tau)\, \dd\tau' .\]
Moreover, as in any Banach algebra, there is a notion of resolvent: we say that $k \in \mK_s$ has resolvent $r \in \mK_s$ if
\[ r + k \star r = r +  r \star k = k. \]
One can show that
\begin{lemma} (Direct adaptation of \cite{gripenberg-londen-staffans-1990-volterra}*{Chapter~9, Lemma~3.4}) \label{lemma:resolvent_solution}
  If $k \in  \mK_s$ has a  resolvent $r \in \mK_s$  then, for any
  $\tilde{f} \in L^\infty(\R_+, H^s)$, equation  \eqref{Volterra_non_conv}  has a unique solution
  $\tilde{u} \in L^\infty(\R_+, H^s)$, given by
  \begin{equation*}
   \tilde{u}(t) = \tilde{v}(t) - \int_{\R_+} r(t,\tau)\, \tilde{v}(\tau)\, \dd \tau.
  \end{equation*}
  In particular, $\|\tilde{u}\|_{L^\infty(\R_+, H^s)} \le  (1 + \|r\|_{\mK_s})  \|\tilde{v}\|_{L^\infty(\R_+,H^s)}$.
\end{lemma}
Moreover, the set of kernels having a resolvent is open, which can be proved through a Von Neumann series argument. Namely,
\begin{proposition} (Direct adaptation of \cite{gripenberg-londen-staffans-1990-volterra}*{Chapter~9, Theorem~3.9}) \label{lemma:resolvent_perturb}

\smallskip
\noindent
  If $k = k_1 + k_2$ is the sum of two elements of $\mK_s$, if $k_1$ has a resolvent $r_1$  and if
  \begin{equation*}
   \|k_2\|_{\mK_s} < \frac{1}{1 +  \|r_1\|_{\mK_s} }
  \end{equation*}
  then $k$ has a resolvent $r$, given by $r = \sum_{n=0}^{+\infty} (-1)^n \big( (k_2 - r_1 \star k_2) \star \big)^n (k - r_1 \star k)$.
\end{proposition}
Following \cite{CZDGV22}, we decompose
\begin{align*}
  k(t,\tau) \:  =   K(t,\tau)  +  \big( \e^{\eta_0 \nu^{\frac12} (t-\tau)}
    -1 \big) K(t,\tau)
    =: k_1(t,\tau) + k_2(t,\tau).
\end{align*}
It is easily seen that $k \in \mK_s$. Moreover, we shall prove below:
\begin{lemma} \label{lemma:resolvent_K}
Assume   \eqref{SC}. There exists   $\nu_0 > 0$ depending on $\gamma$ and $\iota$ such that: for all $\nu \le \nu_0$, if \eqref{BA'} holds,  then  the kernel $(t,\tau) \mapsto K(t,\tau)$ has a resolvent $R = R(t,\tau)$, and
\[\|K(t,\tau)\|_{\mK_s} \lesssim 1 , \quad \|R(t,\tau)\|_{\mK_s} \lesssim 1. \]
\end{lemma}
Assuming for the moment that this lemma is satisfied, we now prove that  $\|k_2\|_{\mK_S} \rightarrow 0$ as $\nu \rightarrow 0$. Let $a \in (0,1)$ to be specified,   $b := \frac{1}{8\eta_0}$, and decompose:
\begin{align*}
  \|k_2\|_{\mK_S}
  & \le \sup_{t \in \R_+}  \int_0^{\infty}  \big( \e^{\eta_0 \nu^{\frac12} \tau} -1 \big)   \|K(t,t-\tau)\|_{L(H^s,H^s)}   \dd\tau \\
  & \le  \sup_{t \in \R_+}  \Big( \int_{0}^{a \nu^{-\frac12}} +  \int_{a \nu^{-\frac12}}^{b \lvert \ln \nu \rvert \nu^{-\frac12}}       + \int_{b \lvert\ln \nu \rvert \nu^{-\frac12}}^{\infty}  \Big)
 \big( \e^{\eta_0 \nu^{\frac12} \tau} -1 \big)  \\ 
 &\qquad \qquad \qquad \qquad \qquad \qquad \qquad \qquad \qquad \times\|K(t,t-\tau)\|_{L(H^s,H^s)}  \dd\tau \\
  &  =:   \sup_{t \in \R_+}  I_1(t) + I_2(t) + I_3(t).
\end{align*}
Let $\varkappa > 0$. Using the  bound $\|K(t,t-\tau)\|_{L(H^s,H^s)} \lesssim \langle \tau \rangle^{-\frac32}$ for $\tau \le a  \nu^{-\frac12}$, see Proposition \ref{prop:estimates_K},  we find
\[ I_1(t) \lesssim (\e^{\eta_0 a} - 1) \le \varkappa \quad \text{for $a$ small enough}. \]
This $a$ being fixed, using the bound
$\|K(t,t-\tau)\|_{L(H^s,H^s)} \lesssim \nu^{\frac34}$ for
$\tau \in [a \nu^{-\frac12}, b \nu^{-\frac12} \lvert \ln \nu \rvert ]$,
\textit{cf.}\ again Proposition~\ref{prop:estimates_K}, we get
\begin{equation*}
  I_2(t) \lesssim \nu^{-\eta_0 b}  \nu^{1/4}
  \lesssim \nu^{1/4 - 1/8} \le \varkappa,
\end{equation*}
for $\nu$ small enough. Eventually, using the second inequality in Proposition \ref{prop:estimates_K}, we get
\begin{equation*}
  I_3(t) \lesssim  \int_{b \lvert \ln \nu\rvert \nu^{-\frac12}}^{+\infty} \e^{(\eta_0 - \eta_1) \nu^{\frac12}\tau} \dd\tau \lesssim \nu^{-\frac12+ (\eta_1 - \eta_0) b} = \nu^{\frac18 (\frac{\eta_1}{\eta_0} - 5) }\le \varkappa,
\end{equation*}
for $\nu$ small enough. Hence,  $\|k_2\|_{\mK_S}$ goes to zero with $\nu$. By  Lemma \ref{lemma:resolvent_perturb} and Lemma \ref{lemma:resolvent_K} (still to be proved), we deduce  that the kernel $k$ has a resolvent $r$, with $\|r\|_{\mK_s} \lesssim 1$. Proposition \ref{prop_volterra} is then a direct consequence of Lemma \ref{lemma:resolvent_solution}.
The only missing step is the proof of Lemma \ref{lemma:resolvent_K}, which is done below.
\end{proof}

\begin{proof}[Proof of Lemma \ref{lemma:resolvent_K}]
  We remind  for $0 \le \tau \le t < T$ the formula
  \[ K(t,\tau) u_0 := \text{St}^{-1} \na_x \cdot  \int_{\mathbb{S}^2} p \otimes p \, \psi_\tau, \]
  where $\psi_0 = \psi_0(t)$ is the solution of
  \[\pa_t \psi_\tau + (u+p) \cdot \na_x \psi_\tau -\nu \Delta_p \psi_\tau =0, \quad  t \ge \tau, \quad \psi_\tau\vert_{t=\tau} = - \frac{3 \gamma \iota}{4\pi} p \otimes p : E(u_0).\]
  The first estimate of the lemma follows from the estimates of
  Proposition~\ref{prop:estimates_K} (see the treatment of the kernel
  $k_2$ for very close computations). To show that $K$ has a resolvent
  with norm $O(1)$, we shall again rely on Lemma
  \ref{lemma:resolvent_perturb}, seeing $K$ as a perturbation of the
  kernel $\bar{K}(t,\tau) = \bar{K}_0(t,\tau) 1_{t < T}$, where
\begin{equation} \label{def_bar_K0}
\text{for all} \: \tau > t, \quad \bar{K}_0(t,\tau) = 0, \quad \text{ while for } \: \tau \le t, \quad   \bar{K}_0(t,\tau) u_0 :=   \text{St}^{-1} \na_x \cdot  \int_{\mathbb{S}^2} p \otimes p \, \bar\psi_\tau(t)
\end{equation}
where this time $\bar{\psi}_\tau$ is the solution (mean-free in $x$) of
\begin{equation} \label{eq_bar_psi}
\pa_t \bar{\psi}_\tau  + p \cdot \na_x \bar{\psi}_\tau  -\nu \Delta_p \bar{\psi}_\tau  =0,  \quad t \ge \tau, \quad \bar{\psi}_\tau \vert_{t=\tau} = - \frac{3 \gamma \iota}{4\pi} p \otimes p : E(u_0).
\end{equation}
The kernel $\bar{K}_0$, corresponding to the case $u=0$, was analyzed
Fourier mode by Fourier mode in \cite{CZDGV22}. We remind some
elements of this analysis in Appendix \ref{appendixA}. This analysis
shows in particular that there exists $\nu_0 > 0$ depending on
$\gamma$ and $\iota$ such that for $\nu \le \nu_0$ the kernel
$\bar{K}_0$ has a resolvent $\bar{R}_0$ satisfying
\begin{equation}
 \|\bar{R}_0(t,\tau) \|_{\mK_s} \lesssim 1.
\end{equation}
It implies directly that $\bar{K}(t,\tau) = \bar{K}_0(t,\tau) 1_{t < T}$ has for resolvent $\bar{R}(t,\tau) = \bar{R}_0(t,\tau) 1_{ t < T}$ whose norm in $\mK_s$ satisfies the same bound.    By Lemma \ref{lemma:resolvent_perturb}, it is then enough to show that  under \eqref{BA'}
\begin{equation} \label{K-barK}
\lim_{\nu \rightarrow 0} \|(K - \bar{K})\|_{\mK^s} = 0
\end{equation}
Let $\varkappa > 0$. We decompose, for some large $\tilde{T}$ the
difference as
\begin{equation*}
  \begin{split}
    \|(K - \bar{K})\|_{\mK^s}
    &\le \int_0^{\tilde{T}} \|(K-\bar{K})(t,t-\tau)\|_{L(H^s,H^s)}\, \dd \tau \\
    &\quad+  \int_{\tilde{T}}^\infty \Big(\|K(t,t-\tau)\|_{L(H^s,H^s)} + \|\bar{K}(t,t-\tau)\|_{L(H^s,H^s)} \Big) \,\dd\tau.\\
  \end{split}
\end{equation*}
For the second part, we use the estimates of Proposition \ref{prop:estimates_K}, which are also valid for $\bar{K}$, as $u=0$ satisfies \eqref{BA'}: introducing $c = \frac{1}{\eta_1}$
\begin{align*}
  & \int_{\tilde{T}}^\infty \Big(\|K(t,t-\tau)\|_{L(H^s,H^s)} + \|\bar{K}(t,t-\tau)\|_{L(H^s,H^s)}  \Big)\, \dd\tau \\
  & = \Big( \int_{\tilde{T}}^{\nu^{-\frac12}} + \int_{\nu^{-\frac12}}^{c \nu^{-\frac12} \lvert \nu\rvert}  + \int_{c \nu^{-\frac12} \lvert \nu\rvert}^{+\infty} \Big)  \Big(\|K(t,t-\tau)\|_{L(H^s,H^s)} + \|\bar{K}(t,t-\tau)\|_{L(H^s,H^s)}  \Big)\,  \dd\tau \\
  & \lesssim  \int_{\tilde{T}}^{\nu^{-\frac12}} \frac{1}{\langle \tau \rangle^{\frac32}}  \dd\tau  + \int_{\nu^{-\frac12}}^{c \nu^{-\frac12} \lvert \nu\rvert}  \nu^{\frac34}  \dd\tau  + \int_{c \nu^{-\frac12} \lvert \nu\rvert}^{+\infty} \e^{-\eta_1 \nu^{\frac12} \tau}\, \dd\tau \\
  & \lesssim \int_{\tilde{T}}^{\infty} \frac{1}{\langle \tau \rangle^{\frac32}}  \dd\tau  + \nu^{1/4} \lvert \nu\rvert + \nu^{-\frac12+\eta_1 c} \le \varkappa,
\end{align*}
for $\tilde{T}$ large enough (depending on $\gamma$ and $\iota$) and $\nu$ small enough. This time $\tilde T$ being fixed, we turn to the first term. We claim that for $\nu$ small enough,  if \eqref{BA'} holds,  for all $\tau \le \tilde{T}$
\begin{equation} \label{estim_K-barK_loss}
\|(K-\bar{K})(t,t-\tau)\|_{L(H^s,H^{s-1})} \lesssim \nu^{\frac54} .
\end{equation}
Let us assume temporarily that \eqref{estim_K-barK_loss} holds. We also have, by the third inequality in Proposition \ref{prop:estimates_K}
\begin{equation} \label{estim_K-barK_gain}
\|K(t,t-\tau)\|_{L(H^s,H^{s+s'})} +  \|\bar{K}(t,t-\tau)\|_{L(H^s,H^{s+s'})} \lesssim  \frac{1}{\tau^{\frac32}}
\end{equation}
for  $s' = \frac18$ (any $s' \in (0, \frac14)$ would do). By interpolation of \eqref{estim_K-barK_loss} and \eqref{estim_K-barK_gain} , we get
\begin{align}
& \|(K-\bar{K})(t,t-\tau)\|_{L(H^s,H^{s})} \\
\nonumber & \le   \|(K-\bar{K})(t,t-\tau)\|_{L(H^s,H^{s-1})}^{1-\theta} \Big( \|K(t,t-\tau)\|_{L(H^s,H^{s+s'})} +  \|\bar{K}(t,t-\tau)\|_{L(H^s,H^{s+s'})}\Big)^{\theta} \\
& \lesssim  \nu^{\frac54 (1-\theta)} \tau^{-\frac{3}{2}\theta} ,
\end{align}
with $\theta$ such that $\theta s' - (1-\theta) = 0$, that is $\theta = \frac{1}{1+s'} = \frac{8}{9}$. We deduce from this estimate and the estimate
\[  \|(K-\bar{K})(t,t-\tau)\|_{L(H^s,H^{s})}  \le   \|K(t,t-\tau)\|_{L(H^s,H^{s})} + \|\bar{K}(t,t-\tau)\|_{L(H^s,H^{s})} \lesssim 1 \]
that
\begin{align*}
 & \int_0^{\tilde{T}} \|(K-\bar{K})(t,t-\tau)\|_{L(H^s,H^s)}\, \dd \tau \\
 &  =   \int_0^\varkappa \|(K-\bar{K})(t,t-\tau)\|_{L(H^s,H^s)}\, \dd \tau
   + \int_\varkappa^{\tilde{T}} \|(K-\bar{K})(t,t-\tau)\|_{L(H^s,H^s)}\,
   \dd \tau \\
 & \le C \varkappa +  C'  \nu^{\frac54 (1-\theta)}   \int_\varkappa^{+\infty} \tau^{-\frac{3}{2}\theta}  \le C  \varkappa +C''   \nu^{\frac54 (1-\theta)}   \varkappa^{-\frac13} \le (C+1) \varkappa,
\end{align*}
for $\nu$ small enough. As $\varkappa$ is arbitrary, this proves
\eqref{K-barK}. The final step is to establish
\eqref{estim_K-barK_loss}. From the definition of the kernel, we find
for \(t \le T\) (otherwise all quantities are zero):
\[ \|(K-\bar{K})(t,\tau) u_0\|_{H^{s-1}} \le \|(\psi_0 - \bar{\psi}_0)(t)\|_{\mH^{s-2}}. \]
The function $\psi := \psi_\tau - \bar{\psi}_\tau$ satisfies
\[ \pa_t \psi + (u+p) \cdot \na_x  \psi - \nu \Delta_p \psi = - u \cdot \na_x  \bar{\psi}_\tau, \quad t \ge \tau, \quad \psi \vert_{t=\tau} = 0. \]
A standard estimate yields (for $s-2 > \frac{5}{2}$),
\[ \pa_t \|\psi\|_{\mH^{s-2}}^2 \lesssim \|u\|_{\mH^{s-2}}  \|\psi\|_{\mH^{s-2}}^2 +  \|u\|_{\mH^{s-2}}   \|\na_x  \bar{\psi}_\tau\|_{\mH^{s-2}}  \|\psi\|_{\mH^{s-2}}. \]
For $\tau \le t$ with $t-\tau \le \tilde{T}$ (where we remind that
$\tilde{T}$ is fixed) this implies
\begin{align*}
\|\psi(t)\|_{\mH^{s-2}} &  \lesssim  \int^t_\tau \|\bar{\psi}_\tau(t')\|_{\mH^{s-1}} \|u(t')\|_{\mH^{s-2}} \dd t'  \lesssim  \sup_{\tau \le t' \le t} \|\bar{\psi}_\tau(t')\|_{\mH^{s-1}} \nu^{\frac54}  \lesssim \|u_0\|_{H^s} \nu^{\frac54},
\end{align*}
where the second inequality is coming from \eqref{BA'} and from the
standard Sobolev estimate
$\sup_{\tau \le t' \le t} \|\bar{\psi}_\tau(t')\|_{\mH^{s-1}} \lesssim
\|\bar{\psi}_\tau(\tau)\|_{\mH^{s-1}} = \|\frac{3 \gamma \iota}{4\pi}
p \otimes p : E(u_0)\|_{\mH^{s-1}}$.  Estimate
\eqref{estim_K-barK_loss} follows, and the proof of the lemma is
concluded.

\end{proof}

\subsection{Improvement of the bootstrap assumptions}\label{sub:boot}

The goal of this section is to prove the following:
\begin{proposition}\label{prop:bootstrap_improved}
  Assume \eqref{SC}. Let $\eta_0 = \frac{\eta_1}{10}$, with $\eta_1$
  the absolute constant given in \eqref{estim_Sv1}. There exists
  $\delta, \delta', \delta'', \nu_0 > 0$ depending on $\gamma$ and
  $\iota$ such that for $\nu \le \nu_0$ the assumptions
  \eqref{BA0}-\eqref{BA1}-\eqref{BA2} imply
 \begin{align}
 \label{BA0better}
   \sup_{0 \le t\le T} \|u(t)\|_{H^s} &  \lesssim \|\psi^{\init}\|_{\mH^{s-1}} \e^{-\eta_0 \nu^{\frac12} t} ,  \\
\label{BA1better}
\nu \int_0^T  \|\na_p \psi(t)\|^2_{\mH^{s-1}}\, \dd t & \lesssim \|\psi^{\init}\|^2_{\mH^{s-1}} \nu^{-1} + \|\psi^{\init}\|^3_{\mH^{s-1}}   \nu^{-\frac32} +  \|\psi^{\init}\|^4_{\mH^{s-1}}  \nu^{-\frac52} ,\\
\label{BA2better}
   \sup_{0 \le t\le T}\|\psi(t)\|_{\mH^{s-1}} & \lesssim  \|\psi^{\init}\|_{\mH^{s-1}} \nu^{-\frac12}.
\end{align}
In particular, there exists $\delta_0$, depending on $\gamma$ and $\iota$, such that for
\[ \|\psi^{\init}\|_{\mH^{s-1}}  \le \delta_0 \nu^{\frac32} \]
\eqref{BA0}-\eqref{BA1}-\eqref{BA2} can be improved, with $\delta/2,\delta'/2,\delta''/2$ replacing   $\delta,\delta',\delta''$, and also such that \eqref{stab_estim_T} holds.
\end{proposition}
As explained in Section~\ref{subsec_bootstrap}, the last part of the proposition
implies Theorem~\ref{thm:nonstab}.

\subsubsection{Bound on $u$}
In all what follows, we take $\delta \le \min\big(\frac{\eps}{\eta_0^{1/4}},1\big)$, so
that \eqref{BA0} implies \eqref{BA'}, and results of Section~\ref{sub:volterra}
can be applied.  The starting point is Proposition~\ref{prop_volterra}, which
says that
\[ \sup_{0 \le t < T} \e^{\eta_0 \nu^{\frac12} t} \|u(t)\|_{H^s} \lesssim \sup_{0 \le t < T} \e^{\eta_0 \nu^{\frac12} t} \|f(t)\|_{H^s}, \]
with the source term
\begin{align*}
f(t) & = \iota \text{St}^{-1} \na_x \cdot \int_{\mathbb{S}^2}  p \otimes p\, S_u(t,0)  \psi^{\init} \\
 & - \iota \text{St}^{-1} \na_x \cdot \int_{\mathbb{S}^2} p \otimes p \int_0^t S_u(t,\tau)  \nabla_p\cdot \Big(\mathbb{P}_{p^\perp}   \left[(\gamma E(u(\tau) ) + W(u(\tau) ))p\right] \psi(\tau)  \Big) \dd\tau \\
&  =:  f^{\init}(t) + f_{NL}(t).
\end{align*}
We remind that $S_u(t,\tau)$ is the solution operator of
$L = -(u+p) \cdot \na_x u + \nu \Delta_p$.  The first term is
estimated as
\begin{align*}
 \| f^{\init}(t) \|_{H^s} \lesssim  \|(S_u(t,0)\psi^{\init})_{\neq 0}\|_{\mH^{s-1}} \lesssim \e^{-\eta_1 \nu^{\frac12}t}  \|\psi^{\init}\|_{\mH^{s-1}},
\end{align*}
where the last inequality comes from \eqref{estim_Sv1}. Similarly,
\begin{align*}
\| f_{NL}(t) \|_{H^s} & \lesssim \int_0^t \|\Big( S(t,\tau) \nabla_p\cdot \Big(\mathbb{P}_{p^\perp}   \left[(\gamma E(u(\tau) ) + W(u(\tau) ))p\right] \psi(\tau)  \Big)  \Big)_{\neq 0} \|_{\mH^{s-1}}\dd\tau \\
&  \lesssim  \int_0^t \e^{-\eta_1 \nu^{\frac12}(t-\tau)} \Big(  \|\na_x u(\tau) \otimes \na_p \psi(\tau) \|_{\mH^{s-1}} + \|\na_x u(\tau) \otimes  \psi(\tau) \|_{\mH^{s-1}} \Big) \dd\tau\\
& \lesssim   \int_0^t \e^{-\eta_1 \nu^{\frac12}(t-\tau)}  \|u(\tau)\|_{H^s} \Big(  \|\na_p \psi(\tau) \|_{\mH^{s-1}} + \|\psi(\tau) \|_{\mH^{s-1}} \Big)\dd\tau \\
&   \lesssim   \big( \sup_{0 \le \tau < T} \e^{\eta_0 \nu^{\frac12}\tau}  \|u(\tau)\|_{H^s} \big) \\
&\quad\times\int_0^t \e^{-\eta_1 \nu^{\frac12}(t-\tau)} \e^{-\eta_0 \nu^{\frac12} \tau} \Big(  \|\na_p \psi(\tau) \|_{\mH^{s-1}} + \|\psi(\tau) \|_{\mH^{s-1}} \Big)\dd\tau.
\end{align*}
Hence,
\begin{multline*}
  \e^{\eta_0 \nu^{\frac12}t} \| f_{NL}(t) \|_{H^s}  \lesssim\\
  \big(  \sup_{0 \le \tau < T}  \e^{\eta_0 \nu^{\frac12}\tau}  \|u(\tau)\|_{H^s} \big)  \int_0^t \e^{- (\eta_1 -\eta_0) \nu^{\frac12} (t-\tau)}  \Big(  \|\na_p \psi(\tau) \|_{\mH^{s-1}} + \|\psi(\tau) \|_{\mH^{s-1}} \Big) \dd\tau,
\end{multline*}
where the second factor at the right-hand side is a convolution, resulting in
\begin{align*}
&\e^{\eta_0 \nu^{\frac12}t} \| f_{NL}(t) \|_{H^s}\\ &  \lesssim  \big(  \sup_{0 \le \tau < T}  \e^{\eta_0 \nu^{\frac12}\tau}  \|u(\tau)\|_{H^s} \big) \Big(  \|\e^{- (\eta_1 -\eta_0) \nu^{\frac12}  } \|_{L^2(\R_+)} \| \|\na_p \psi(\cdot) \|_{\mH^{s-1}} \|_{L^2(0,T)} \\
& \qquad \qquad  +  \|\e^{- (\eta_1 -\eta_0) \nu^{\frac12} \cdot } \|_{L^1(\R_+)} \| \|\psi(\cdot) \|_{\mH^{s-1}} \|_{L^\infty(0,T)} \Big) \\
& \lesssim  \big( \sup_{0 \le \tau < T}  \e^{\eta_0 \nu^{\frac12}\tau}  \|u(\tau)\|_{H^s} \big) \Big( \nu^{-1/4}   \| \|\na_p \psi(\cdot) \|_{\mH^{s-1}} \|_{L^2(0,T)} + \nu^{-\frac12}\| \|\psi(\cdot) \|_{\mH^{s-1}} \|_{L^\infty(0,T)} \Big) \\
& \lesssim  \big( \sup_{0 \le \tau < T}  \e^{\eta_0 \nu^{\frac12}\tau}  \|u(\tau)\|_{H^s} \big) ( \delta' + \delta'') ,
\end{align*}
where the last inequality comes from \eqref{BA1}-\eqref{BA2}. Hence,
\begin{align*}
  \sup_{0 \le \tau < T}  \e^{\eta_0 \nu^{\frac12}\tau}  \|u(t)\|_{H^s} \lesssim \sup_{t \ge 0} \e^{(\eta_0 - \eta_1) \nu^{\frac12} t}  \|\psi^{\init}\|_{\mH^{s-1}} +  (\delta' + \delta'')  \big( \sup_{\tau \ge 0} \e^{\eta_0 \nu^{\frac12}\tau}  \|u(\tau)\|_{H^s} \big) .
 \end{align*}
For $\delta', \delta''$ small enough (with a threshold depending on
$\gamma$ and $\iota$), we can absorb the second term at the right-hand
side,  which implies the first bound of the proposition.

\subsubsection{Bound on $\psi$}
We come back to the equation \eqref{SS3}. Performing standard $\mH^{s-1}$ Sobolev estimates on the equation, we get
\begin{align*}
\frac12 \frac{\dd}{\dd t} \|\psi\|_{\mH^{s-1}}^2 + \nu  \|\na_p \psi\|_{\mH^{s-1}}^2  & \lesssim \|u\|_{H^s}  \|\psi\|_{\mH^{s-1}} + \|u\|_{H^{s-1}} \|\psi\|_{\mH^{s-1}}^2 \\
& \quad+  \|u\|_{H^s}   \left( \|\psi\|_{\mH^{s-1}} +   \|\na_p \psi\|_{\mH^{s-1}}  \right) \|\psi\|_{\mH^{s-1}}.
\end{align*}
At the right-hand side, the first term corresponds to the contribution
of the linear term $-\frac{3\Gamma}{4\pi} (p \otimes p) : E(u)$. The
second term corresponds to the contribution of the transport term
$(u+p) \cdot \na_x \psi$, while the third one corresponds to the
contribution of
\[\nabla_p\cdot \big(\mathbb{P}_{p^\perp} \left[(\gamma E(u) +
  W(u))p\right] \psi\big).\] It implies  that
\begin{equation}  \label{estimatepsi}
\begin{aligned}
\frac12 \frac{\dd}{\dd t} \|\psi\|_{\mH^{s-1}}^2 + \frac{\nu}{2}  \|\na_p \psi\|_{\mH^{s-1}}^2  & \lesssim \|u\|_{H^s}  \|\psi\|_{\mH^{s-1}} + \|u\|_{H^{s}} \|\psi\|_{\mH^{s-1}}^2 \\
&\quad +  \nu^{-1} \|u\|_{H^s}^2 \|\psi\|_{\mH^{s-1}}^2
\end{aligned}
\end{equation}
so that
\begin{align*}
 \frac{\dd}{\dd t} \|\psi\|_{\mH^{s-1}} \lesssim \|u\|_{H^s}  +   (\|u\|_{H^s} +  \nu^{-1}  \|u\|_{H^s}^2) \|\psi\|_{\mH^{s-1}} .
\end{align*}
The Gronwall lemma together with the bound \eqref{BA0better} that we have just established yield
\begin{align*}
  \|\psi(t)\|_{\mH^{s-1}}
  & \lesssim  \|\psi^{\init}\|_{\mH^{s-1}} \nu^{-\frac12} \exp\left(
    \int_0^t  (\|u\|_{H^s} +  \nu^{-1}  \|u\|_{H^s}^2) \dd \tau \right)\\
 & \lesssim   \|\psi^{\init}\|_{\mH^{s-1}} \nu^{-\frac12} \exp\left(
   \int_0^t ( \delta \nu^{\frac32} \e^{-\eta_0 \nu^{\frac12}\tau} +
   \delta^2 \nu^{2}  \e^{-2\eta_0 \nu^{\frac12}\tau} ) \dd\tau
   \right)\\
 & \lesssim   \|\psi^{\init}\|_{\mH^{s-1}}  \nu^{-\frac12},
 \end{align*}
for $\nu$ small enough. Back to \eqref{estimatepsi}, integrating from $0$ to $t$ and using \eqref{BA0better},  we find
\begin{align*}
  \frac{\nu}{2} \int_0^t   \|\na_p \psi\|_{\mH^{s-1}}^2
  & \lesssim  \|\psi^{\init}\|^2_{\mH^{s-1}} \\
&\quad+ \int_0^t \left(  \|u\|_{H^s}  \|\psi\|_{\mH^{s-1}} + \|u\|_{H^{s}} \|\psi\|_{\mH^{s-1}}^2 + \nu^{-1} \|u\|_{H^s}^2 \|\psi\|_{\mH^{s-1}}^2 \right) \dd\tau\\
& \lesssim \|\psi^{\init}\|^2_{\mH^{s-1}}  +   \|\psi^{\init}\|^2_{\mH^{s-1}} \nu^{-1} + \|\psi^{\init}\|_{\mH^{s-1}}^3  \nu^{-\frac32} +  \|\psi^{\init}\|^4_{\mH^{s-1}}  \nu^{-\frac52} \\
& \lesssim   \|\psi^{\init}\|^2_{\mH^{s-1}} \nu^{-1} + \|\psi^{\init}\|_{\mH^{s-1}}^3  \nu^{-\frac32} +  \|\psi^{\init}\|^4_{\mH^{s-1}}  \nu^{-\frac52}.
\end{align*}

\appendix
\section{Resolvent estimates in the linear setting} \label{appendixA}
The goal of this appendix is to show that the kernel $\bar{K}_0$ defined in \eqref{def_bar_K0} has a resolvent $\bar{R}_0$ satisfying $\|\bar{R}_0\|_{\mK_s} \lesssim 1$. As the equation in  \eqref{eq_bar_psi} is autonomous in time, we have $\bar{K}_0(t,\tau) = K_0(t-\tau) 1_{\tau \le t}$, with
\[  K_0(t) u_0 :=   \text{St}^{-1} \na_x \cdot  \int_{\mathbb{S}^2} p \otimes p \, \bar\psi_0(t). \]
Accordingly, we look for a resolvent under the form $\bar{R}_0(t,\tau) = R_0(t-\tau) 1_{\tau \le t}$, where $R_0$ is the resolvent of $K_0$ for the usual convolution product
\[ R_0 + K_0 \star R_0 = R_0 + R_0 \star K_0 = K_0, \quad f \star g(t) = \int_0^t f(\tau)\, g(t-\tau)\,\dd\tau.   \]
In this case,  $\|\bar{R}_0\|_{\mK_s}  = \|R_0\|_{L^1_t(\R_+, L(H^s, H^s))}$. Hence, we want to show
\begin{equation} \label{goal_appendix}
  \|R_0\|_{L^1_t(\R_+, L(H^s, H^s))} \lesssim 1.
\end{equation}
The properties of $K_0$ have been studied  in \cite{CZDGV22}. More precisely, we performed a mode by mode Fourier analysis in $x$, with the study for an arbitrary $k \in 2\pi \ZZ^3_*$ of
\[\pa_t \hat{\psi}_k + \ii p \cdot k  \hat{\psi}_k -\nu \Delta_p \hat{\psi}_k =0, \quad  t \ge 0, \quad \hat{\psi}_k\vert_{t=0} = - \frac{3 \gamma \iota}{4\pi} p \otimes p :  \frac{k \otimes u_0 + u_0 \otimes k}{2}\]
and
\begin{equation*}
  \hat{K}_k(t) =  \frac{\ii}{|k|^2}    \int_{\mathbb{S}^2} p \cdot k
  P_{k^\perp} p  \,  \hat{\psi}_k (t) \in L(\CC^3, \CC^3).
\end{equation*}
The  analysis in \cite{CZDGV22} showed in particular that under the
spectral condition $\frac{\gamma |\iota|}{|k|} < \Gamma_c$, the $\hat{K}_k$ has a unique resolvent $\hat{R}_k$, satisfying $\|\hat{R}_k\|_{L^1(\R_+)} \le C_k$  where $C_k$ is independent of $\nu$, depends on $\gamma$, $\iota$ and possibly on $k$ (the dependence with respect to $k$ was not examined in details in \cite{CZDGV22}, see below). Note that for $u = \sum_{k \in 2\pi \ZZ^3_*} \hat{u}_k \e^{\ii k \cdot x}$, one has
 \[ K_0(t) u = \sum_{k \in 2\pi \ZZ^3_*}  \hat{K}_k(t) \hat{u}_k \e^{\ii k \cdot x} \]
and, at least formally,
\begin{equation*}
  R_0(t) u = \sum_{k \in 2\pi \ZZ^3_*}  \hat{R}_k(t) \hat{u}_k \e^{\ii k \cdot x}.
\end{equation*}
The remaining step is to show convergence of the series defining $R_0$, and to get a bound independent of $\nu$. This requires more accurate bounds on the $\hat{R}_k$ for $k \in 2\pi \ZZ^3_*$. We shall distinguish between low and high frequencies: for some cut-off frequency $N$ to be fixed later, we write
\[  R_0^\flat(t) u = \sum_{|k| \le N}  \hat{R}_k(t) \hat{u}_k \e^{\ii k \cdot x}, \quad   R_0^\sharp(t) u = \sum_{|k| > N}  \hat{R}_k(t) \hat{u}_k \e^{\ii k \cdot x}.\]
We will prove that there exists $N$, depending only on $\gamma$ and $\iota$ such that
\begin{equation}\label{control_Rsharp}
  \sup_{|k| > N} |\hat{R}_k(t)| \lesssim \frac{1}{(1+t)^{\frac32}}, \quad \forall t \ge 0.
\end{equation}
As $\|R_0^\sharp(t) u\|_{H^s} \le \sup_{|k| > N} |\hat{R}_k(t)| \|u\|_{H^s}$, one has
easily $\|R_0^\sharp\|_{L^1_t(\R_+, L(H^s, H^s))} \lesssim 1$. This $N$ being fixed, $R_0^\flat$ is
made of a finite number of terms, for which the bound
$\|\hat{R}_k\|_{L^1(\R_+)} \le C_k$ allows to conclude that
$\|R_0^\flat\|_{L^1_t(\R_+, L(H^s, H^s))} \lesssim 1$, which in turn yields
\eqref{goal_appendix}.

To prove \eqref{control_Rsharp}, let
\[ t' = |k| t, \quad k' = \frac{k}{|k|}, \nu' = \frac{\nu}{|k|},\quad \hat{\psi}'_{k'}(t') = \frac{1}{|k|} \hat{\psi}_k(t)   \]
so that
\[ \pa_{t'} \hat{\psi}'_{k'} + \ii p \cdot k'  \hat{\psi}'_{k'} -\nu' \Delta_p  \hat{\psi}'_{k'} =0, \quad  t \ge 0, \quad  \hat{\psi}'_{k'} = -  \frac{3 \gamma \iota}{4\pi} p \otimes p :  \frac{k' \otimes u_0 + u_0 \otimes k'}{2}\]
We further introduce
\[ \hat{K}_{k'}(t') u_0 =   \ii    \int_{\mathbb{S}^2} p \cdot k' P_{(k')^\perp} p  \,  \hat{\psi'}_{k'} (t)\]
It is an easy verification that
\[  \hat{K}_{k}(t) = \hat{K}_{k'}\big(|k|t\big). \]
The analysis of \cite{CZDGV22} was actually focused on the normalized
kernels $\hat{K}_{k'}$ (normalized because $|k'| = 1$). There, it was
shown that there exists $\nu_0, C_0, \eta_1, m$ depending on $\gamma$,
$\iota$ such that for $\nu' \le \nu_0$ following inequalities hold
 \begin{align*}
  |\hat{K}_{k'}(t')| & \le C_0 \frac{\ln(2+t')}{(1+t')^2} \quad \forall t' \le (\nu')^{-\frac12}, \\
  |\hat{K}_{k'}(t')| & \le C_0 \nu' \lvert \ln(\nu') \rvert^m  \quad \forall t' \in [\nu'^{-\frac12}, c \lvert \ln(\nu')\rvert (\nu')^{-\frac12}] , \\
   |\hat{K}_{k'}(t')| & \le C_0 e^{-\eta_1 (\nu')^{\frac12} t'}  \quad \forall t' \ge c \lvert \ln(\nu') \rvert (\nu')^{-\frac12}.
  \end{align*}
 Taking $c$ large enough compared to $\eta_1$, these estimates imply
 \[   |\hat{K}_{k'}(t')|  \lesssim \frac{1}{(1+t')^{\frac32}} \]
 (any power less than $2$ would do). Hence,
 \[   |\hat{K}_{k}(t)|  \lesssim \frac{1}{(1+|k|t)^{\frac32}}. \]
 This implies
 \begin{equation*}
   \|\hat{K}_{k}\|_{L^1(\R_+)} \lesssim \frac{1}{|k|} \lesssim \frac{1}{N}.
 \end{equation*}
 For $N$ large enough, we have in particular that $ \|\hat{K}_{k}\|_{L^1(\R_+)} < 1$,  and we know in this case that the resolvent is given explicitly by the Neumann series
 \[ \hat{R}_k = \sum_{j\ge 0} (-1)^{j} (\hat{K}_{k} \, \star \, )^j \hat{K}_{k}. \]
 It is moreover straightforward to show that if
 \[ |f(t)| \lesssim \frac{C_f}{(1+|k| t)^{\frac32}}, \quad   |\hat{K}_{k}(t)|  \le \frac{C_K}{(1+|k|t)^{\frac32}} \]
 then for some absolute constant $C_0$
 \[ |\hat{K}_{k} \star f(t)| \le \frac{C_0 C_f C_K}{|k|(1+|k| t)^{\frac32}}. \]
 By induction we get
 \[  |(\hat{K}_{k} \, \star \,)^j \hat{K}_{k}(t)| \le \Big( \frac{C_0 C_K}{|k|} \Big)^j  \frac{C_K}{(1+|k| t)^{\frac32}}. \]
 Eventually, for $N$ large enough so that $\frac{C_0 C_K}{N} < 1$, we find
 \begin{align*}
 |R_k(t)| \lesssim \sum_{j \ge 0}  \Big( \frac{C_0 C_K}{|k|} \Big)^j   \frac{1}{(1+|k|t)^{\frac32}} \lesssim    \frac{1}{(1+|k|t)^{\frac32}}.
 \end{align*}
 The result \eqref{control_Rsharp} follows.



\begin{funding}
  MCZ acknowledges support of the Royal Society through grant URF\textbackslash
  R1\textbackslash 191492 and the ERC/EPSRC through grant Horizon Europe Guarantee
  EP/X020886/1. The work of DGV was supported by project Singflows grant
  ANR-18-CE-40-0027 and project Bourgeons grant ANR-23-CE40-0014-01 of the
  French National Research Agency (ANR).
\end{funding}


\bibliography{CZDGV-NonlinearSuspBiblio.bib}

\end{document}